\documentclass[11pt]{amsart}

\usepackage{amssymb, xcolor,amsmath,url}
\newtheorem{theorem}{Theorem}[section]
\newtheorem{lemma}[theorem]{Lemma}
\newtheorem{proposition}[theorem]{Proposition}
\newtheorem{proposition-and-definition}[theorem]{Proposition and Definition}
\newtheorem{corollary}[theorem]{Corollary}
\newtheorem{review}[theorem]{Review}

\theoremstyle{definition}
\newtheorem{definition}[theorem]{Definition}
\newtheorem{example}[theorem]{Example}
\newtheorem{remark}[theorem]{Remark}
\newtheorem{notation}[theorem]{Notation}
\newtheorem{remark-and-notation}[theorem]{Remark and Notation}
\newtheorem{notation-and-remark}[theorem]{Notation and Remark}

\numberwithin{equation}{section}

\newcommand{\cA}{ \mathcal{A} }
\newcommand{\cB}{ \mathcal{B} }
\newcommand{\bC}{ \mathbb{C} }
\newcommand{\bE}{ E }
\newcommand{\cF}{ \mathcal{F} }
\newcommand{\cK}{ \mathcal{K} }
\newcommand{\cM}{ \mathcal{M} }
\newcommand{\bN}{ \mathbb{N} }
\newcommand{\bR}{ \mathbb{R} }
\newcommand{\cV}{ \mathcal{V} }
\newcommand{\cW}{ \mathcal{W} }

\newcommand{\oneA}{ 1_{ { }_{\cA} } }
\newcommand{\zeroA}{ 0_{ { }_{\cA} } }
\newcommand{\eqdef}{ \stackrel{\mathrm{def}}{\Leftrightarrow} }
\newcommand{\muovxy}{ \mu^{( \mathrm{ov} )}_{x,y} }
\newcommand{\muovadd}{ \mu^{( \mathrm{ov} )}_{a,a+b} }
\newcommand{\muovmulp}{ \mu^{( \mathrm{ov} )}_{a,a^{1/2}ba^{1/2}} }
\newcommand{\ovxy}{ {\mathrm{o}}_{x,y} }
\newcommand{\ovadd}{ {\mathrm{o}}_{a,a+b} }
\newcommand{\ovmulp}{ {\mathrm{o}}_{a,a^{1/2}ba^{1/2}} }

\newcommand{\Bor}{ \mathrm{Bor} }
\newcommand{\Borpol}{ \mathrm{Bor}_{\mathrm{pol}} }
\newcommand{\Entry}{ \mathrm{Entry} }
\newcommand{\Spec}{ \mathrm{Spec} }
\newcommand{\ks}{k_s^{(1)}}
\newcommand{\kt}{k_t^{(2)}}

\title[Free denoising via overlap measures and c-freeness]{Free denoising via
	overlap measures and \\
	$c$-freeness techniques}

\author[M. Fevrier]{Maxime Fevrier}
\address{Maxime Fevrier: Universit\'e Paris-Saclay, CNRS, Laboratoire de math\'ematiques d’Orsay, 91405, Orsay, France. } 
\email{ maxime.fevrier@universite-paris-saclay.fr  }

\author[A. Nica]{Alexandru Nica}
\address{Alexandru Nica: Department of Pure Mathematics, 
	University of Waterloo, Ontario, Canada.}
\email{anica@uwaterloo.ca}
\thanks{AN: research supported by a Discovery Grant from NSERC, Canada.}

\author[K. Szpojankowski]{Kamil Szpojankowski}
\address{ Kamil Szpojankowski:
	Faculty of Mathematics and Information Science,
	Warsaw University of Technology, Poland.  } 
\email{kamil.szpojankowski@pw.edu.pl  }
\thanks{KSz: This research was funded in part by National Science Centre, Poland WEAVE-UNISONO grant BOOMER 2022/04/Y/ST1/00008.
	\\ For the purpose of Open Access, the authors have applied a CC-BY public copyright licence to any Author Accepted Manuscript (AAM) version arising from this submission.}

\date{\today}

\begin{document}
	
	\begin{abstract}
		We study the problem of \emph{free denoising}. For free selfadjoint random variables $a,b$,
		where we interpret $a$ as a signal and $b$ as noise, we find $\bE(a \, | \, a+b)$. To that end, we
		study a probability measure $\muovadd$ on $\bR^2$ which we call the \emph{overlap measure}. 
		We show that $\muovadd$ is absolutely continuous with respect to the product measure 
		$\mu_a\times \mu_{a+b}$. The Radon-Nikodym derivative gives direct access to $\bE(a \, | \, a+b)$. 
		We show that analogous results hold in the case of multiplicative noise when $a,b$ are positive
		and the aim is to find $\bE(a \, | \, a^{1/2}ba^{1/2})$. 
		In a parallel development we show that, for a general selfadjoint expression $P(a,b)$ 
		made with $a$ and $b$, finding $\bE(a \, | \, P(a,b))$ is equivalent to finding the distribution 
		of $P(a,b)$ in a certain two-state probability space
		$(\cA,\varphi,\chi)$, where $a,b$ are c-free with
		respect to $(\varphi,\chi)$ in the sense of Bo\.zejko-Leinert-Speicher. 
		We discuss how free denoising (which is set in the framework of an 
			abstract $W^{*}$-probability space) relates to the notion of ``matrix denoising'' 
			previously discussed in the random matrix literature.
	\end{abstract}
	
	\maketitle
	
	\section{Introduction}
	\subsection{Problem description.}
	In this paper we consider some instances (primarily the additive instance 
	and the multiplicative instance) of a problem which could go under the 
	name of {\em free denoising}.  The setting for free denoising involves 
	two freely independent selfadjoint random variables $a,b$ 
	in a tracial $W^{*}$-probability space $\left(\cA,\varphi\right)$,
	where $a$ is viewed as signal (it is ``the random variable of interest''), 
	while $b$ is viewed as noise.  In the additive version of the free 
	denoising problem we assume that the
	sum $a+b$ is given, and our goal is to find the best approximation of 
	$a$ by an element $h(a+b) \in \cA$, where $h$ is a real-valued Borel function 
	on the spectrum of $a+b$.  The multiplicative version of the problem goes 
	in the same vein, but where $a,b$ are now assumed to be positive and the task 
	is to approximate $a$ by an element $h( \,a^{1/2} b a^{1/2} \, )$. 
	In full generality, the random variable $c \in \cA$ which is assumed to be 
	given is some selfadjoint expression $c = P(a,b)$ formed with $a$ and $b$ 
	(instead of $a+b$ or $a^{1/2} b a^{1/2}$ one could have, for instance, 
	a quadratic expression such as $P(a,b) = i(ab-ba)$ or $P(a,b) = ab+ba$). 
	
	When viewed as an approximation problem with respect to the $|| \cdot ||_2$-norm
	associated to the trace $\varphi$, the above mentioned problem is turned into the 
	one of understanding what is the conditional expectation 
	$E \bigl( a \mid P(a,b) \bigr)$, where 
	$E \bigl( \cdot \mid P(a,b) \bigr) : \cA \to \cB$ is the unique 
	$\varphi$-preserving conditional expectation onto the von Neumann algebra 
	$\cB \subseteq \cA$ generated by $P(a,b)$.  In Sections 
\ref{AdditiveFormulas} and \ref{MultiplicativeFormulas}  of the 
	present paper we provide explicit formulas for such conditional expectations, 
	holding in the two main cases under consideration, $E( a \, | \, a+b)$ and respectively
	$E \left(a \, | \, a^{1/2} b a^{1/2} \right)$.  Then, for a general $c = P(a,b)$, 
	Section 
    \ref{CFreeFormulas} of the paper puts into evidence an alternative approach to free denoising
	which is found by relating it to the notion of ``conditionally free independence'' 
	of Bo\.zejko-Leinert-Speicher \cite{BozejkoLeinertSpeicher}. In the appendix we explain
	that the methodology described in this paper 
	also works in the case 
	of the multiplication of two freely independent unitaries.
	
%	\vspace{10pt}
	
	\subsection{Motivation: analogies with classical probability. } 
	In order to put things into perspective, we note here that free 
	denoising is the free probabilistic counterpart 
	of a well-known formalism of classical probability, concerning
	conditional expectations (taken in usual ``commutative'' sense) 
	$\mathbb{E} (X \, | \, Y )$, where $(X,Y)$ is a random vector in $\bR^2$.  
	Under fairly general conditions (see e.g.~the survey paper 
	\cite{ChangPollard}), the distribution $\sigma$ of $(X,Y)$ can be 
	disintegrated with respect to the 1-dimensional distribution $\nu$ of 
	$Y$, leading to a family of 1-dimensional distributions
	$( \mu_t )_{t \in \bR}$ such that for a bounded Borel function 
	$f : \bR^2 \to \bC$ we have
	\begin{equation}   \label{eqn:12a}
		\int_{\bR^2} f(s,t) \, d \sigma (s,t) = \int_{\bR} \, \Bigl[ 
		\, \int_{\bR} f(s,t) \, d \mu_t (s) \, \bigr] \, d \nu (t).
	\end{equation}
	In this setting, one gets that $\mathbb{E} (X \, | \, Y ) = h(Y)$, where we
	put
	\begin{equation}   \label{eqn:12b}
		h(t) = \int_{\bR} s \, d \mu_t (s), \ \ t \in \bR .
	\end{equation}
	
	The preceding formula shows that, in general, the initial data needed in order 
	to find $\mathbb{E} (X \, | \, Y)$ is the full joint distribution $\sigma$ of $X$ and $Y$. 
	But there exist cases when the said conditional expectation can be 
	expressed solely in terms of $Y$ and its distribution.  To illustrate this, 
	let us focus, in analogy to the previous subsection, on the case where 
	$Y = X + Z$ with $X,Z$ independent (in the commutative sense) and 
	$Z \sim \mathcal{N}(0,\sigma^2)$. 
	The distribution of $Y$ is then sure to
	have a density $f_Y$ with respect to Lebesgue measure, and the well-known 
	Tweedie's formula in statistics (see \cite{Robbins}) says that in this case one has
	\begin{align} \label{eqn:tweedie}
		\mathbb{E}(X \, | \, Y) = Y+\sigma^2 g(Y),
		\ \mbox{ where $g(t)=\tfrac{d}{dt} \log(f_Y (t))$.}  
	\end{align}
	
	\noindent
	For the free analogue of the above mentioned development, we will use a 
	sibling of the notion of joint distribution which we call ``overlap measure''
	(Section 1.4 below).  The methodology developed in connection to this yields, 
	in particular, both an additive and a multiplicative free analogue for Tweedie's 
	formula (\ref{eqn:tweedie}), where the formula arising in the additive analogue is 
	of the same nature as (\ref{eqn:tweedie}).  More
	precisely, for this additive analogue we pick $a,b \in \cA$ to be freely independent
	and selfadjoint (like at the beginning of Section 1.1 above), and we assume that 
        the distribution of $b$ is the Wigner semicircle law with variance $\sigma^2$.  
        We then find that   
	$\bE( a \, | \, a+b)=h(a+b)$ for the function $h : \bR \to \bR$ defined by 
	\begin{equation}    \label{eqn:freetweedieboxplus}
		h(t) = t - 2 \pi \, \sigma^2 \, H_{\mu\boxplus\nu}(t), 
	\end{equation}
	where $H_{\mu\boxplus\nu}$ is the Hilbert transform of the measure 
	$\mu \boxplus \nu$.  It is remarkable that, analogous to the classical case, 
	the right-hand side of (\ref{eqn:freetweedieboxplus}) depends only on the 
	distribution $\mu \boxplus \nu$ of $a+b$. The details of the derivation of 
	(\ref{eqn:freetweedieboxplus}) are presented in Example \ref{example:47} below.
	We mention that this formula had been spotted in the physics literature, 
	via a method called ``the replica trick'' (see \cite{BunAllezBouchaudPotters}).
	
	The multiplicative free analogue of Tweedie's formula is presented in 
	Example \ref{example:15} below.  As it turns out, this free analogue
	had also been spotted in previous research literature on random matrices, 
	in \cite{LedoitPeche}.
	
%	\vspace{10pt}
	
	\subsection{Relation to subordination functions.}
	Let $a,b$ be freely independent selfadjoint elements in a $W^{*}$-probability 
	space $( \cA , \varphi )$.  In \cite{Biane}, Biane showed how conditional 
	expectations of the form $E( f(a+b) \, | \, a )$ and $E( g(a^{1/2} b a^{1/2}) \, | \, a )$ 
	can be treated by using the notion of additive and respectively of multiplicative 
	{\em subordination function} for the Cauchy transforms of the distributions of interest
	(a precise review of subordination terminology and some statements of results appear
	in Section \ref{SubordReview} below).  The development of the 
        present paper makes significant use of
	subordination functions and of some related Feller-Markov kernels that were identified 
	in \cite{Biane}; but we are in a certain sense ``going in reverse'' from how Biane did, 
	since we are looking at $E( f(a) \, | \, a+b )$ rather than $E( f(a+b) \, | \, a )$, and 
	likewise in the multiplicative case.  The outcome of the calculations is quite different.  
	For example, as noted by Biane in \cite[Remark 3.3]{Biane}, it is always the case that 
	$E( \, (a+b)^n \, | \, a )$ is a polynomial of degree $n$ in $a$; in contrast to that,
	a conditional expectation $E( a^n \, | \, a+b)$ is of the form $h(a+b)$, where $h$ 
	(a function on the spectrum of $a+b$) may not be a polynomial. In fact,
	Bo\.zejko and Bryc \cite{BozejkoBryc} showed that merely assuming $E(a \, | \, a+b)$ and 
	$E(a^2 \, | \, a+b)$ to be linear and respectively quadratic polynomials of $a+b$ forces
	$a,b$ to have some special distributions, coming from the free Meixner family.
	
	On the general subject of work which uses subordination and has a topic related to 
	the one of the present paper, we mention here the work \cite{ArizmendiTarragoVargas}, 
	where the authors study a problem called ``free deconvolution'', which is about retrieving
	the distribution of $a$ when the distribution of $a+b$ and the distribution of the 
	noise $b$ are given.  Free denoising is different from free deconvolution: it aims to
	find the element $h(a+b)$ which is at minimal $|| \cdot ||_2$-distance from $a$, and where 
	(as easily seen by computing variance in simple examples with $a, b$ identically distributed)
	$h(a+b)$ will generally not have the same distribution as $a$.  
	
%	\vspace{10pt}

	\subsection{Overlap measure and overlap function.}
	In order to discuss the non-commutative analogue of the disintegration procedure 
        from Equations (\ref{eqn:12a}), (\ref{eqn:12b}) of Section 1.2, one replaces the 
	random vector $(X,Y)$ considered there with a pair of selfadjoint elements $x,y$ 
        in a $W^{*}$-probability space $( \cA , \varphi )$.  The notion of joint distribution 
        of $x$ and $y$ is usually defined, in this non-commutative setting, as a linear 
        functional $\mu_{x,y} : \mathcal{P} \to \bC$, where $\mathcal{P}$ is the algebra 
	of non-commutative polynomials in two indeterminates and we have 
	$\mu_{x,y}(P)=\varphi(P(x,y))$ for $P \in \mathcal{P}$.  
For the purpose of finding the conditional expectation $E(x \, | \, y )$, the 
functional $\mu_{x,y}$ contains too much information; it is convenient 
to use in its place a bona fide measure $\muovxy$ on the Borel sigma-algebra 
of $\bR^2$, which we call the {\em overlap measure} of $x$ and $y$, and is
defined via the requirement that
\[
\int_{\bR^2}f(s)g(t) d \muovxy (s,t) =  \varphi(f(x)g(y)), 
\ \ \forall \, f,g \in \Bor_b ( \bR ),
\]
where $\Bor_b ( \bR )$ is the space complex-valued bounded Borel functions on $\bR$. 

In the $W^{*}$-framework, the measure $\muovxy$ has been in use for 
quite some time, and is presented for instance in Section I.1 of \cite{Connes}.  
It is worth observing that its construction can be performed more generally, in any 
$*$-probability space $( \cA , \varphi )$ where we have an ability to do Borel 
functional calculus for selfadjoint elements; we refer to such a structure by calling
it a ``Borel-ncps'', see Sections \ref{section:X} and \ref{section:Y} below.
Using the Borel-ncps framework is relevant, because it allows one to rigorously 
consider overlap measures in a $*$-probability space of random $N \times N$ matrices 
(cf.~Example \ref{example:X9}), and to study the weak convergence of such overlap 
measures in the large $N$ limit.  Doing so helps explaining why some explicit formulas
of free denoisers arising in the $W^{*}$-framework coincide with formulas that had
been previously found, via random matrix techniques, in papers about matrix denoising. 
We comment on this in the final Section \ref{Sec:MatrixDenoising} of the paper. 

The two marginals of $\muovxy$ turn out to be the distributions $\mu_x$ and $\mu_y$
of the two selfadjoint elements $x, y$ we started with.  If $\muovxy$ happens to be 
absolutely continuous with respect to the direct product $\mu_x \times \mu_y$, then
the Radon-Nikodym derivative $d \muovxy / (d \mu_x \times d \mu_y)$ will be referred 
to as the {\em overlap function} of $x$ and $y$.  We note that $\muovxy$ does not 
always have to be absolutely continuous with respect to $\mu_x \times \mu_y$ (e.g.~when
$y=x$, where $x$ has absolutely continuous distribution with respect to Lebesgue measure);
but this will nevertheless be the case in several relevant examples.

Our motivation for the ``overlap'' terminology comes from the case when
$\cA = \cM_N ( \bC )$, considered with the normalized trace 
$\varphi = \frac{1}{N} \mathrm{Tr}$, when the pair of selfadjoint elements 
$x,y \in \cA$ to be considered is thus a pair of complex Hermitian matrices. 
In this case, the overlap function $\ovxy$ is always sure 
to exist.  It is defined on $\Spec (x) \times \Spec (y)$ where 
$\Spec (x), \Spec (y) \subseteq \bR$ are the sets of eigenvalues of the 
matrices $x$ and $y$, and its values are computed in terms of quantities 
$| \langle u,v \rangle |^2$, where $u,v \in \bC^N$ are eigenvectors 
of $x$ and respectively of $y$ (cf.~Example \ref{example:Y5} below). 
For instance: if $\lambda, \rho$ are eigenvalues of multiplicity $1$ of 
$x$ and $y$, respectively, and if $u,v \in \bC^N$ are unit vectors such that 
$x(u) = \lambda u$ and $y(v) = \rho v$, then we get 
$\ovxy ( \lambda , \rho ) = N \, | \langle u, v \rangle |^2$
-- this is precisely the kind of quantity called an overlap in the 
physics literature. 

	%  Our interest in $\muovxy$ stems from the fact that 
	% it can be used to compute the conditional expectation $E( x \, | \, y )$ via 
	  % an analogue of the procedure from Equations (\ref{eqn:12a}), 
	% (\ref{eqn:12b}): disintegrate $\muovxy$ into a family of probability 
	% measures $( \mu_t )_{t \in \bR}$, with respect to the distribution 
	% of $y$, then write $E(x \, | \, y )$ as a function of $y$ in exactly 
	% the same way as in (\ref{eqn:12b}), $E( x \, | \, y ) = h(y)$ with 
	% $h(t) := \int_{\bR} s \, d \mu_t (s).$
	% For the precise statement of how all this goes, see the review of $\muovxy$ 
	% in Section 3 below, and particularly Proposition \ref{prop:Z3}.
	%
	% A point to be noted here is that,
	% in the aforementioned disintegration of $\muovxy$, one only needs to look at the
	% measures $\mu_t$ for $t \in \Spec (y)$, and that every such $\mu_t$ is supported 
	% on $\Spec (x)$, where $\Spec (x), \Spec (y) \subseteq \bR$ are the spectra of
	% the elements $x$ and $y$, respectively.
	%

The fact that $\muovxy$ is compactly supported on $\Spec (x) \times \Spec (y)$ 
holds in general in the $W^{*}$-framework, where 
$\Spec (x), \Spec (y) \subseteq \bR$ are now the spectra of the selfadjoint 
elements $x, y \in \cA$.  As mentioned at the beginning of the subsection, our
reason for considering $\muovxy$ in this framework is that it can be used to 
compute the conditional expectation $E( x \, | \, y )$ via an analogue of the 
procedure from Equations (\ref{eqn:12a}), (\ref{eqn:12b}): disintegrate $\muovxy$ 
into a family of probability measures $( \mu_t )_{t \in \Spec (y)}$ with respect 
to the distribution of $y$ (where every $\mu_t$ is a probability measure on $\Spec (x)$),
then write $E(x \, | \, y )$ as a function of $y$ in the same way as in (\ref{eqn:12b}), 
$E( x \, | \, y ) = h(y)$ with $h(t) := \int_{\Spec (x)} s \, d \mu_t (s).$
For the details of how this goes, see Section \ref{section:Z} below, and particularly
Proposition \ref{prop:Z3}.

	% \vspace{10pt}

	\subsection{Description of results -- additive case.}
        In the case when $x=a$ and $y=a+b$, with $a,b$ freely independent selfadjoint 
        random variables, we show that the absolute continuity 
        $\muovxy \ll \mu_x\times\mu_y$ is always sure to hold, and we give an explicit
        description of the overlap function $d \muovxy / d ( \mu_x\times\mu_y )$.
        This is obtained by using the relation to subordination functions mentioned 
        above in Section 1.3, and by relying on a known body of results 
        (cf.~\cite{Belinschi, Belinschi2, BelinschiBercoviciAdditiveSemigroup, BercoviciVoiculescu}) 
        about atoms and regularity for a free additive convolution $\mu \boxplus \nu$.  
        The precise description of the result is given as Theorem \ref{thm:1} below.  The 
        statement of Theorem \ref{thm:1}.$2^o$ uses some basic structural details of
        $\mu \boxplus \nu$ -- for their review, the reader can have an advanced look at the 
        Section \ref{AdditiveFormulas} of the paper (and specifically at the 
        ``Review \ref{review:43}'' item over there).

        \vspace{6pt}
        
        \begin{theorem}  \label{thm:1}
	Let $a,b$ be freely independent selfadjoint elements in a 
        $W^{*}$-pro\-ba\-bi\-li\-ty space $( \cA, \varphi )$. Let $\mu$ and respectively $\nu$
        be the distributions of $a$ and $b$ (with respect to $\varphi$), and assume that neither
        of $\mu, \nu$ is a point mass.  Note that, due to the free independence of
        $a$ and $b$, the distribution of $a+b$ is equal to $\mu \boxplus \nu$.  We let 
        $\omega : \bC^{+} \to \bC^{+}$ be the subordination function of $\mu \boxplus \nu$
        with respect to $\mu$ and for every $t \in \bR$ we denote by $\omega (t)$ (known to
        exist in $\bR \cup \bC^{+} \cup \{ \infty \}$) the limiting value of $\omega$ at the 
        point $t$.
        
        \vspace{6pt}

        \noindent
        $1^o$ Let $\muovadd$ be the overlap measure (as discussed in Section 1.4) for the 
        elements $a$ and $a+b$.  One has $\muovadd \ll \mu \times ( \mu \boxplus \nu )$, 
        and therefore it makes sense to consider the overlap function 
        $\ovadd = d \muovadd / d \bigl( \, \mu \times ( \mu \boxplus \nu ) \, \bigr) $.

        \vspace{6pt}

        \noindent
        $2^o$ Let $\Gamma = \{ \gamma_1, \ldots , \gamma_n \}$ be the 
        (necessarily finite) set of atoms of $\mu \boxplus \nu$, and for every $1 \leq i \leq n$
        consider the decomposition (known to exist and be unique) 
        $\gamma_i = \alpha_i + \beta_i$, with $\alpha_i$ an atom of $\mu$ and 
        $\beta_i$ an atom of $\nu$, such that 
        $\mu ( \, \{ \alpha_i \} \, ) + \nu ( \, \{ \beta_i  \} \, ) > 1$.  
        On the other hand let $( \mu \boxplus \nu )^{\mathrm{ac}}$ denote the 
        absolutely continuous part of $\mu \boxplus \nu$ with respect to Lebesgue measure. 
        It is known that one can write the support of $\mu \boxplus \nu$ (that is, the 
        set $\Spec (a+b)$) as a disjoint union $\Gamma \cup U \cup Z$, 
        with $( \mu \boxplus \nu) (Z) = 0$ and where $U$ is a relatively open subset 
        of $\Spec (a+b)$, carrying a continuous function 
        $f_{\mu \boxplus \nu} : U \to (0, \infty )$ which is a version of the 
        Radon-Nikodym derivative $d ( \mu \boxplus \nu )^{\mathrm{ac}} (t) / dt$. 
        Moreover, $U$ can be picked such that $\omega (t) \in \bC^{+}$, $\forall \, t \in U$.

        With the notation introduced above, one can describe the overlap function 
        $\ovadd$ from part $1^o$ by indicating (in a $( \mu \boxplus \nu)$-almost 
        everywhere sense) what are the partial functions $\ovadd ( \cdot , t )$, 
        as follows.

        \vspace{6pt}
        
        -- For $t = \gamma_i \in \Gamma$ (with $1 \leq i \leq n$) one has
        \begin{equation}  \label{eqn:11a}
        \ovadd (s, \gamma_i ) = \left\{   \begin{array}{ll}
        1/ \mu ( \, \{ \alpha_i \} \, ),  & \mbox{ if $s = \alpha_i$,}  \\
        0,                                & \mbox{ otherwise}
        \end{array}  \right\},  \ \ s \in \Spec (a).
        \end{equation}

        \vspace{6pt}
        
        -- For $t \in U$ one has
        \begin{equation}  \label{eqn:11b}
        \ovadd (s, t) = 
        - \frac{1}{\pi \cdot f_{\mu \boxplus \nu} (t)} \, \mathrm{Im}
        \bigl( \, \frac{1}{ \omega (t) - s } \, \bigr), \ \ s \in \Spec (a).
        \end{equation}
        \end{theorem}

        \vspace{10pt}

      Upon combining Theorem \ref{thm:1} with a calculation of conditional expectation 
      via disintegration of the overlap measure (as mentioned in Section 1.4), we find: 

      \vspace{6pt}
      
      \begin{theorem}  \label{thm:12}
      Consider the notation of Theorem \ref{thm:1}, and let
      $f : \Spec(a) \to \bR$ be a bounded Borel function.  One has 
			$\bE[ f(a) \mid a +b ] = h(a+b),$
      where $h:\Spec(a+b)\to \bR$ is defined (in $( \mu\boxplus\nu )$-almost everywhere sense) by
        \begin{equation}   \label{eqn:12x}
			h(t)=\begin{cases}
				-\frac{1}{\pi}\frac{1}{f_{\mu\boxplus\nu}(t)}\operatorname{Im} 
				\left(\int_{\bR}f(s)  \frac{1}{\omega(t)-s} d\mu(s)\right), &\mbox{ if } t\in U,  \\
				f( \alpha_i ), &\mbox{ if } t = \gamma_i, \ 1 \leq i \leq n.
			\end{cases}
	\end{equation}
	 \end{theorem}

       \vspace{10pt}
    
	Moreover, when the function $f$ from Theorem \ref{thm:12} is $f(s)=s$, one can 
        further process the formula (\ref{eqn:12x}), as follows.

        \vspace{6pt}

        \begin{proposition}\label{prop:13}
	With the notation of Theorem \ref{thm:1}, one has $\bE(a \, | \, a+b)=h(a+b)$ 
        with $h:\Spec(a+b)\to \mathbb{R}$ defined 
        (in $( \mu\boxplus\nu )$-almost everywhere sense) by
		\begin{equation}   \label{eqn:13x}
			h(t)=\begin{cases}
				-\frac{\operatorname{Im}\left(\omega(t)
               \, G_{\mu}( \, \omega (t) \, ) \right)}{\pi \, f_{\mu\boxplus\nu}(t)}, 
                                                             &\mbox{ if } t\in U,  \\
				\alpha_i, &\mbox{ if } t = \gamma_i, \ 1 \leq i \leq n,
			\end{cases}
		\end{equation}
        where $G_{\mu} : \bC^{+} \to \bC^{-}$ is the Cauchy transform of $\mu$.
	\end{proposition} 

        % [Write-up of Proposition 1.3 in the preceding version of the paper.]
        %
        % \begin{proposition}\label{prop:13}
	%	With the notation of Theorem \ref{thm:1}, one has $\bE(a \, | \, a+b)=h(a+b)$ 
        % where $h:\Spec(a+b)\to \mathbb{R}$ is defined 
        % (in $( \mu\boxplus\nu )$-almost everywhere sense) by
	% 	\begin{equation}   \label{eqn:13x}
	%		h(t)=\begin{cases}
	%			-\frac{\operatorname{Im}\left(\omega(t)
        %       \, G_{\mu\boxplus\nu}(t) \right)}{\pi \, f_{\mu\boxplus\nu}(t)}, 
        %                                                     &\mbox{ if } t\in U,  \\
	%			\alpha_i, &\mbox{ if } t = \gamma_i, \ 1 \leq i \leq n.
	%		\end{cases}
	%	\end{equation}
        %   \end{proposition}
        %
	% \vspace{10pt}
        % 
	% \begin{remark}    \label{rem:14}
	% $1^o$ In the special case when $b$ is a semicircular element, the formula (\ref{eqn:13x})
        % leads to the free additive analogue of Tweedie's formula which was mentioned in Section 1.2 
        % above.
        %

	\vspace{10pt}
    
	\begin{remark}    \label{rem:14}
        $1^o$ A relevant fact for our considerations is that, in the setting of 
        Theorem \ref{thm:1}, the Cauchy transform 
        $G_{\mu \boxplus \nu} : \bC^{+} \to \bC^{-}$ has continuous 
        extension to $\bR$; as a consequence, the subordination relation 
        $G_{\mu \boxplus \nu} (z) = G_{\mu} \bigl( \, \omega (z) \, \bigr)$, 
        $z \in \bC^{+}$, extends to 
        \begin{equation}   \label{eqn:14x}
        G_{\mu \boxplus \nu} (t) 
        = G_{\mu} \bigl( \, \omega (t) \, \bigr), \ \ t \in U.
        \end{equation}
        Another useful fact concerning the extension of $G_{\mu \boxplus \nu}$ to $U$, 
        which appears in connection to the so-called Stieltjes inversion procedure for 
        $\mu \boxplus \nu$, is that the Radon-Nikodym derivative $f_{\mu \boxplus \nu}$
        used in the preceding statements of results can be written as
        \begin{equation}   \label{eqn:14y}
        f_{\mu \boxplus \nu} (t) = - \frac{1}{\pi} 
        \, \mathrm{Im} \bigl( \, G_{\mu \boxplus \nu} (t) \, \bigr), \ \ t \in U. 
        \end{equation}
        With these observations, the first branch in Equation (\ref{eqn:13x}) of the 
        preceding proposition can be put in the form
        \[
        \frac{\operatorname{Im} \bigl(\omega(t)
               \, G_{\mu \boxplus \nu }(t) \bigr)}{ \mathrm{Im} \bigl( 
                             \, G_{\mu\boxplus\nu}(t) \, \bigr) }, \ \ t\in U.
        \]
        It is pleasing that, upon doing some parallel considerations on the 
        second branch of (\ref{eqn:13x}), one can arrive to write an equation of the form 
			\begin{equation}   \label{eqn:14a}
            h(t) = \lim_{\varepsilon \to 0} 
            \frac{\operatorname{Im} \bigl( \omega(t + i \varepsilon)
               \, G_{\mu \boxplus \nu }(t + i \varepsilon) \bigr)}{ \mathrm{Im} 
                                 \bigl( \, G_{\mu\boxplus\nu}(t + i \varepsilon) \, \bigr) }, 
			\end{equation}
which covers at the same time the atomic and the absolutely continuous part 
in the description of $h(t)$.  For the details of how this goes, 
we refer to Subsection \ref{subsec:continuity} below. 
		
\vspace{6pt}

            \noindent
            $2^o$ In the special case when $b$ is a centred semicircular element, 
            Equation (\ref{eqn:13x}) leads to the free additive analogue of Tweedie's 
            formula which was announced in Subsection 1.2.
            The details of how this goes are given in Example \ref{example:47} below.
	\end{remark}
	
%	\vspace{10pt}
	
	\subsection{Description of results -- multiplicative case.}
	Results analogous to those of the preceding subsection hold in connection 
        to the multiplication of two freely independent random variables $a$ and $b$,
	in two cases -- when $a,b$ are assumed to be positive, 
	and respectively when $a,b$ are assumed to be unitary.
	We will focus on the positive case, which is presented in this subsection and 
        then detailed in Section 5 below.  
	For the unitary case, a brief exposition of the relevant statements appears
	in Appendix A at the end of the paper.
	
	Concerning the multiplication of freely independent positive random variables, 
        we prove the following analogue of Theorem \ref{thm:1}.
	
	\vspace{6pt}
	
	\begin{theorem}
		Let $a,b$ be freely independent positive random variables in a $W^{*}$-probability space $( \cA, \varphi )$. Let $\mu$ and respectively $\nu$ be the distributions of $a$ and $b$ with respect to $\varphi$, and assume that neither of $\mu, \nu$ is a point mass. Denote by $f_{\mu\boxtimes\nu}$
		the density of the absolutely continuous part of $\mu\boxtimes\nu$ and by 
		$\omega$ the continuation to the real line of the subordination function of 
		$\mu\boxtimes \nu$ with respect to $\mu$. Then there exists a closed 
		set $0 \in E \subseteq [0, \infty )$, of Lebesgue measure zero, such that  
		\begin{align*}
			\ovmulp(s,t)=
			\begin{cases}
				-\frac{1}{\pi} \frac{1}{tf_{\mu\boxtimes\nu}(t)}
				\operatorname{Im}\left(\frac{1}{1-\omega(1/t)s}\right)\quad 
				&\mbox{ if } s\in\mathrm{supp}(\mu)\\ 
				&\mbox{ and } t\in \mathrm{supp}(\mu\boxtimes\nu)^{ac}\setminus E,\\
				\frac{1}{\mu(\{s\})} 1_{\omega(1/t)=1/s}\qquad 
				&\mbox{ if } s\in\mathrm{supp}(\mu)\\ 
				&\mbox{ and $t>0$ is an atom of $\mu\boxtimes\nu$},\\
				\frac{1}{\nu(\{0\})}\frac{1}{1-s \psi_\mu^{-1}(\nu(\{0\})-1)}1_{\nu(\{0\})>\mu(\{0\})} \quad 
				&\mbox{ if } s\in\mathrm{supp}(\mu)\setminus\{0\}, t=0,\\
				\frac{1}{\mu\boxtimes\nu(\{0\})}1_{\mu\boxtimes\nu(\{0\})>0} \qquad 
				&\mbox{ if } s=t=0,
			\end{cases}
		\end{align*}
		defines $\mu\times\mu\boxtimes\nu$-almost everywhere a function $\ovmulp:\Spec(a)\times \Spec(a^{1/2}ba^{1/2})\to \bR_+$. Then
		$d\muovmulp(s,t)=\ovmulp(s,t) \mu(ds)\,\mu\boxtimes\nu(dt)$. In particular $\muovmulp\ll \mu\times \mu\boxtimes\nu$ and $\ovmulp$ is the overlap function of $a$ and $a^{1/2}ba^{1/2}$.
	\end{theorem}
	
	Same as with the description of the overlap function, the general description 
	of conditional expectations $E(f(a)|a^{1/2}ba^{1/2})$ is more involved than in the additive case;
	we refer to Corollary \ref{cor:55} for the precise statement. In the case of $f(x)=x$ we can get a compact form similar to the one in the additive case
		\[
		h(t)=\lim_{\varepsilon\to 0}\frac{\operatorname{Im}\left(\frac{G_{\mu\boxtimes \nu}(t+i\varepsilon)-\frac{1}{t+i\varepsilon}}{\omega(\frac{1}{t+i\varepsilon})}\right)}{\operatorname{Im}\left(G_{\mu\boxtimes \nu}(t+i\varepsilon)\right)}.
		\]
		We make more comments about this formula in Subsection \ref{subsec:continuity}.
	
	We next present here the special 
	case when $b$ is a free Poisson element, which gives the free multiplicative analogue of Tweedie's formula.
	
	\vspace{10pt}
	
	\begin{example}   \label{example:15}
		{\em (Free multiplicative Tweedie's formula.)}
		Let $a,b$ be freely independent positive random variables with respective distributions $\mu, \nu$ in a $W^{*}$-probability space $( \cA , \varphi )$. Assume that $\nu$ is the free Poisson distribution with parameter $\lambda>0$ and, for simplicity, let us also assume
        that $\mu$ has no atom at zero. We then get 
		$E(a \mid a^{1/2}ba^{1/2}) = h(a^{1/2}ba^{1/2})$ with
		\begin{equation}  \label{eqn:16a}
			h(t)=
			\begin{cases}
				\frac{\lambda t}{|\lambda-1+tG_{\mu\boxtimes\nu}(t)|^2} & \mbox{ if } t\in\mathrm{supp}(\mu\boxtimes\nu)\setminus \{0\},\\
				-\frac{\lambda}{(1-\lambda)G_{(\mu\boxtimes\nu)^{ac}}\left(0\right)} &\mbox{ if } t=0 \mbox{ when } \lambda<1,
			\end{cases}
		\end{equation}
		where $G_{\mu\boxtimes\nu}$ and $G_{(\mu\boxtimes\nu)^{ac}}$ are the 
		respective Cauchy transforms of $\mu\boxtimes\nu$ and of its absolutely continuous part.
        As explained in Section \ref{Sec:MatrixDenoising} below, this 
        recovers a well-known formula of Ledoit and P\'ech\'e \cite{LedoitPeche} 
        for the shrinkage estimator of a covariance matrix.  Similarly to what was the case 
        for the free additive Tweedie's formula (\ref{eqn:freetweedieboxplus}), we note here 
        the remarkable fact that, in Equation (\ref{eqn:16a}), the description of $h$ only 
        makes use of the distribution of $a^{1/2}ba^{1/2}$. 
        %
	% This is another remarkable example of situation in which the conditional
        % expectation depends only on $a^{1/2}ba^{1/2}$ and its distribution. 
        % As explained in Section \ref{Sec:MatrixDenoising} below, in this way we 
        % recover a well-known formula of Ledoit and P\'ech\'e \cite{LedoitPeche} 
        % for the shrinkage estimator of a covariance matrix.
        %
	\end{example}
    
% \vspace{10pt}
	
	\subsection{General noise and conditional freeness.}
	In the general case when the noisy element is given by $P(a,b)$ our goal is, once again,
	to find $\bE(a \mid P(a,b))$. We show that this problem can be reduced
	to the problem of finding the distribution of $P(a,b)$ where $a,b$ are conditionally 
	free (or ``c-free'', for short) in the sense of 
	Bo\.zejko-Leinert-Speicher \cite{BozejkoLeinertSpeicher}. 
    
	\begin{theorem}   \label{thm:17}
	Let $a,b$ be freely independent selfadjoint non-commutative random variables 
        in a $W^{*}$-probability space $( \cA , \varphi )$. Let $f:\mathbb{R}\to \mathbb{R}$ 
	be a bounded non-negative Borel-measurable function such that $\varphi(f(a))=1$. Define 
        another state $\chi:\mathcal{A}\to \mathbb{C}$ by
		\[
		\chi(y):=\varphi(f(a)y),\qquad \forall y\in\mathcal{A}.
		\]
	Then $a,b$ are conditionally free with respect to the pair of states $(\varphi,\chi)$.
	Moreover, let $P$ be a selfadjoint polynomial in two non-commuting variables and
	denote by $\mu^{\varphi}_{P(a,b)}$ and $\mu^{\chi}_{P(a,b)}$ the distributions of 
	$P(a,b)$ with respect to $\varphi$ and $\chi$, respectively. Then 
	$\mu^{\chi}_{P(a,b)}\ll\mu^{\varphi}_{P(a,b)}$ and
		\[
		E(f(a)\mid P(a,b)) = h \bigl( \, P(a,b) \, \bigr) \mbox{ for }
		h = \frac{d\mu^{\chi}_{P(a,b)}}{d\mu^{\varphi}_{P(a,b)}}.
		\]
	\end{theorem}

        In the setting of Theorem \ref{thm:17}, the elements $a,b$ are freely 
        independent with respect to $\varphi$, hence $\mu^{\varphi}_{P(a,b)}$ 
        can be approached by the methods developed in \cite{BelinschiMaiSpeicher}
        for finding distributions of general polynomials and rational functions 
        in free random variables. Theorem \ref{thm:17} brings up the question of 
        understanding, in the $c$-free setting, what is the distribution 
        $\mu^{\chi}_{P(a,b)}$. We mention that the latter question is currently 
        being investigated in an ongoing work that comes in the continuation of 
        the paper \cite{LehnerSzpojankowski}; this is related to the fact 
        that \cite{LehnerSzpojankowski} presents a solution to  the complementary 
        problem of computing conditional expectations of resolvents 
        $E((z-P(a,b))^{-1}|a)$ (which offers another method for approaching the 
        distribution $\mu^{\varphi}_{P(a,b)}$).
        %
	% There exist methods of finding distributions of general polynomials and 
        % rational functions of free random variables (see \cite{BelinschiMaiSpeicher}).
        % Moreover in \cite{LehnerSzpojankowski} the authors present a solution to 
        % the complementary problem of computing conditional expectations 
        % $E((z-P(a,b))^{-1}|a)$, for a polynomial $P$ of free random variables 
        % $a,b$ which, as a consequence, gives another method for determining the 
        % distribution of $P(a,b)$.  Hence the above theorem reduces the problem of 
        % free denoising to the problem of finding the distribution of general 
        % elements in conditionally free random variables. This question is currently 
        % being investigated.
        %
        
% \vspace{10pt}

\subsection{Relation to matrix denoising.}  
The problem of free denoising studied 
    in the present paper is a counterpart, set in the framework of a $W^{*}$-probability
    space, for the statistical problem studied in random matrix theory under the name 
    of {\em matrix denoising}.  In the latter problem, an unknown signal represented by a 
    Hermitian $N \times N$ random matrix $A_N$ is corrupted by a noise which is itself a 
    Hermitian $N \times N$ random matrix $B_N$, and one would like to ``recover the signal'' 
    $A_N$ from a noisy observation of the form $P(A_N,B_N)$ -- that is, find an estimate of 
    $A_N$ as a function of $P(A_N,B_N)$.  Free denoising and matrix denoising are related,
    due to a well-studied phenomenon of asymptotic free independence for random matrices 
    (see e.g.~Chapter 4 of the monograph \cite{MingoSpeicher}).  
    This explains, in particular, why the free denoisers giving analogues for Tweedie's 
    formula (Equations \eqref{eqn:freetweedieboxplus} and \eqref{eqn:16a}) are in 
    agreement with formulas that appeared previously in the random matrix literature 
    \cite{BunAllezBouchaudPotters, LedoitPeche}, in papers about matrix denoising.  
    The methods used in these papers are different in scope from 
    those that come up naturally in the $W^{*}$-probability setting (but see, however, the 
    brief discussion about this made in Section 19.4 of the monograph \cite{PottersBouchaud}).  

    Section \ref{Sec:MatrixDenoising} of the present paper gives a more detailed account 
    of the connection between free denoising and matrix denoising, starting with the general
    observation that the notion of overlap measure (which also makes sense in a Borel-ncps
    of $N \times N$ random matrices) is well-behaved under the natural notion of convergence 
    in moments for pairs of selfadjoint elements.  In a nutshell, the point made in 
    Section \ref{Sec:MatrixDenoising} is as follows: suppose that a sequence of pairs of 
    Hermitian random matrices $(A_N, B_N)_{N=1}^{\infty}$ converges in moments to a pair 
    $(a,b)$ of freely independent selfadjoint elements in a $W^{*}$-probability space; then 
    the free denoiser $h$ for $a$ with respect to $P(a,b)$ provides an asymptotically optimal 
    solution, in $L^2$-sense, to the corresponding 
    matrix denoising problem of $A_N$ with respect to $P(A_N, B_N)$.  The term 
    ``asymptotically optimal'' is considered here with the meaning that no bounded
    continuous function can asymptotically outperform $h$ in the $L^2$-sense; for the precise 
    statements, see Corollary \ref{cor:93} in 
    Subsection \ref{Sec:MatrixDenoising}.1, and the specific discussion of additive 
    and multiplicative denoising in Subsections \ref{Sec:MatrixDenoising}.2 and 
    \ref{Sec:MatrixDenoising}.3.

\subsection{Organization of the paper.}
Besides the Introduction, this paper has eight more sections and an Appendix,
as follows.

\vspace{4pt}

\noindent
-- We introduce and study the overlap measure 
and overlap function in Section \ref{section:Y}, in the setting of a Borel-ncps,
after having discussed the Borel-ncps framework in Section \ref{section:X}.

\vspace{4pt}

\noindent
-- Starting with Section \ref{section:Z}, we focus on the setting of a 
$W^{*}$-probability space; in particular, Section \ref{section:Z} discusses 
how, in general, the disintegration of the overlap measure is used to compute 
conditional expectations $E(x \, | \, y)$ in the $W^{*}$-framework.  

\vspace{4pt}

\noindent
-- After a review of 
necessary background in Section \ref{SubordReview}, we take on additive free denoisers 
$E( a \, | \, a+b )$ in Section \ref{AdditiveFormulas}, where we detail the results 
announced in the above Subsection 1.5 of the Introduction.  Likewise, in 
Section \ref{MultiplicativeFormulas} we detail the results about multiplicative free 
denoisers $E( a \, | \, a^{1/2} b a^{1/2} )$ that were announced in Subsection 1.6. 

\vspace{4pt}

\noindent
-- In Section \ref{CFreeFormulas} we consider the case of a general noise 
(beyond the additive and multiplicative cases), and we explain how free denoising 
relates to conditional freeness.  

\vspace{4pt}

\noindent
-- Section \ref{Sec:MatrixDenoising}
is devoted to discussing the relation between free denoising and matrix denoising. 

\vspace{4pt}

\noindent
-- The paper concludes with an appendix which outlines how the overlap measure 
approach can be adapted, in the multiplicative case, in connection to the 
multiplication of free unitaries.

\vspace{0.5cm}

\section{The framework of a Borel-ncps}   \label{section:X}

\begin{notation-and-remark}  \label{def:X1}
The underlying structure we start from is the one of a {\em $*$-probability space} 
$( \cA, \varphi )$, with $\varphi$ a faithful trace.  Thus $\cA$ is a unital 
$*$-algebra over $\bC$ and $\varphi : \cA \to \bC$ is a linear map which is positive 
and faithful ($\varphi (a^{*}a ) \geq 0$ for all $a \in \cA$, with equality if and 
only if $a = \zeroA$), is tracial ($\varphi (ab) = \varphi (ba)$ for all 
$a,b \in \cA$), and has $\varphi ( \oneA ) = 1$.

For $( \cA , \varphi )$ as above, one considers the set of selfadjoint elements
$\cA^{\mathrm{sa}} := \{ x \in \cA \mid x = x^{*} \}$, and the set of positive 
elements
\[
\cA^{+} := \Bigl\{ x \in \cA \begin{array}{ll} 
\vline & \exists \, k \in \bN \mbox{ and } a_1, \ldots , a_k \in \cA \\
\vline & \mbox{such that $x = a_1^{*} a_1 + \cdots + a_k^{*} a_k$}
\end{array}  \Bigr\}.
\]
It is immediate that $\cA^{+}$ is a cone, i.e.~it is stable under linear 
combinations with coefficients in $[0, \infty )$, and that $\cA^{\mathrm{sa}}$ 
is the vector space over $\bR$ which is generated by $\cA^{+}$ (in the verification
of latter fact it is useful to observe that an $x \in \cA^{\mathrm{sa}}$
can be written as $x = \frac{1}{4} (x+ \oneA )^2 - \frac{1}{4} (x- \oneA)^2$, 
with $(x \pm \oneA)^2 \in \cA^{+}$).  The faithfulness of $\varphi$ entails the 
implication 
$\bigl[ x \mbox{ and $-x$ both in $\cA^{+}$} \bigr] \Rightarrow [x= \zeroA]$, 
and this in turn implies that the cone $\cA^{+}$ gives a partial 
order on $\cA^{\mathrm{sa}}$, where for $x,y \in \cA^{\mathrm{sa}}$ we put 
$\bigl[ \, x \leq y \, \bigr] \ \eqdef \,  
\bigl[ \, y-x \in \cA^{+} \, \bigr]$.
It is useful to record that, besides its obvious properties related to 
addition and scalar multiplication, the partial order so obtained has the 
property that:
\begin{equation}  \label{eqn:X1a}  
\bigl[ \, x \leq y \mbox{ in } \cA^{\mathrm{sa}} \, ] \ \Rightarrow
\ \bigl[ \, a^{*} x a \leq a^{*} y a \mbox{ in } \cA^{\mathrm{sa}}, 
\ \forall \, a \in \cA \, ].
\end{equation}
In order to verify the latter inequality, one picks
$a_1, \ldots , a_k \in \cA$ such that
$y - x = \sum_{j=1}^k a_j^{*} a_j$ and observes that 
$a^{*} y a - a^{*} x a = \sum_{j=1}^k (a_j a)^{*} (a_j a) \in \cA^{+}$.

For the purpose of this paper, it is important to be able to invoke 
a suitable ``functional calculus'' for an element $x \in \cA^{\mathrm{sa}}$.
It is evident how that can be done in connection to a polynomial function: if
$f : \bR \to \bC$ is defined by 
$f(t) = \alpha_0 + \alpha_1 t + \cdots + \alpha_k t^k$
(with $\alpha_0, \ldots , \alpha_k \in \bC$) then for
$x \in \cA^{\mathrm{sa}}$ we simply put 
$f(x) := \alpha_0 \oneA + \alpha_1 x + \cdots + \alpha_k x^k \in \cA$. 
But we will need to consider elements $f(x) \in \cA$ which are defined
for $f : \bR \to \bC$ belonging to a larger space of functions, 
as described in Definition \ref{def:X3} below.
\end{notation-and-remark}

\vspace{10pt}

\begin{notation}  \label{def:X2}
We let $\Bor ( \bR )$ denote the unital $*$-algebra of Borel functions 
from $\bR$ to $\bC$.  We will work with the unital $*$-subalgebra
$\Bor_b ( \bR ) := \bigl\{ f \in \Bor ( \bR ) \mid f 
\mbox{ is bounded} \bigr\},$ and with the slightly larger unital
$*$-subalgebra of Borel functions with polynomial growth,
\[
\Borpol ( \bR ) = \Bigl\{ f \in \Bor ( \bR )
\begin{array}{ll}
\vline &  \exists \, \alpha, \beta  \in [ 0, \infty ) 
          \mbox{ and $p \in \bN$ such that} \\ 
\vline & \mbox{ $|f(t)| \leq \alpha + \beta \, t^{2p}$
               for all $t \in \bR$}
 \end{array} \Bigr\} .
\]
\end{notation}

\vspace{10pt}

\begin{definition}  \label{def:X3}
We will use the name {\em Borel-ncps} 
for a triple $( \cA , \varphi ; \Gamma )$ where: 
$( \cA , \varphi )$ is a $*$-probability space with $\varphi$ faithful trace, 
and $\Gamma$ is a family $\bigl( \Gamma_x )_{ x \in \cA^{\mathrm{sa}} }$ 
of unital $*$-homomorphisms from $\Borpol ( \bR )$ to $\cA$, 
such that the following conditions are fulfillled:
\begin{equation}   \label{eqn:X3a}
% \mbox{(Borpol-Id)}  \hspace{0.5cm} 
\left\{  \begin{array}{l}
\mbox{Let $\mathrm{id} \in \Borpol ( \bR )$ be the identity function, 
      $\mathrm{id} (t) = t$.}  \\
\mbox{Then one has: $\Gamma_x ( \mathrm{id} ) = x, \ \forall \, x \in \cA^{\mathrm{sa}}$.}
\end{array}  \right.
\end{equation}

\begin{equation}   \label{eqn:X3b}
% \mbox{(Borpol-Conv)}  \hspace{0.5cm}
\left\{  \begin{array}{l}
\mbox{Let $f,g$ and $f_1, f_2, \ldots , f_n, \ldots$ in $\Borpol ( \bR )$ be}  \\
\mbox{ $\ $ such that $|f_n| \leq g$ for all $n \in \bN$, and}  \\
\mbox{ $\ $ such that $\lim_{n \to \infty} f_n (t) = f(t)$ for every $t \in \bR$.} \\
\mbox{Then one has: } \lim_{n \to \infty} 
      \varphi \bigl( \, x \, \Gamma_y ( f_n ) \, \bigr) 
    = \varphi \bigl( \, x \, \Gamma_y (f) \, \bigr), \ \forall \, x,y \in \cA^{\mathrm{sa}}.
\end{array}  \right.
\end{equation}
\end{definition}

\vspace{10pt}

\begin{remark}   \label{rem:X4}
The way to think about $\Gamma$ in the preceding definition is that,
for $x \in \cA^{\mathrm{sa}}$, the unital $*$-homomorphism 
$\Gamma_x : \Borpol ( \bR ) \to \cA$ tells us how to do functional calculus 
of $x$, with functions from $\Borpol ( \bR )$.  We will in fact use the 
shorthand notation
\begin{equation}   \label{eqn:X4a}
\Gamma_x (f) =: f(x), 
\ \mbox{for $x \in \cA^{\mathrm{sa}}$ and $f \in \Borpol ( \bR )$. }
\end{equation}
With this notation, the condition (\ref{eqn:X3a}) takes the form 
``$\mathrm{id} (x) = x$''; together with the homomorphism properties 
postulated for $\Gamma_x$, this immediately implies that (\ref{eqn:X4a})
extends the formula considered for $f(x)$ in the last paragraph of 
Remark \ref{def:X1}, in the case when $f$ is a polynomial function.
Another consequence of the $*$-homomorphism properties of $\Gamma_x$
is that $f(x) \in \cA^{\mathrm{sa}}$ whenever $f \in \Borpol ( \bR)$ is 
real-valued, and one has that
\begin{equation}   \label{eqn:X4b}
\left[ \begin{array}{c}
f,g \in \Borpol ( \bR ), \mbox{ real-valued,}  \\
\mbox{with $f(t) \leq g(t)$ for every $t \in \bR$}
\end{array} \right] \ \Rightarrow \bigl[ \, f(x) \leq g(x) \, \bigr] 
\end{equation}
(for the latter fact, one writes $g(x) - f(x) = a^{*} a$ with 
$a = \sqrt{g-f} (x)$).

Let us also record that, in terms of the notation from 
(\ref{eqn:X4a}), the convergence stated on the last line of (\ref{eqn:X3b}) 
says that, for functions $f$ and $f_n$ as described there, one has
\begin{equation}   \label{eqn:X4c}
\lim_{n \to \infty} \varphi \bigl( \, x \, f_n (y) \, \bigr) 
= \varphi \bigl( \, x \, f(y) \, \bigr), \ \ \forall \, x,y \in \cA^{\mathrm{sa}}.
\end{equation}
%
% The requirement that selfadjoint elements have Borel functional calculus is 
% quite restrictive, since it forces the $*$-probability space $( \cA , \varphi )$
% to have a large population of elements obtained via functional calculus.  
% Nevertheless, this requirement will be fulfilled in the two types of examples 
% that are considered in the present paper, and are reviewed next.
%
\end{remark}

\vspace{10pt}

\begin{remark-and-notation}   \label{rem:X5}
Let $( \cA , \varphi ; \Gamma )$ be a Borel-ncps and let 
$x \in \cA^{\mathrm{sa}}$. The functional calculus for $x$ allows us 
to consider the family of projections
$P(t) := \mathbf{1}_{]-\infty;t]}(x) \in \cA$, $t \in \bR$, and
to define a non-decreasing function $F : \bR \to [0,1]$ by
$F(t) := \varphi ( \, P(t) \, ), \ t \in \bR.$  By using
the condition (\ref{eqn:X3b}) from Definition \ref{def:X3} one sees 
that $F$ is continuous on the right, with $\lim_{t \to - \infty} F(t) = 0$ 
and $\lim_{t \to \infty} F(t) = 1$.  This implies the existence of a 
Borel probability measure $\mu_x$ on $\bR$, uniquely determined, such that 
$F$ is the cumulative distribution function of $\mu_x$.  This $\mu_x$ will 
be referred to as the {\em distribution of $x$} with respect to $\varphi$.
Standard integration arguments show that $\mu_x$ has finite moments of all 
orders, where for every $n \in \bN$ the moment of order $n$ of $\mu_x$ is 
equal to $\varphi (x^n)$.  More generally, the functions from
$\Borpol ( \bR )$ are $\mu_x$-integrable, and one has
\begin{equation}  \label{eqn:X5a}
\int_{\bR} f \, d \mu_x = \varphi \bigl( \, f(x) \, \bigr),
\ \ \forall \, f \in \Borpol ( \bR ).
\end{equation}
\end{remark-and-notation}

\vspace{10pt}

\begin{remark-and-notation}   \label{rem:X6}
Let $( \cA , \varphi, \Gamma )$ be a Borel-ncps.  We will use the notation 
``$\Spec (a)$'' for the spectrum of an $a \in \cA$,
$\Spec (a) := \bigl\{ \lambda \in \bC \mid \lambda \oneA - a
\mbox{ is not invertible} \bigr\},$
and we note the implication
\begin{equation}   \label{eqn:X6b}
\bigl[ \, x \in \cA^{\mathrm{sa}} \, \bigr] \ \Rightarrow
\bigl[ \, \Spec (x) \subseteq \bR \, \bigr].
\end{equation}
In other words, (\ref{eqn:X6b}) says that $\lambda \oneA - x$ is 
invertible for every $x \in \cA^{\mathrm{sa}}$ and
$\lambda \in \bC \setminus \bR$.  This holds because 
we can let $f_{\lambda} \in \Borpol ( \bR )$ be defined 
by $f_{\lambda} (t) = 1/ ( \lambda - t )$, $t \in \bR$, and we can use 
the functional calculus of $x$ to get
$f_{\lambda} (x) \cdot ( \lambda \oneA - x )$
$ = \oneA = ( \lambda \oneA - x ) \cdot f_{\lambda} (x),$
implying the required invertibility of $\lambda \oneA - x $.
\end{remark-and-notation}

\vspace{10pt}

% We next review the two main examples of Borel-ncps structure that 
% will be considered in this paper.
%
% \vspace{6pt}

\begin{example}  \label{example:X7}
Let $( \cA , \varphi )$ be a tracial $W^{*}$-probability space, that is, a 
$*$-probability space where $\cA$ is a von Neumann algebra and $\varphi$ is
a normal faithful trace-state.  This provides an example of
Borel-ncps, where for every $x \in \cA^{\mathrm{sa}}$ the $*$-homomorphism 
$\Gamma_x : \Borpol ( \bR ) \to \cA$ is the customary functional 
calculus for selfadjoint elements (see e.g.~\cite[Sections I.4 and III.1]{Takesaki}).
A special feature of this framework is that, for every 
$x \in \cA^{\mathrm{sa}}$, the spectrum $\Spec (x)$ is a non-empty compact
subset of $\bR$, precisely equal to the support of the distribution 
$\mu_x$ (see e.g.~\cite[Lecture 3, Proposition 3.15]{NicaSpeicher}).  For any 
$f \in \Borpol ( \bR )$ and $x \in \cA^{\mathrm{sa}}$, the element $f(x)$ only 
depends on the restriction of $f$ to $\Spec (x)$, which is sure to be bounded 
(which explains why $f(x)$ comes out as a bounded linear operator, belonging 
to $\cA$). 
\end{example}

\vspace{10pt}

In preparation of our second main example of Borel-ncps structure, we review 
the elementary notion of functional calculus for a complex Hermitian matrix.

\vspace{6pt}

\begin{notation-and-remark}   \label{def:X8}
Pick $N \in \bN$ and let 
$\cM_N^{\mathrm{sa}} ( \bC ) = \{ X \in \cM_N ( \bC ) \mid X = X^{*} \}$.  

\vspace{6pt}

\noindent
$1^o$ Every function $f : \bR \to \bC$ induces a function
\begin{equation}   \label{eqn:X8a}
\Phi_f : \cM_N^{\mathrm{sa}} ( \bC ) \to \cM_N ( \bC ), 
\ \ \Phi_f (X) := f(X),
\end{equation}
with ``$f(X)$'' defined in the natural way,
$f(X) := \alpha_0 I_N + \alpha_1 X + \cdots + \alpha_k X^k$ where
$p(t) = \alpha_0 + \alpha_1 t + \cdots + \alpha_k t^k$ is any polynomial in
$\bC [t]$ which agrees with $f$ on the set of eigenvalues of $X$.
Equivalently, one can write $f(X) = U \, f(D) \, U^{*}$, where $X = UDU^{*}$ 
is any orthogonal diagonalization of $X$, and $f(D)$ is obtained by applying
$f$ to the diagonal entries of $D$. 

We note the following properties of the function $\Phi_f$ from (\ref{eqn:X8a}):

(i) If $f : \bR \to \bC$ is continuous, then $\Phi_f$ is continuous 
in the norm topology.

(ii) If $f : \bR \to \bC$ is a Borel function, then so is $\Phi_f$ (where
$\cM_N^{\mathrm{sa}} ( \bC )$ and $\cM_N ( \bC )$ are 

\hspace{0.4cm} considered with their 
natural Borel sigma-algebras).

\noindent
The verification of (i) is made by uniformly approximating $f$ with 
polynomials on compact intervals of $\bR$.  Then (ii) follows from
(i) and the observation that the algebra of functions
$\bigl\{ f \in \Bor ( \bR ) \mid \mbox{$\Phi_f$ is Borel} \bigr\}$
is closed under sequential pointwise convergence.

\vspace{6pt}

\noindent
$2^o$ Suppose now that we are also given a probability space
$( \Omega , \cF , P )$, and let us consider the unital $*$-algebra
of matrix-valued random variables
\begin{equation}   \label{eqn:X8b}
L^0 \bigl( \Omega, \cM_N ( \bC ) \bigr) :=
\Bigl\{ a : \Omega \to \cM_N ( \bC ) \begin{array}{ll}
\vline & a \mbox{ is measurable between $\cF$ and}  \\
\vline & \mbox{the Borel sigma-algebra of $\cM_N ( \bC )$}
\end{array}  \Bigr\}
\end{equation}
(with random variables $a,b$ identified when they coincide $P$-almost
everywhere).  For every 
$x = x^{*} \in L^0 \bigl( \Omega, \cM_N ( \bC ) \bigr)$ and 
$f \in \Bor ( \bR )$ we then define
$f(x) : \Omega \to \cM_N ( \bC )$ by
\begin{equation}   \label{eqn:X8c}
\bigl( \, f(x) \, \bigr) ( \omega ) := f ( \, x(\omega) \, ) 
\in \cM_N ( \bC ), \ \ \omega \in \Omega.
\end{equation}
The function $f(x)$ defined in (\ref{eqn:X8c}) is measurable, hence
is in $L^0 \bigl( \Omega, \cM_N ( \bC ) \bigr)$, since
$f(x) = \Phi_f \circ x$ with $\Phi_f$ as in (\ref{eqn:X8a}).
It is immediate that, for a fixed  
$x = x^{*} \in L^0 \bigl( \Omega, \cM_N ( \bC ) \bigr)$,
the functional calculus defined by (\ref{eqn:X8c}) respects the algebraic 
structure of $\Bor ( \bR )$, e.g.~one has that $(f \, g ) (x) = f(x) \, g(x)$ 
for any $f,g \in \Bor ( \bR )$, or that 
$f(x) \in \bigl[ \, L^0 \bigl( \Omega, \cM_N ( \bC ) \bigr) \, \bigr]^{+}$
for any real non-negative valued $f \in \Bor ( \bR )$.
\end{notation-and-remark}

\vspace{10pt}

\begin{example}  \label{example:X9}
We continue to use the framework of the preceding notation, and for every 
$a \in L^0 \bigl( \Omega, \cM_N ( \bC ) \bigr)$ and $1 \leq j,k \leq N$
we let $\Entry_{j,k} (a) : \Omega \to \bC$ be the random variable which 
selects the $(j,k)$-entry of $a$.  For
$a \in L^0 \bigl( \Omega, \cM_N ( \bC ) \bigr)$ we will also use the 
customary notation $\mathrm{Tr} (a) := \sum_{j=1}^N \Entry_{j,j} (a)$.

The unital $*$-algebra used by this example of Borel-ncps is
\begin{equation}   \label{eqn:X9a}
\cA := \bigl\{ a \in L^0 \bigl( \Omega, \cM_N ( \bC ) \bigr) \mid
\Entry_{j,k} (a) \in L^{\infty -} ( \Omega ) \mbox{ for all }
1 \leq j,k \leq N \bigr\}, 
\end{equation}
where $L^{\infty -} ( \Omega )$ is the algebra of complex-valued 
random variables with finite moments of all orders on
$( \Omega , \mathcal{F} , P )$.  We note that, thanks to the identity
\[
\mathrm{Tr} (a^{*} a ) = \sum_{j,k=1}^N | \Entry_{j,k} (a) |^2, 
\ a \in L^0 \bigl( \Omega, \cM_N ( \bC ) \bigr),
\]
one can also describe our algebra $\cA$ from (\ref{eqn:X9a}) 
in the form
\begin{equation}   \label{eqn:X9b}
\cA = \bigl\{ a \in L^0 \bigl( \Omega, \cM_N ( \bC ) \bigr) \mid
\mathrm{Tr} (a^{*} a ) \in L^{\infty -} ( \Omega ) \bigr\}.
\end{equation}
The functional $\varphi : \cA \to \bC$ that we consider is defined by
$\varphi (a) := \frac{1}{N} \, \mathbb{E} ( \, \mathrm{Tr} (a) \, ), 
\ a \in \cA.$
It is straightforward to check that in this way
we get a $*$-probability space $( \cA , \varphi )$ where $\varphi$ is a 
faithful trace.  

We next observe that the selfadjoint elements of $\cA$ have 
a natural notion of functional calculus, coming from the following fact.

\vspace{6pt}

\noindent
{\em Fact 1.} Let $x \in \cA^{\mathrm{sa}}$ and $f \in \Borpol ( \bR )$,
and let 
$f(x) \in L^0 \bigl( \Omega, \cM_N ( \bC ) \bigr)$
be defined by (\ref{eqn:X8c}).  

\hspace{0.75cm} Then $f(x) \in \cA$.

\noindent
{\em Verification of Fact 1.} Since $( \, f(x) \, )^{*} \, f(x) = |f|^2 (x)$,
what we must check here (according to (\ref{eqn:X9b})) is that the non-negative 
random variable $\mathrm{Tr} \bigl( \, |f|^2 (x) \, \bigr)$ belongs to 
$L^{\infty -} ( \Omega )$.
To that end we pick $\alpha, \beta \in [0, \infty )$ and $p \in \bN$ such that 
$|f|^2 (t) \leq \alpha + \beta \, t^{2p}$ for all $t \in \bR$, and we use the 
properties of functional calculus on $L^0 \bigl( \Omega, \cM_N ( \bC ) \bigr)$
to find that
\begin{equation}   \label{eqn:X9c}
\mathrm{Tr} \bigl( \, |f|^2 (x) \, \bigr) \leq
\mathrm{Tr} \bigl( \, \alpha I_N + \beta x^{2p} \, \bigr)
= N \alpha + \beta \mathrm{Tr} ( \, x^{2p} \, ).
\end{equation}
The random variable $\mathrm{Tr} ( \, x^{2p} \, )$ is in 
$L^{\infty -} ( \Omega )$, because it can be written as an algebraic 
expression in the entries 
$\Entry_{j,k} (x) \in L^{\infty -} ( \Omega )$.  From the inequality (\ref{eqn:X9c})
it then follows that $\mathrm{Tr} \bigl( \, |f|^2 (x) \, \bigr)$ is in 
$L^{\infty -} ( \Omega )$ as well.
\hfill  $\square$

\vspace{6pt}

It is immediate that, for every $x \in \cA^{\mathrm{sa}}$, the assignment
$f \mapsto f(x)$ defines a unital $*$-homomorphism
$\Gamma_x : \Borpol ( \bR ) \to \cA$ such that $\Gamma_x ( \mathrm{id} ) = x$.
In order to confirm that we are dealing with a Borel-ncps structure, we are 
thus left to check a dominated convergence condition:

\vspace{6pt}

\noindent
{\em Fact 2.} The condition (\ref{eqn:X3b}) from Definition \ref{def:X3} is satisfied.

\noindent
{\em Verification of Fact 2.} Let $( f_n )_{n=1}^{\infty}, f,g \in \Borpol ( \bR )$ 
and $x,y \in \cA$ be as in (\ref{eqn:X3b}). We have
\begin{equation}  \label{eqn:X9d}
\varphi \bigl( \, x \, f_n (y) \, \bigr)
= \frac{1}{N} \int_{\Omega} \mathrm{Tr} 
\bigl( \, x( \omega ) \, f_n ( y( \omega ) ) \, \bigr) \, dP( \omega ),
\ n \in \bN,
\end{equation}
and a similar formula holds for $\varphi \bigl( \, x \, f(y) \, \bigr)$.
The pointwise convergence of the $f_n$'s to $f$ implies that
for every $\omega \in \Omega$ we have 
$|| \cdot ||-\lim_{n \to \infty} f_n ( y( \omega ) )
= f( y( \omega ) )$ (norm-convergence in $\cM_N ( \bC )$),
and this immediately entails that
$\lim_{n \to \infty} \mathrm{Tr} 
\bigl( \, x( \omega ) \, f_n ( y( \omega ) ) \, \bigr)
= \mathrm{Tr} 
\bigl( \, x( \omega ) \, f( y( \omega ) ) \, \bigr).$
Thus the integrand on the right-hand side of (\ref{eqn:X9d})
converges pointwise, for $n \to \infty$, to its counterpart having
$f$ instead of $f_n$ in its description.

In order to complete the verification of Fact 2, it will thus be sufficient 
to find a dominating function (independent of $n$) for the integrand on the 
right-hand side of (\ref{eqn:X9d}) -- once this is done, the Lebesgue dominated 
convergence theorem will give us the desired convergence of 
$\varphi \bigl( \, x \, f_n (y)\, \bigr)$ to
$\varphi \bigl( \, x \, f(y) \, \bigr)$.

Towards finding a dominating function, we pick some 
$\alpha, \beta \in [ 0, \infty )$ and $p \in \bN$ such that 
$g^2 (t) \leq \alpha + \beta \, t^{2p}$ for all $t \in \bR$; this implies 
that $|f_n|^2 (t) \leq \alpha + \beta \, t^{2p}$ for all $n \in \bN$ 
and $t \in \bR,$ with the further consequence that
for every $n \in \bN$ and $\omega \in \Omega$ we have
\begin{equation}   \label{eqn:X9e}
\mathrm{Tr} 
\bigl( \, |f_n|^2 ( y( \omega ) ) \, \bigr)
\leq \alpha \, N + \beta \, \mathrm{Tr} \bigl( \,  y^{2p} ( \omega ) \, \bigr)
=: v( \omega ).
\end{equation}
Returning to the integrand on the right-hand side of (\ref{eqn:X9d}),
we then invoke the Cauchy-Schwarz inequality to infer that
\begin{equation}   \label{eqn:X9f}
\vline \ \mathrm{Tr} 
\bigl( \, x( \omega ) \, f_n ( y( \omega ) ) \, \bigr) \ \vline
\leq \sqrt{ u( \omega ) \, v ( \omega ) }, \ \ \forall \, n \in \bN
\mbox{ and } \omega \in \Omega,
\end{equation}
where $v( \omega )$ is defined in (\ref{eqn:X9e}) and we put 
$u( \omega ) := \mathrm{Tr} \bigl( \,  x^2 ( \omega ) \, \bigr)$.
We know (from how the algebra $\cA$ was defined) that 
$\mathrm{Tr} (x^2), \mathrm{Tr} (y^{2p}) \in L^{\infty -} ( \Omega )$, 
and this immediately implies that 
$\sqrt{u \, v }$ is in $L^{\infty -} ( \Omega )$ as well.  Thus 
$\sqrt{u \, v }$ is an integrable non-negative random variable which 
(in view of (\ref{eqn:X9f})) can be used as our dominating function.
\hfill $\square$

\end{example} 

\vspace{0.5cm}

\section{Overlap measure and overlap function}   \label{section:Y}

In this section we review the construction and some basic properties of a 
probability measure $\muovxy$ on $\bR^2$, which is associated to a pair of 
selfadjoint elements $x,y$ in a Borel-ncps $( \cA , \varphi ; \Gamma )$.  
In the special case when our Borel-ncps is a $W^{*}$-probability space, this 
construction has been known for a long time, and appears for instance in 
Section I.1 of \cite{Connes}.
%
% Our interest in $\muovxy$ stems from the fact 
% that it can be used in computations of conditional expectations of the
% form ``$E( f(x) \mid y )$'' or ``$E( g(y) \mid x )$''. 
%
For the reader's reassurance that the construction extends to the Borel-ncps 
setting, we provide its details in the proof of Proposition \ref{prop:Y1} below.

The terms ``overlap measure'' and ``overlap function'' used
in the present paper are inspired from the case when $\cA = \cM_N ( \bC )$ and 
$\varphi = \frac{1}{N} \mathrm{Tr}$.  
In that case, the description of $\muovxy$ boils down to looking at inner products 
between the eigenvectors of $x$ and those of $y$ (cf.~Example \ref{example:Y5} below), 
and such inner products are known as ``overlaps'' in the physics and in the random matrix 
literature.  

\vspace{6pt}

\begin{proposition-and-definition}  \label{prop:Y1}
Let $(\mathcal{A},\varphi ; \Gamma )$ be a Borel-ncps,
and let $x,y \in \cA^{\mathrm{sa}}$.
There exists a probability measure $\mu$ on the Borel sigma-algebra
of $\bR^2$, uniquely determined, such that		
\begin{equation}  \label{eqn:Y1a}
\int_{\bR^2}f(s)g(t) d \mu (s,t) =  \varphi(f(x)g(y)),
\ \ \forall \, f,g \in \Bor_b ( \bR ),
\end{equation}
where the elements $f(x), g(y) \in \cA$ that appear in (\ref{eqn:Y1a}) are 
obtained by performing Borel functional calculus on $x$ and on $y$, respectively.
		
\vspace{6pt}
		
\noindent
We will refer to this probability measure $\mu$ as the \emph{overlap measure} 
of $x$ and $y$, and we will denote it as $\muovxy$.
\end{proposition-and-definition}

\begin{proof}
{\em Construction of $\mu$.}  We consider the families of projections
\[
P(s) := \mathbf{1}_{]-\infty;s]}(x),  \ Q(t) := \mathbf{1}_{]-\infty;t]}(y),
\ \ s,t \in \bR,
\]
and we define $F : \mathbb{R}^2 \to \bR$ by
$F(s,t) = \varphi ( \, P(s)Q(t) \, )$, for $s,t \in \bR$.
We will verify that $F$ has the properties required from a cumulative distribution 
function on $\mathbb{R}^2$.

We start by observing that, due to the trace property of $\varphi$, we can write 
\begin{equation}   \label{eqn:Y1b}
			F(s,t) = \varphi \bigl( \, P(s) \, Q(t) \, P(s) \, \bigr),
% \ \mbox{ or } \ F(s,t) =	\varphi \bigl( \, Q(t) \, P(s) \, Q(t) \, \bigr),
\ \ s,t \in \bR,
\end{equation}
with $P(s) \, Q(t) \, P(s) \in \cA^{+}$.  For any fixed $s_o \in \bR$, we 
can invoke (\ref{eqn:X4b}) and then (\ref{eqn:X1a}) in order to find that
\[
\zeroA \leq P(s_o) \, Q(t_1) \, P(s_o) \leq P(s_o) \, Q(t_2) \, P(s_o) 
\leq P(s_o) \, \oneA \, P(s_o) = P(s_o),
\ \ \forall \, t_1 \leq t_2 \in \bR;
\]
this, in turn, implies that the function $\bR \ni t \mapsto F(s_o,t)$ 
is non-decreasing, with values in $[0, \varphi (  P(s_o) )].$
A standard calculation based on (\ref{eqn:X4c}) then shows that the latter
function is continuous from the right and has
$\lim_{t \to - \infty} F(s_o,t) = 0$,
$\lim_{t \to \infty} F(s_o,t) = \varphi (  P(s_o)  ).$

Likewise, one sees that for fixed $t_o \in \bR$ the function
$\bR \ni s \mapsto F(s,t_o)$ is non-decreasing and continuous
from the right, with
$\lim_{s \to - \infty} F(s,t_o) = 0$,
$\lim_{s \to \infty} F(s,t_o) = \varphi (  Q(t_o)  ).$
Arguments similar to the above show, moreover, that $F$ has correct limits 
(of $0$ or $1$) when both $s, t \to \pm \infty$. 
		
In order to argue that $F$ is a $2$-dimensional cumulative function, we are left 
to verify that for every $s_1 < s_2$ and $t_1 < t_2$ in $\bR$ we have
		\begin{equation}   \label{eqn:Y1d}
			F(s_2,t_2) - F(s_2,t_1) - F(s_1,t_2) + F(s_1,t_1) \geq 0.
		\end{equation}
The reader should have no difficulty to perform the bit of linear algebra 
%
% (combined with observations like the one that 
% $P(s_2) - P(s_1) = \mathbf{1}_{]s_1, s_2]} (x)$, for $s_1 < s_2$) 
% 
which rewrites the quantity from (\ref{eqn:Y1d}) in the form
		$\varphi \bigl( \,  (P(s_2) - P(s_1)) \,
		(Q(t_2) - Q(t_1)) \, (P(s_2) - P(s_1) \, \bigr)$; the latter
		quantity is indeed sure to be $\geq 0$, 
		since $\varphi$ is applied there to an element from $\cA^{+}$.
		
		\vspace{6pt}
		
		As a consequence of the above, we infer the existence of 
		a probability measure $\mu$ on the Borel sigma-algebra of 
		$\bR^2$, such that
		$\mu \bigl( \, ] - \infty , s] \, \times  \, ] - \infty, t] \, \bigr)
		= F(s,t), \ \ \forall \, s,t \in \bR.$
		
		\vspace{10pt}
		
		{\em Verification that (\ref{eqn:Y1a}) holds.} The 
		measure $\mu$ constructed above has the property that
		\begin{equation}   \label{eqn:Y1f}
			\mu \bigl( \, ] s_1, s_2] \, \times \, ] t_1, t_2] \, \bigr)
			= \varphi \bigl( \, \mathbf{1}_{ ] s_1, s_2] } (x)
			\, \mathbf{1}_{ ] t_1, t_2] } (y) \, \bigr),
			\ \ \forall \, s_1 < s_2 \mbox{ and $t_1 < t_2$ in $\bR$;}
		\end{equation}
		this holds because, as immediately verified, both sides of (\ref{eqn:Y1f}) 
		are equal to the algebraic expression considered in (\ref{eqn:Y1d}).
		
		In order to upgrade (\ref{eqn:Y1f}) to the formula indicated in (\ref{eqn:Y1a}),
		we use an intermediate step, as follows: for every $s_1 < s_2$ in $\bR$ we put 
		\[
		\cV_{s_1,s_2} := \{ g \in \mathrm{Bor}_b ( \bR ) \mid 
		\mbox{ (\ref{eqn:Y1a}) holds for } \mathbf{1}_{ ]s_1, s_2] }
		\mbox{ and } g \},
		\]
		and (by invoking appropriate features of the Borel functional calculus 
		for the element $y$) we observe that
		$\cV_{s_1, s_2}$ is a linear subspace of $\mathrm{Bor}_b ( \bR )$ which 
		is closed under the pointwise convergence of uniformly bounded 
		sequences of functions.  
		But let us also observe that, as a consequence of (\ref{eqn:Y1f}), the space 
		$\cV_{s_1, s_2}$ contains the indicator functions of all the half-open intervals 
		$] t_1, t_2]$ with $t_1 < t_2$ in $\bR$.  Together with stability
		under linear combinations and under uniformly bounded pointwise convergence, 
		this implies that $\cV_{s_1, s_2} = \mathrm{Bor}_b ( \bR )$.
		
		Now, for every $g \in \mathrm{Bor}_b ( \bR )$ let us put 
		$\cW_g := \{ f \in \mathrm{Bor}_b ( \bR ) \mid 
		\mbox{(\ref{eqn:Y1a}) holds for $f$ and $g$} \}.$
		The intermediate step taken in the preceding paragraph assures us that
		$\cW_g$ contains the indicator functions of all the half-open intervals 
		$] s_1, s_2]$ with $s_1 < s_2$ in $\bR$.  On the other hand, $\cW_g$
		is found (by invoking appropriate features of the Borel functional calculus 
		for the element $x$) to be a linear subspace of $\mathrm{Bor}_b ( \bR )$ 
		which is closed under the pointwise convergence of uniformly bounded 
		sequences of functions.  Putting together all these properties forces the 
		conclusion that $\cW_g = \mathrm{Bor}_b ( \bR )$.
		
		Finally: the equality $\cW_g = \Bor_b ( \bR )$, holding for all
		$g \in \Bor_b ( \bR )$, amounts precisely to the fact that the required 
		Equation (\ref{eqn:Y1a}) holds for all $f,g \in \mathrm{Bor}_b ( \bR )$.
		
		\vspace{10pt}
		
		{\em Uniqueness of $\mu$.}  Let $\nu$ be a probability measure on $\bR^2$ which has 
		the same property as described for $\mu$ in (\ref{eqn:Y1a}).  For any two Borel sets 
		$A,B \subseteq \bR$ we then find that 
		\[
		\nu (A \times B) = \varphi \bigl( \, \mathbf{1}_A (x) \, \mathbf{1}_B (y) \, \bigr)
		= \mu (A \times B).
		\]
		The measures $\mu$ and $\nu$ thus agree on the collection of sets 
		$\{ A \times B \mid A,B \mbox{ Borel subsets of } \bR \}$.  Since the latter 
		collection is a $\pi$-system which generates the Borel sigma-algebra of $\bR^2$, 
		we conclude that $\nu = \mu$.
\end{proof}

\vspace{10pt}

\begin{remark}    \label{rem:Y2}
$1^o$ In the preceding proposition it was convenient to only state 
Equation (\ref{eqn:Y1a}) for $f,g \in \Bor_b ( \bR )$, as that does not 
require any discussion concerning the integrability of the function 
$(s,t) \mapsto f(s) g(t)$ on the left-hand side.  We leave it as an exercise 
to the reader to check that, due to how the framework of a Borel-ncps was 
set in Section 2, the equality stated in (\ref{eqn:Y1a}) continues to make 
sense and to hold true for functions $f,g \in \Borpol ( \bR )$ -- that is, 
the function $(s,t) \mapsto f(s) g(t)$ is $\muovxy$-integrable, and its integral
against $\muovxy$ is equal to $\varphi \bigl( \, f(x) \, g(y) \, \bigr)$.

\vspace{6pt}

\noindent
$2^o$ When in Equation (\ref{eqn:Y1a}) we set the function $g$ to be identically 
equal to $1$, we find that the first one-dimensional marginal of $\muovxy$ (that is,
the Borel probability measure on $\bR$ defined by $A \mapsto \muovxy ( A \times \bR )$)
coincides with the distribution $\mu_x$ of the element $x$ with respect to $\varphi$, 
as defined in Section 2 (cf.~Remark and Notation \ref{rem:X5}). 
A similar argument shows that the second marginal of $\muovxy$ is equal to $\mu_y$,
the distribution of $y$ with respect to $\varphi$.
\end{remark}

%	Since in a number of examples of interest it will turn out 
%	that $\muovxy \ll \mu_x \times \mu_y$, we make a definition which 
%	gives a name to the corresponding Radon-Nikodym derivative.
		
\vspace{10pt}
		
\begin{definition}  \label{def:Y3}
Consider the same framework and notation as above and suppose that the 
overlap measure $\muovxy$ is absolutely continuous with respect to the 
direct product of its marginals $\mu_x$ and $\mu_y$.  The Radon-Nikodym derivative
\begin{equation}   \label{eqn:Y3a}   
\ovxy := \frac{d \muovxy}{d ( \mu_x \times \mu_y )}
\end{equation}
will be referred to as the {\em overlap function} of $x$ and $y$.
%	
% We will view $\ovxy$ as a Borel measurable function from 
% $\Spec(x) \times \Spec(y)$ to $[0, \infty )$; note, however, that 
% $\ovxy$ is only defined in the $(\mu_x \times \mu_y)$-a.e.~sense.
%
\end{definition}

\vspace{10pt}
	
\begin{example}   \label{example:Y4}  
In the setting of Proposition and Definition \ref{prop:Y1}, consider the 
		situation where $x$ and $y$ are independent, either in the 
		classical (commutative) sense, or in the sense of free probability.  This 
		assumption has the consequence that one can factor
		\[
		\varphi \bigl( \, f(x) \, g(y) \, \bigr) =
		\varphi \bigl( \, f(x) \, \bigr) \cdot
		\varphi \bigl( \, g(y) \, \bigr),
		\ \ \forall \, f,g \in \Bor_b ( \bR ).
		\]
		The latter equality can be read as saying that
		\[
		\int_{\bR^2} f(s) \, g(t) \, d \muovxy (s,t) =
		\int_{\bR} f(s) \, d \mu_x (s) \cdot
		\int_{\bR} g(t) \, d \mu_y (t),
		\]
which entails that the overlap measure $\muovxy$ is equal to the direct
product $\mu_x \times \mu_y$. This is hence a case when the overlap function 
$\ovxy$ is defined, and is identically equal to $1$.   
	%	
	%	\vspace{6pt}
	%	
	%	\noindent
	%	$2^o$ For contrast, consider the case where, in Proposition and Definition 
	%	\ref{prop:31}, we have $x = y$.  In this case, the overlap measure $\muovxy$ 
	%	is concentrated on the line $\{ (s,s) \mid s \in \bR \}$.  Indeed, the 
	%	complement $\bR^2 \setminus \{ (s,s) \mid s \in \bR \}$ can be written as a
	%	union $\cup_{n=1}^{\infty} I_n \times J_n$ where, for every $n \in \bN$, 
	%	$I_n$ and $J_n$ are disjoint intervals in $\bR$. We then observe that, for
	%	every $n \in \bN$:
	%	\[
	%	\muovxy (I_n \times J_n) = \varphi \bigl( \, \mathbf{1}_{I_n} (x)
	%	\, \mathbf{1}_{J_n} (x) \, \bigr)
	%	= \varphi \bigl( \, \mathbf{1}_{I_n \cap J_n} (x) \, \bigr) = 0.
	%	\]
	%	This entails that 
	%	$\muovxy \bigl( \, \bR^2 \setminus \{ (s,s) \mid s \in \bR \} \, \bigr) = 0$
	%	as well.
	%	
	%	A bit of further calculation (left to the reader) shows that, in this case, 
	%	$\muovxy$ is in fact just a copy of the distribution $\mu_x$ which is
	%	positioned along the diagonal of the square $\Spec(x) \times \Spec(x)$.
        %
\end{example}
	
\vspace{10pt}

\begin{example}   \label{example:Y5}   
Consider the situation where our Borel-ncps $( \cA , \varphi )$ has 
$\cA = \cM_N ( \bC )$, with $\varphi = \mathrm{tr}_N$ (normalized trace),
and where the elements of $\cA^{\mathrm{sa}}$ are thus $N \times N$ 
Hermitian matrices.  Let us pick some $x,y \in \cA^{\mathrm{sa}}$, and 
let $\lambda_1, \ldots , \lambda_N$ and $\rho_1, \ldots , \rho_N$ be the 
eigenvalues of $x$ and respectively of $y$, counted with multiplicities.        
Moreover, let us consider two orthonormal bases $u_1, \ldots , u_N$ and 
$v_1, \ldots , v_N$ for $\bC^N$, such that 
$x (u_k) = \lambda_k u_k$ and $y (v_k) = \rho_k v_k$ for
$1 \leq k \leq N$.  Some elementary linear algebra calculations show that,
in this situation, the overlap measure of $x$ and $y$ is
\[
\muovxy = \sum_{k, \ell = 1}^N 
\frac{ \mid \langle u_k \, , \, v_{\ell} \rangle \mid^2 }{N} 
\, \delta_{( \lambda_k, \rho_{\ell} )}
\]
(where $\delta_{( \lambda, \rho )}$ stands for  the Dirac mass concentrated 
at the point $(\lambda , \rho ) \in \bR^2$).

We note that in this example we have
\[
\mu_x \times \mu_y = \Bigl( \, \frac{1}{N} \sum_{k=1}^N \delta_{\lambda_k} \, \Bigr)
\times \Bigl( \, \frac{1}{N} \sum_{k=1}^N \delta_{\rho_k} \, \Bigr) 
= \frac{1}{N^2} \sum_{k, \ell =1}^N \delta_{ ( \lambda_k, \rho_{\ell} ) },
\]
which makes it clear that $\muovxy \ll \mu_x \times \mu_y$ (indeed, absolute 
continuity simply amounts here to the fact that the set of atoms of 
$\muovxy$ is contained in the set of atoms of $\mu_x \times \mu_y$).  Given 
eigenvalues $\lambda$ of $x$ and $\rho$ of $y$, the value 
at $( \lambda , \rho )$ of the overlap function $\ovxy$ is the ratio
$\muovxy \bigl( \, \{ \, ( \lambda, \rho ) \, \} \, \bigr)/\Bigl( 
\mu_x ( \, \{\lambda\} \, ) \cdot \mu_y ( \, \{ \rho \} \, ) \Bigr).$  
An easy application of Parseval's formula shows that the latter quantity can 
be also written as
\[
\ovxy ( \lambda , \rho ) = N \, \frac{\mathrm{Tr} (p \, q)}{\mathrm{Tr}(p) \cdot 
\mathrm{Tr} (q) },
\]
where $p$ is the projection onto the eigenspace of $x$ corresponding to $\lambda$,
and $q$ is the projection onto the eigenspace of $y$ corresponding to $\rho$.
\end{example}
    
\vspace{0.5cm}

\section{Conditional expectations, via disintegration of the overlap measure}
\label{section:Z}

In this section we fix a $W^{*}$-probability space
$( \cA , \varphi )$, as considered in Example \ref{example:X7}, and we discuss 
some additional structure that appears in this setting.

\vspace{10pt}

\begin{remark-and-notation}   \label{rem:Z1}
% {\em (Disintegration of $\muovxy$.)}
%
% \noindent
As mentioned in Example \ref{example:X7},
the spectrum $\Spec (x)$ of an $x \in \cA^{\mathrm{sa}}$ is a non-empty 
compact subset of $\bR$, and the support of the distribution $\mu_x$ is precisely 
equal to $\Spec (x)$. 

Let us now simultaneously consider two elements $x,y \in \cA^{\mathrm{sa}}$, and
their overlap measure $\muovxy$.  The observation recorded in the preceding
paragraph entails that $\muovxy$ has compact support, contained in 
$\Spec (x) \times \Spec (y)$ (the Cartesian product of the supports of the two 
marginals $\mu_x$ and $\mu_y$ of $\muovxy$).  In connection to that, an important 
piece of structure which will be consistently used in what follows is that 
$\muovxy$ can be {\em disintegrated} with respect to either of its marginals $\mu_x$
or $\mu_y$ (see e.g.~\cite[Chapter III, pages 78-79]{DellacherieMeyer}). 
For instance, the disintegration with
respect to $\mu_y$ comes in the guise of a family
		$\bigl\{ k_t^{(2)} \mid t \in \Spec(y) \bigr\}$, where every 
		$k_t^{(2)}$ is a probability measure on the Borel sigma-algebra of $\Spec(x)$,
		such that the following things happen.
		First, for every bounded Borel measurable function $f : \Spec(x) \to \bC$, 
		the ``fiber-wise integration of $f$'',
\begin{equation}   \label{eqn:Z1a}
			t \mapsto \int_{\Spec(x)} f(s) \, d k_t^{(2)} (s) \mbox{ for } t \in \Spec(y),
\end{equation}
		is a bounded Borel measurable function on $\Spec(y)$.
		Then, for every pair of bounded Borel measurable functions 
		$f : \Spec(x) \to \bC$ and $g : \Spec(y) \to \bC$, one has that
		\begin{equation}   \label{eqn:Z1b}
			\int_{\Spec(y)} \Bigl[ \, \int_{\Spec(x)} f(s) \, d k_t^{(2)} (s) \, \Bigr]
			\, g(t) \, d \mu_y (t) 
			= \int_{ \Spec(x) \times \Spec(y) } f(s)g(t) \, d \muovxy (s,t).
\end{equation}
		
		\vspace{6pt}
		
		\noindent
		Likewise, disintegration of $\muovxy$ with respect to $\mu_x$ creates a
		family $\bigl\{ k_s^{(1)} \mid s \in \Spec(x) \bigr\}$, where every $k_s^{(1)}$ 
		is a probability measure on the Borel sigma-algebra of $\Spec(y)$, and where 
		formulas symmetric to (\ref{eqn:Z1a}) and (\ref{eqn:Z1b}) are holding.
\end{remark-and-notation} 
        
\vspace{10pt}

Another relevant piece of structure that we are sure to have for the 
$( \cA , \varphi )$ of this section concerns conditional expectations.
    
\vspace{6pt}

\begin{remark-and-notation}   \label{rem:Z2}
% {\em (Non-commutative conditional expectation.)}
%
% \noindent
For every von Neumann subalgebra $\cB \subseteq \cA$ one has a 
$\varphi$-preserving {\em conditional expectation}
onto $\cB$, which will be denoted as $E( \cdot | \cB )$; this is the 
linear map from $\cA$ to $\cB$ uniquely determined by the requirement that 
        \begin{equation}  \label{eqn:Z2a}
        \varphi \bigl( \, E(a |\mathcal{B}) \, b \, \bigr) = \varphi (ab), 
        \ \ \forall \, a\in \cA \mbox{ and } b\in\cB.
        \end{equation}
From (\ref{eqn:Z2a}) it is immediate that the map $E( \cdot | \cB )$ is 
a projection (that is, $E(b | \cB ) = b$ for $b \in \cB$), and indeed has the 
``conditional expectation property'' (a.k.a.~$\cB$-bimodule property) that 
$E(b_1 a b_2 | \cB ) = b_1 \, E(a | \cB ) \, b_2$ for all $a \in \cA$
and $b_1, b_2 \in \cB$.  $E( \cdot | \cB )$ is, moreover, a normal, 
positive and faithful map (see e.g.~\cite[Section V.2]{Takesaki}); its positivity 
        also implies, in particular, that $E(a| \cB)$ is selfadjoint in $\cB$ whenever 
        $a$ is selfadjoint in $\cA$.  For the present paper, a relevant 
        property of $E( \cdot | \cB )$ (easily derived from (\ref{eqn:Z2a})) is 
        that it serves as orthogonal projection in the $L^2$-space associated to
        the trace $\varphi$; thus, for every $a \in \cA$ one has
        \begin{equation}  \label{eqn:Z2b}
        || \, a - E(a | \cB ) \, ||_2 \leq || a-b ||_2, 
        \ \ \forall \, b \in \cB, 
        \mbox{ with equality if and only if } b = E(a | \cB ),
        \end{equation}
	where $|| \cdot ||_2$ is the $L^2$-norm associated to $\varphi$.

       A special case of the above: given an 
       $y \in \cA^{\mathrm{sa}}$, we will use the notation $\bE(\cdot | y)$ 
       to refer to the conditional expectation onto the von Neumann 
       subalgebra $\cB \subseteq \cA$ generated by $y$.  This subalgebra can 
       be explicitly described by using functional calculus: 
       $\cB = \bigl\{ g(y) \mid g \in \Bor_b ( \, \Spec (y) \, ) \bigr\},$
       where $\Bor_b ( \, \Spec (y) \, )$ is the space of bounded Borel 
       functions from $\Spec (y)$ to $\bC$.  For $x \in \cA^{\mathrm{sa}}$, a 
       convenient way of describing $\bE(x | y)$ is thus in the form 
       $E(x | y ) = h(y)$, 
       where $h : \Spec (y) \to \bR$ is a bounded Borel function such that
       \begin{align}\label{eqn:Z2c}
	    \varphi(x \, g(y))=\varphi \bigl( \, h(y)\, g(y) \, \bigr), 
            \ \ \forall \, g \in \Bor_b ( \, \Spec (y) \, ).
	\end{align} 
        Note that, in view of (\ref{eqn:Z2b}), $h(y)$ is 
        the unique element of the von Neumann algebra generated by $y$ 
	which is at minimal distance from $x$, in the $|| \cdot ||_2$-norm 
        associated to $\varphi$.
\end{remark-and-notation} 

\vspace{10pt}
	
Let us now clarify how the disintegration of the overlap measure
appears in calculations of conditional expectations.
	
\vspace{6pt}
	
\begin{proposition}   \label{prop:Z3} 
Let $x,y \in \cA^{\mathrm{sa}}$ and let $f : \Spec(x) \to \bR$ be a bounded Borel 
function.  Consider the element $f(x) \in \cA^{\mathrm{sa}}$ obtained by functional 
calculus, and then consider the conditional expectation $E( f(x) \, | \, y )$.
On the other hand, let $\muovxy$ be the overlap measure of $x$ and $y$, and consider
the disintegration $( k_t^{(2)} )_{t \in \Spec (y)}$ of $\muovxy$ that was mentioned
in Remark and Notation \ref{rem:Z1}.  One has that
\begin{equation}   \label{eqn:Z3a}
E( f(x) \mid y ) = h(y) \ \mbox{ (functional calculus of $y$)},
\end{equation}
where $h : \Spec(y) \to \bR$ is obtained by fiber-wise integration as in (\ref{eqn:Z1a}),
\begin{equation}   \label{eqn:Z3b}  
			h(t) := \int_{\Spec(x)} f (s) \, d k_t^{(2)} (s), 
			\mbox{ for } t \in \Spec(y).
\end{equation}
\end{proposition}
	
\begin{proof} Let $\cB = \bigl\{ g(y) \mid g \in \Bor_b ( \Spec (y) ) \bigr\}$ be the 
von Neumann subalgebra of $\cA$ generated by $y$.  We have $h(y) \in \cB$, and in order 
to establish (\ref{eqn:Z3a}) we need to check that 
\begin{equation}   \label{eqn:Z3c}
\varphi ( f(x) g(y) ) = \varphi ( h(y) g(y) ), \ \ \forall \, g \in \Bor_b ( \Spec (y) ).
\end{equation}
And indeed, by starting from the left-hand side of (\ref{eqn:Z3c}), we can write:
		\begin{align*}
			\varphi ( f(x) g(y) )
			& = \int_{\Spec(x) \times \Spec(y)} f(s) g(t) \, d \muovxy (s,t) 
			\ \mbox{ (by (\ref{eqn:Y1a})) }  \\
			& = \int_{\Spec(y)} \Bigl[ \, \int_{\Spec(x)} f(s) \, d k_t^{(2)} (s) \, \Bigr]
			\, g(t) \, d \mu_y (t)   \ \mbox{ (by (\ref{eqn:Z1b})) }   \\ 
			& = \int_{ \Spec(y)} h(t) g(t) \, d \mu_y (t) 
			\ \mbox{ (by how $h$ is defined)}   \\
			& = \varphi (h(y) g(y)) \ \mbox{ (by how $\mu_y$ is defined).} 
		\end{align*}
		The latter quantity is the right-hand side of (\ref{eqn:Z3c}), as required.
	\end{proof}
	
	\vspace{10pt}
	
	\begin{remark}   \label{rem:Z4}
		In the setting of the preceding proposition, the function 
		$h : \Spec(y) \to \bC$ defined in Equation (\ref{eqn:Z3b}) has in particular 
		the property that
		\begin{equation}   \label{eqn:Z4a}
			\int_{\Spec (y)} h (t) \, d \mu_y (t) = \int_{\Spec (x)} f(s) \, d \mu_x (s).
		\end{equation}
		Indeed, both sides of (\ref{eqn:Z4a}) are found to be equal to 
		$\varphi ( \, f(x) \, )$, where on the left-hand side we start from 
		$\varphi ( \, h(y) \, )$, re-write it as $\varphi ( \, E(f(x) \mid y) \, )$,
		and use the fact that 
		$\varphi \circ E( \cdot | y ) = \varphi$. 
        
We record here the following consequence of (\ref{eqn:Z4a}), which will be used 
in Section 6 below:
		\begin{equation}   \label{eqn:Z4b}
			\left\{ \begin{array}{l}
				\mbox{Let $V \subseteq \Spec (x)$ be a Borel set with $\mu_x (V) = 1$.}   \\ 
				\mbox{There exists a Borel set $N \subseteq \Spec (y)$, with $\mu_y (N) = 0$,} \\
					\mbox{such that $k_t^{(2)} (V) = 1$ for every $t \in \Spec (y) \setminus N$.}
				\end{array}  \right.
			\end{equation}
			For the verification of (\ref{eqn:Z4b}) we let $f : \Spec (x) \to \bR$
			be the indicator function of $\Spec (x) \setminus V$, which forces the 
			quantities in (\ref{eqn:Z4a}) to be equal to $0$, and where the function 
			$h$ on the left-hand side has taken the form
			$h(t) = 1 - k_t^{(2)} (V), \, t \in \Spec(y).$
			Since $h$ is non-negative with $\int h \, d \mu_y = 0$, we get the 
			conclusion stated in (\ref{eqn:Z4b}).
		\end{remark}

\vspace{0.5cm}

\section{Review of some subordination results in free probability}
\label{SubordReview}
   
We will consider several analytic transforms which characterize probability 
measures on the real line, as follows.
    
\vspace{10pt}
	
\begin{notation}   \label{def:21}
Let $\mu$ be a probability measure on the Borel sigma-algebra of $\bR$.
We will use the customary notation $G_{\mu}$ for the Cauchy transform of
$\mu$, defined by
		\begin{equation*}
			G_{\mu} (z) := \int_{\bR} \frac{1}{z-t}d\mu(t),
            %	=\varphi\left( \, (z-a)^{-1} \, \right)
            \ \  z \in \bC \setminus \mathrm{supp} ( \mu )
		\end{equation*}
(where $\mathrm{supp} ( \mu ) \subseteq \bR$ is the support of $\mu$).
Then for $z \in \bC \setminus \{ 0 \}$ such 
that $1/z \not\in \mathrm{supp} ( \mu )$ we define
		\begin{equation*}     
			\psi_{\mu} (z) := \int_{\bR}\frac{zt}{1-zt}d\mu(t) 
		%	= \varphi \Bigl( \, za \, \bigl( 1 - za \bigr)^{-1} \, \Bigr),
			\mbox{ and }
			\eta_{\mu} (z) := \frac{\psi_{\mu} (z)}{1+\psi_{\mu} (z)}.
		\end{equation*}
We will refer to the functions $\psi_{\mu}$ and $\eta_{\mu}$ by calling 
them {\em moment transform} and respectively {\em Boolean-cumulant transform} 
of $\mu$.
        %		
	%	\vspace{6pt}
	%	
	%	\noindent
	%	$3^o$ Since the transforms introduced in $1^o$ and $2^o$ above only depend on $\mu$,
	%	we will also write $G_\mu, \psi_\mu$ and $\eta_\mu$ instead of $G_a,\psi_a$ and $\eta_a$,
	%	respectively. Note that the functions $G_{\mu}$ and $\psi_{\mu}$ are closely related,
	%	we have $G_\mu(z)=\tfrac{1}{z}(\psi_\mu(\tfrac{1}{z})+1)$.
        %		
		
We will also make occasional use of the {\em Hilbert transform} of $\mu$,
which is defined as
		\begin{align*}
			H_\mu(t)=\frac{1}{\pi}\lim_{\varepsilon\to 0} Re (G_{\mu}(t+i\varepsilon)),
		\end{align*}
with the latter limit known to exist for Lebesgue almost every $t \in \bR$
(see e.g.~\cite[Chapter X.3]{Torchinsky}). 
	\end{notation}
	
	\vspace{10pt}
	
	The next lemma records some basic properties of the Cauchy transform.
	
	\vspace{6pt}
	
	\begin{lemma}  \label{lem:22}
        {\em (Cf.~\cite[Lemma 2.17]{Belinschi2}, and the proof references indicated there.)}

        \noindent
	Let $\mu$ be a Borel probability measure on $\bR$, and denote by $\mu^s$,
	and $\mu^{ac}$ the singular and absolutely continuous (with respect to Lebesgue measure) 
        parts of $\mu$.
		
	\vspace{6pt}
		
	\noindent    
	$1^o$ For $\mu^s$-almost all $x\in\bR$, the non-tangential limit of the 
	Cauchy transform $G_\mu$ at $x$ is infinite : $\sphericalangle\lim_{z\to x}G_\mu(z)=\infty$.
		
	\vspace{6pt}
		
	\noindent
	$2^o$ For every $x\in\bR$, we have $\mu(\{x\})=\sphericalangle\lim_{z\to x}(z-x) G_\mu(z)$.
		
	\vspace{6pt}
		
	\noindent
	$3^o$ Denote by $f$ the density of $\mu^{ac}$ w.r.t. the Lebesgue measure,
	then for Lebesgue--almost all $x\in\bR$, 
	$f(x)=-\frac{1}{\pi}\sphericalangle\lim_{z\to x} \operatorname{Im} (G_\mu(z))$.
	\hfill $\square$
	\end{lemma}
	
\vspace{10pt}
	
\subsection{Review of subordination results -- additive case.} \label{subsec:subordination}
	
	$\ $
	
	\noindent
	Let $(\cA,\varphi)$ be an ncps and consider two freely independent selfadjoint elements  
	$a,b \in \cA$.
	We denote the distributions of $a$ and $b$ with respect to $\varphi$ by $\mu$ and $\nu$,
	respectively; note that, due to the free independence of $a$ and $b$, the distribution 
	of $a + b$ is then the free additive convolution $\mu \boxplus \nu$.
	
	As proved in \cite{Biane}, there exists an analytic map $\omega : \bC^{+} \to \bC^{+}$,
	uniquely determined, with the following properties:
	(i) $G_{\mu\boxplus\nu}(z) = G_{\mu}(\omega(z))$ for all $z \in \bC^{+}$;  
	(ii) $\omega$ increases the imaginary part;
	(iii) $\omega(iy)/(iy) \to 1 \mbox{ as } y\to +\infty.$
	This $\omega$ is called the {\em subordination function} of $a+b$ with 
	respect to $a$; if we just want to refer to distributions, $\omega$ will be
	referred to as subordination function of $\mu \boxplus \nu$ with respect to $\mu$.
	
	For the $\omega$ considered above, it can be shown that: for every $s \in \bR$, the 
	assignment $z \mapsto 1/ ( \omega (z) - s )$ defines a function on $\bC^{+}$ which 
	can be identified as Cauchy transform for a Borel probability measure on $\bR$.  
	In the same vein as in the notation from \cite[Theorem 3.1]{Biane}, we denote 
	the latter measure as $k_s$; we thus have
	\begin{align} \label{eqn:subordAdditive}
		\int_{\bR}\frac{1}{z-t}k_s(dt)=\frac{1}{\omega(z)-s}, \ \ \forall
		\, s \in \bR \mbox{ and } z \in \bC^{+}.
	\end{align}
	Moreover, bundling together the measures $k_s$ turns out to create a Feller-Markov 
	kernel $\cK$, where for any bounded Borel function $g : \bR \to \bC$ we define the 
	new function $\cK g$ by putting
	\[
	\cK g(s)=\int_\bR g(t) k_s(dt), \ \ s \in \bR.
	\]
	A relevant point for the considerations of this paper is that the kernel $\cK$
	can be used for computing conditional expectations onto the $W^{*}$-subalgebra 
	$W^{*} (a) \subseteq \cA$ which is generated by $a$; that is, we have
	\[
	\bE(g(a+b)|a)=\cK g(a),
	\mbox{ for $g : \bR \to \bC$ bounded Borel function,}
	\]
	where $\bE ( \cdot | a )$
	is our notation for conditional expectation onto $W^{*} (a)$.
	
	\vspace{6pt}
	
	Finally, a notational comment:
	it is customary to also consider the subordination function of $a+b$ with respect 
	to $b$.  In the few places of the paper where we will do that, we will denote the 
	two subordination functions that arise as $\omega_1$ and $\omega_2$ (hence the 
	$\omega$ discussed above will become $\omega_1$).  In that setting we have 
	the relations
	\[
	G_{\mu\boxplus\nu}(z)=G_\mu(\omega_1(z))=G_\nu(\omega_2(z)),
	\ \ z \in \bC^{+};
	\]
	another useful formula which is known to hold in connection to 
	$\omega_1$ and $\omega_2$ says that
	\[
	\omega_1(z)+\omega_2(z)=\frac{1}{G_{\mu\boxplus\nu}(z)} + z, 
	\ \ z\in\bC^+.
	\]
	
	\vspace{5pt}
	
	\subsection{Review of subordination results -- multiplicative case.} 
	
	$\ $
	
	\noindent
	Subordination results similar to those reviewed in the preceding subsection 
	also hold true in relation to the operation of free multiplicative convolution 
	on $\bR_+=[0, \infty)$.  More precisely, let us look once again at an ncps 
    $(\cA,\varphi)$, where we now consider two {\em positive}
	elements $a,b \in \cA$, such that $a$ and $b$ are freely independent.
	We denote the distributions of $a$ and $b$ by $\mu$ and $\nu$, respectively, and we
	consider the free multiplicative convolution $\mu \boxtimes \nu$, which is the 
	distribution of the positive element $a^{1/2} b a^{1/2}$.  In connection to this,
	\cite[Theorem 3.6]{Biane} proves that there exists an analytic map 
	$\omega : \bC^{+} \to \bC^{+}$, uniquely determined by its properties that it increases
	the argument of every $z \in \bC^{+}$ and that it satisfies the equation
	\[
	\psi_{\mu\boxtimes\nu}(z)=\psi_{\mu}(\omega(z)), 
	\ \ z \in \bC^{+}, 
	\]
	where the ``$\psi$'' notation is used for
	moment transforms, as in Notation \ref{def:21}.
	This $\omega$ is called the {\em multiplicative subordination function} of 
	$\mu \boxtimes \nu$ with respect to $\mu$.  
	
	In the multiplicative setting we are once again interested in a family of Borel
	probability measures $k_s$, which are now defined on $\bR_{+}$, via the following 
        prescription on their moment transforms:
	\begin{equation}   \label{eqn:25a}
		\int_{\bR}\frac{zt}{1-zt}dk_s(t)
		=\frac{\omega(z)s}{1-\omega(z)s}, \ \ z \in \bC^{+}.
	\end{equation}
	The relevance of these measures $k_s$ is that
	for any bounded Borel function $g$ on $\bR_+$ we have
	\begin{equation*}
		\bE( \, g(a^{1/2}ba^{1/2}) \mid a \, ) \ = \ \cK g(a),
	\end{equation*}
	with the kernel $\cK$ defined by the formula
	\begin{equation*}
		\cK g(s)=\int_{\bR_+} g(t) k_s(dt). 
	\end{equation*}
	
	One considers also subordination with respect to $\nu$, if we do so we will
	use notation $\omega=\omega_1$ and we denote by $\omega_2$ the second 
	subordination function for which we have $\psi_{\mu\boxtimes\nu}=\psi_\nu(\omega_2(z))$. 
	One has $\omega_1(z)\omega_2(z)=z\eta_{\mu\boxtimes\nu}(z)$, where 
	$\eta_{\mu\boxtimes\nu}$ is the Boolean-cumulant transform function mentioned 
	in Notation \ref{def:21}.
	
\vspace{0.5cm}
		
\section{Overlap function for $x = a$ and $y = a+b$, where $a$ is free from $b$}
\label{AdditiveFormulas}
		
		In this section, the selfadjoint elements $x,y \in \cA$ from the preceding section 
		take the form $x = a$ and $y = a + b$, where $a,b \in \cA$
		are selfadjoint and freely independent.  We will denote the distributions of 
		$a$ and $b$ by $\mu$ and $\nu$, respectively; due to the freeness assumption 
            on $a$ and $b$, the distribution of $a+b$ is then the free additive 
		convolution $\mu\boxplus\nu$.  We will show that, in this setting, the overlap 
		measure $\muovadd$ is sure to be absolutely continuous with respect to 
		$\mu_a \times \mu_{a+b}$, and we will give an explicit formula for the 
		resulting overlap function, with the important benefit that this will enable 
            us to make explicit computations of conditional expectations of the 
		form $\bE(f(a)|a+b)$.
		
		\vspace{10pt}
		
		\begin{remark}\label{rem:41}
			The subordination results reviewed in Subsection \ref{subsec:subordination} 
			allow us to write, for any bounded Borel functions $f,g$
			\begin{align*}
				\varphi(f(a)g(a+b))
                & =\varphi\left(f(a)\bE(g(a+b)|a)\right)=\varphi(f(a)\mathcal{K}g(a))
                                            =\int_{\bR} f(s) \mathcal{K}g(s)d\mu(s) \\ 
                & =\int_{\bR} f(s) \int_{\bR}g(t)dk_s(t)d\mu(s),
			\end{align*}
            where the probability measures $k_s$ are described by Equation 
            (\ref{eqn:subordAdditive}).
		Hence for the overlap measure $\muovadd$ considered here, the first of the 
            two disintegrations considered in Remark and Notation \ref{rem:Z1} is simply 
            described as 
            \[
            \ks = k_s, \ \ s \in \Spec (a).
            \]
            Our goal in this section is that, by starting from the measures $\ks$, we find 
            the second disintegration $\kt$ (with $t$ running in $\Spec (a+b)$) for $\muovadd$.
            When working towards this goal we will rely on a number of known facts about 
            regularity of free additive convolution and subordination functions, which we 
            review next. 
            %
		% We rely on several known facts about regularity of free additive convolution 
		% and subordination functions, which we collect below. 
		\end{remark}
		
		\vspace{6pt}
		
		\begin{review}   \label{review:43}
			[Regularity of free additive convolution.]
			Let $\mu,\nu$ be two Borel probability measures on $\bR$, where neither of them 
			is a point mass.  We consider the free additive convolution 
			$\mu\boxplus\nu$, and let $\omega$ be the subordination function of 
			$\mu \boxplus \nu$ with respect to $\mu$.
			
			\vspace{6pt}
			
			\noindent
			$1^o$ $\mu\boxplus\nu$ has at most finitely many atoms 
			$\gamma_i,\, i=1,\ldots,n$. For every $1 \leq i \leq n$ 
			there exist $\alpha_i,\beta_i\in\bR$, uniquely determined, such that 
			$\alpha_i+\beta_i=\gamma_i$ and $\mu(\{\alpha_i\})+\nu(\{\beta_i\})>1$ 
			(see \cite{BercoviciVoiculescu}).
			
			\vspace{6pt}
			
			\noindent
			$2^o$ With notation as in $1^o$ we have $\lim_{z\to\gamma_i} \omega(z)=\alpha_i$, 
			for every $i=1,\ldots,n$ (see \cite{BercoviciVoiculescu}).
			
			\vspace{6pt}
			
			\noindent
			$3^o$ $\mu\boxplus\nu$ has no singular continuous part
			(see \cite{Belinschi2}, Theorem 4.1).
			
			\vspace{6pt}
			
			\noindent
			$4^o$ Let $(\mu\boxplus\nu)^{ac}$ denote the absolutely continuous part 
			(with respect to Lebesgue measure) of $\mu \boxplus \nu$, and let 
			$f_{\mu\boxplus\nu}$ denote the density of $(\mu\boxplus\nu)^{ac}$ with 
			respect to Lebesgue measure.  There exists an open set $U \subset \bR$, with 
			$(\mu\boxplus\nu)^{ac} ( \bR \setminus U) = 0$, such that $f_{\mu\boxplus\nu}$
			is strictly positive on $U$ (see \cite{Belinschi2}, Theorem 4.1).
			
			\vspace{6pt}
			
			\noindent
			$5^o$ Let $U$ be as in $4^o$ above. The subordination function $\omega$ 
			can be continued continuously to the real line, and the continuation 
			(still denoted $\omega$) has strictly positive imaginary part at every 
			point $x\in U$ (see \cite{Belinschi2}, Theorem 4.1).
		\end{review}
		
		\begin{proof}
			Except $5^o$, all statements above are explicitly stated in references. 
			Point $5^o$ requires a bit of explanation. Recall 
			(cf.~final paragraph of Section 2.4) that we denote $\omega=\omega_1$, where 
			$\omega_1$ and $\omega_2$ are the subordination functions of $\mu \boxplus \nu$ 
			with respect to $\mu$ and to $\nu$, respectively.
			As stated in the proof of Theorem 4.1 in \cite{Belinschi2}, at least one of the 
			subordination functions $\omega_1, \omega_2$ has strictly positive imaginary part
			at $x\in U$. Suppose that the continuation of $\omega_2$  has strictly positive
			imaginary part at $x\in\mathbb{R}$. We will argue that so does the continuation 
			of $\omega_1=\omega$. One has $\omega_1(z)+\omega_2(z)=1/G_{\mu\boxplus\nu}(z)+z=1/G_\mu(\omega_1(z))+z=1/G_\nu(\omega_2(z))+z$.  Then we have
			\[\omega_1(x)=1/G_\nu(\omega_2(x))-\omega_2(x)+x.\]
			Since for a measure $\nu$ not being a point mass we have strict inequality 
			$\operatorname{Im}( 1/G_\nu(z))>\operatorname{Im}(z)$ for 
			$z = \omega_2 (x) \in\bC^{+}$, we see that $\omega_1(x)$  
			also has strictly positive imaginary part. 
		\end{proof}

\vspace{10pt}

We now proceed with the plan announced above, of examining the measures $k_s$.

\vspace{6pt}
		
	\begin{proposition}\label{prop:44}
        Consider the setting indicated at the beginning of this section (with $x=a$ and $y=a+b$,
        where $a,b$ are freely independent), and let $\mu, \nu$ be the distributions of $a$ and $b$,
        respectively.  We assume that neither $\mu$ nor $\nu$ is a point mass and we consider,
        in connection to them, the various items of notation used in Review \ref{review:43} above. 
        Let the function $o: \Spec (a) \times \Spec (a+b) \to \mathbb{R}_+$ be defined 
        (in $\mu\times (\mu\boxplus \nu)$-almost everywhere sense) by the following formula:
			\begin{align}\label{eqn:44a}
				\forall s\in \mathrm{supp}(\mu), o(s,t)=
				\begin{cases}
					-\frac{1}{\pi} \frac{1}{f_{\mu\boxplus\nu}(t)}
                      \operatorname{Im}\left(\frac{1}{\omega(t)-s}\right)\qquad 
                                                               &\mbox{ if } t\in U,  \\
					\frac{1}{\mu(\{s\})} 1_{\omega(t)=s}\qquad 
                                        &\mbox{ if $t$ is an atom of $\mu\boxplus\nu$.} 
                   \end{cases}
			\end{align}
			Consider, on the other hand, the measure $k_s$ from \eqref{eqn:subordAdditive}. 
			Then: for $\mu$-almost every $s \in \bR$ we have $k_s\ll \mu\boxplus\nu$ and the 
                corresponding Radon-Nikodym derivative is given by the partial function $o(s, \cdot )$ 
                of the $o$ from (\ref{eqn:44a}).
		 \end{proposition}

            %
            % [Write-up from the preceding version of the paper.]
            %
		% The next proposition states the precise formula that one has for the overlap 
            % function in this case.
		% The proposition does not refer to some freely independent selfadjoint elements 
            % $a,b$ in an ncps -- but if such $a,b$ were in the picture, then the function $o$
            % defined in (\ref{eqn:44a}) would be precisely $o_{a,a+b}$.
		%
		% \begin{proposition}\label{prop:44}
		%	With notation from Review \ref{review:43} above, consider the function 
		%	
		%	\noindent
		%	$o:\mathrm{supp}(\mu)\times \mathrm{supp}(\mu\boxplus\nu)\to \mathbb{R}_+$ 
		%	defined (in $\mu\times \mu\boxplus \nu$-almost everywhere sense)
		%	by the following formula:
		%	\begin{align}\label{eqn:44a}
		%		\forall s\in \mathrm{supp}(\mu), o(s,t)=
		%		\begin{cases}
		%			-\frac{1}{\pi} \frac{1}{f_{\mu\boxplus\nu}(t)}
            %          \operatorname{Im}\left(\frac{1}{\omega(t)-s}\right)\qquad 
            %                                                   &\mbox{ if } t\in U,  \\
		%			\frac{1}{\mu(\{s\})} 1_{\omega(t)=s}\qquad 
            %                            &\mbox{ if $t$ is an atom of $\mu\boxplus\nu$.} 
		%		\end{cases}
		%	\end{align}
		%	Consider the measure $k_s$ from \eqref{eqn:subordAdditive}. 
		%	Then: for $\mu$-almost every $s \in \bR$ we have $k_s\ll \mu\boxplus\nu$ and the Radon-Nikodym 
		%	derivative is given by the partial function $o(s, \cdot )$ (sending $t \mapsto o(s,t)$).
		% \end{proposition}
		%
        
		\begin{proof}
		As reviewed in Remark \ref{review:43}, there exist a finite set $\Gamma$ 
            and an open set $U$, with 
			$\Gamma \cap U  = \emptyset$, such that: the support of the singular part of 
			$\mu\boxplus \nu$ is $\Gamma$, the density $f_{\mu\boxplus \nu}$ of the absolutely
			continuous part of $\mu\boxplus \nu$ is positive on $U$ and 
			$\mu\boxplus \nu(\Gamma\cup U)=1$. These results are obtained by showing that 
			the subordination map $\omega$ of $\mu\boxplus \nu$ with respect to $\mu$ has
			a continuous extension 
			$\omega:\mathbb{C}^+ \cup \mathbb{R} 
			\to \mathbb{C}^+ \cup \mathbb{R} \cup \{\infty\}$, 
			and that $\omega(\Gamma)$ is a set of atoms of $\mu$, while 
			$\omega(U) \subset \mathbb{C}^+$. 
			
			Upon invoking Remark \ref{rem:Z4} in connection to the set 
			$V=\Gamma\cup U$ 
			one gets that $k_s(\Gamma\cup U)=1$ for $\mu$-almost all $s\in \mathbb{R}$.  
			For the rest of the proof, let us fix an $s$ which has 
			$k_s ( \Gamma \cup U)=1$. 
			
			Since the Cauchy transform of $k_s$ takes finite 
			values at all points of $U$, it follows from Lemma \ref{lem:22} that:
			
			\vspace{6pt}
			
			$\bullet$ the singular part of $k_s$ is purely atomic and supported on $\Gamma$;
			
			\vspace{6pt}
			
			$\bullet$ a density of the absolutely continuous part of $k_s$ is 
			$-\pi^{-1}\operatorname{Im}[(\omega-s)^{-1}]\mathbf{1}_{U}$.
			
			\vspace{6pt}
			
			\noindent
			Since the singular part of $\mu\boxplus \nu$ is purely atomic with support $\Gamma$
			and since the density $f_{\mu\boxplus \nu}$ of the absolutely continuous part of 
			$\mu\boxplus \nu$ is positive on $U$, the above bulleted items allow
			us to conclude that $k_s$ is absolutely continuous with respect to $\mu\boxplus\nu$.
			
			In order to prove that the Radon-Nikodym derivative is indeed 
			given by \eqref{eqn:44a}, it remains to verify that 
			\begin{equation}   \label{eqn:43x}
				k_s( \{\gamma\} ) = \frac{\mu \boxplus \nu (\{\gamma\})}{ \mu(\{s\}) }
				1_{s=\omega(\gamma)},\quad \gamma \in \Gamma.
			\end{equation}
			Towards this verification, we consider two cases.
			
			\vspace{6pt}
			
			\noindent
			-- When $s\neq \omega(\gamma)$, (\ref{eqn:43x}) follows from 
			\[
			k_s(\{\gamma\})=\sphericalangle\lim_{z\to\gamma}(z-\gamma)G_{k_{s}}(z)
			=\sphericalangle\lim_{z\to\gamma} \frac{z-\gamma}{\omega(z)-s}=0.
			\]
			
			\vspace{6pt}
			
			\noindent
			-- When $s=\omega(\gamma)$: since $\omega(z)$ tends nontangentially to $\omega(\gamma)=s$
			as $z$ tends nontangentially to $\gamma$ 
			as shown in \cite{BercoviciVoiculescu}), 
			(\ref{eqn:43x}) follows from taking the 
			nontangential limit as $z$ tends to $\gamma$ in 
			\[
			(z-\gamma)G_{k_{s}}(z)=\frac{z-\gamma}{\omega(z)-s}
			=\frac{(z-\gamma)G_{\mu\boxplus \nu}(z)}{(\omega(z)-s)G_{\mu}(\omega(z))}.
			\]  
		\end{proof}
		
		\vspace{6pt}

	\begin{theorem}\label{thm:45}
        Consider the framework and notation of Proposition \ref{prop:44}, and let $\muovadd$ be
        the overlap measure for the elements $a$ and $a+b$.  One has that
        \begin{equation}   \label{eqn:45a}
        d\muovadd(s,t) = o(s,t) \, \mu(ds) \ \mu\boxplus\nu(dt),
        \end{equation}
        where the function $o : \Spec (a) \times \Spec (a+b) \to \bR_{+}$ is defined by
        Equation (\ref{eqn:44a}).  In particular, this implies that
        $\muovadd \ll \mu\times ( \mu\boxplus\nu )$, and that the corresponding overlap function 
        $\ovadd$ is the function $o$ from (\ref{eqn:44a}). 
	\end{theorem}
    
	%
        % [Write-up from the preceding version of the paper.]
        % 
        % \begin{theorem}\label{thm:45}
	%	Let $a,b$ be freely independent selfadjoint random variables in 
        % an ncps $( \cA, \varphi )$. Let $\mu$ and respectively $\nu$ be the distributions
        % of $a$ and $b$ with respect to $\varphi$, and assume that neither of $\mu, \nu$ is 
        % a point mass. Then $d\muovadd(s,t)=\ovadd(s,t) \mu(ds)\,\mu\boxplus\nu(dt)$.
        % In particular $\muovadd\ll \mu\times \mu\boxplus\nu$ and $\ovadd$ is the overlap 
        % function of $a$ and $a+b$.
	% \end{theorem}
        %
        %    
	% \begin{proof}
	% 	We will verify that the defining property for the overlap measure holds. 
	%	
	%	Let $g\in \Bor_b ( \bR )$ be a bounded Borel function. According to 
        % Proposition \ref{prop:44}, for $\mu$-almost every $s\in \mathbb{R}$, 
        % $\int_{\bR}g(t)\ovadd(s,t)d\mu\boxplus\nu(t)=\int_{\bR}g(t)dk_s(t)$. 
        % Hence, using Fubini's Theorem and Remark \ref{rem:41},...
        
        \begin{proof}  Let $\nu$ be the measure on $\Spec (a) \times \Spec (a+b)$
        which is defined by the right-hand side of (\ref{eqn:45a}),
        $d \nu (s,t) = o(s,t) \, \mu(ds) \ \mu\boxplus\nu(dt).$
        We will verify that $\nu$ fulfills the requirement used in order to define  
        the overlap measure $\muovadd$. 

        We know from Proposition \ref{prop:44} that, for any given 
        $g\in \Bor_b ( \bR )$ we have the equality 
        \[
        \int_{\bR} g(t) \, o(s,t) \, d (\mu\boxplus\nu) (t)
        =\int_{\bR} g(t)dk_s(t),
        \]
        holding for $\mu$-almost every $s\in \mathbb{R}$. 
        By using Fubini's Theorem and Remark \ref{rem:41}, we then see that for every
        $f,g \in \Bor_b ( \bR )$ we have
        \begin{eqnarray*}  
        \int_{\bR^2} f(s) \, g(t) \, d \nu (s,t)
         &=& \int_{\bR^2} f(s) \, g(t) \, o(s,t) \, d\mu(s) \, d ( \mu\boxplus\nu) (t)  \\
         &=& \int_{\bR}f(s) \left[ \int_{\bR} g(t)\, o(s,t) \, d\mu\boxplus\nu(t)\right]d\mu(s)\\
         &=& \int_{\bR}f(s) \left[ \int_{\bR} g(t) \, dk_s(t)\right] \, d\mu(s)  \\
         &=& \varphi(f(a)g(a+b)),
   \end{eqnarray*}
   as required.
        %
        %
	% Using Fubini's theorem we have
	% \begin{align}\label{eqn:42}
			%     \int_{\bR^2} f(s)g(t) \ovadd(s,t)d\mu(s)d\mu\boxplus\nu(t)=\int_{\bR}f(s)\underbrace{\left(\int_{\bR}g(t)\ovadd(s,t)d\mu\boxplus\nu(t)\right)}_{=I}d\mu(s)
	% \end{align}
		% We have
	% \begin{align*}
			%     I&=\int_{\bR}g(t)\ovadd(s,t)d(\mu\boxplus\nu)^{ac}(t)+\int_{\bR}g(t)\ovadd(s,t)d(\mu\boxplus\nu)^{d}\\
			%     &=\int_{U}-\frac{g(t)}{\pi f_{\mu\boxplus\nu}(t)}\operatorname{Im}\left(\frac{1}{\omega(t)-s}\right) f_{\mu\boxplus\nu}(t)dt+\sum_{k=1}^n g(\gamma_k) o(s,\gamma_k) \mu\boxplus\nu(\{\gamma_k\})
			%     \\
			%     &=\int_{\bR} g(t) k_s^{ac}(dt)+\int_{\bR} g(t) k_s^{d}(dt)=\int_{\bR} g(t) k_s(dt)=\left(\mathcal{K} g\right)(s).
	% \end{align*}
		% Thus we see that RHS of \eqref{eqn:42} can be written as 
		% \[
            % \int_{\bR} f(s)\left(\mathcal{K} g\right)(s) d\mu(s)
            %   =\varphi\left(f(a)\left(\mathcal{K}g\right)(a)\right)=\varphi(f(a)g(a+b)).
            % \]
	% Hence the defining equality for overlap measure is satisfied.
        %
	\end{proof}

    \vspace{10pt}
    
	\begin{corollary}  \label{cor:46}
        Consider the framework and notation from Proposition \ref{prop:44} and
	  Theorem \ref{thm:45}.  For every bounded Borel function $f : \Spec(a) \to \bR$, 
        one has
		\begin{align*}
			\bE[ f(a) \mid a +b ] = h(a+b),
		\end{align*}
		where $h:\Spec(a+b)\to \mathbb{C}$ is defined $\mu\boxplus\nu$-almost everywhere by
		\begin{align}\label{eqn:46a}
			h(t)=\begin{cases}
				-\frac{1}{\pi}\frac{1}{f_{\mu\boxplus\nu}(t)}\operatorname{Im} 
				\left(\int_{\bR}f(s)  \frac{1}{\omega(t)-s} d\mu(s)\right) &\mbox{ if } t\in U,  \\
				f(\omega(t))&\mbox{ if } t \mbox{ is an atom of } \mu\boxplus\nu.
			\end{cases}
		\end{align}
	\end{corollary}
    
	\begin{proof}
		 From Theorem \ref{thm:45} and Fubini's theorem one gets that
		\begin{align*}
			\varphi(f(a)g(a+b))&=\int_{\bR^2}f(s)g(t)\ovadd(s,t)\mu(ds)\mu\boxplus\nu(dt)\\
			&=\int_{\bR}\left(\int_{\bR}f(s)\ovadd(s,t)\mu(ds)\right)g(t)\mu\boxplus\nu(dt).
		\end{align*}
		Denoting $h(t)=\int_{\bR}f(s)\ovadd(s,t)\mu(ds)$ we see that
		\[\varphi(f(a)g(a+b))=\varphi(h(a+b)g(a+b)).\]
		Explicit formula for the function $h$ follows directly from the formula \eqref{eqn:44a}. 
        One should observe that $h$ is only defined $(\mu\boxplus\nu)$-almost everywhere, 
        since $U$ may be a proper relatively open subset of $\Spec (a+b)$.
        However, this is sufficient for the correct definition of conditional expectation.
	\end{proof}

    \vspace{10pt}
    
	Since the atomic case considered in (\ref{eqn:46a}) is immediate to apply, 
        we investigate further how the formula (\ref{eqn:46a}) works in the other case, $t\in U$.
	  In the next proposition we observe that for $f(s)=s$ one can process the first branch of
        formula \eqref{eqn:46a} a bit further, and evaluate the integral included in it.
        
\vspace{6pt}

	\begin{proposition}   \label{prop:47}
	In the particular case where $f(s)=s$, $s\in\bR$, the function $h$
        defined by \eqref{eqn:46a} satisfies
		\begin{align}\label{eqn:47a}
		h(t) = -\frac{\operatorname{Im}\left(\omega(t)
              G_{\mu\boxplus\nu}(t)\right)}{\pi f_{\mu\boxplus\nu}(t)}, \ \
              \forall \, t\in U.
		\end{align} 
	\end{proposition}
	
	\begin{proof}
		Direct calculation gives
		\begin{align*}
		h(t)
            &= -\frac{1}{\pi}\int_{\bR} s \frac{1}{f_{\mu\boxplus\nu}(t)} 
               \operatorname{Im} \left(\frac{1}{\omega(t)-s}\right) d\mu(s) \\
            & = -\frac{1}{\pi f_{\mu\boxplus\nu}(t)}
                \operatorname{Im}\left(\int_{\bR} \left( 
                \frac{\omega(t)}{\omega(t)-s}-1\right)d\mu(s)\right)   \\
		& = -\frac{\operatorname{Im}\left(
                \omega(t)G_{\mu}(\omega(t))\right)}{\pi f_{\mu\boxplus\nu}(t)}
              = -\frac{\operatorname{Im}\left(\omega(t)
               G_{\mu\boxplus\nu}(t)\right)}{\pi f_{\mu\boxplus\nu}(t)}, \mbox{ as stated.}
		\end{align*}
	\end{proof}

\vspace{10pt}

	Formula (\ref{eqn:47a}) does not seem to further simplify in general, but it allows 
    for some more processing in the case when (the distribution $\mu$ of $a$ remains general, 
    but) the distribution $\nu$ of the element $b$ is assumed to be $\boxplus$-infinitely 
    divisible.  Under the latter assumption, the $R$-transform $r_\nu$ has an analytic 
    continuation to $\bC^+$ and one can write $\omega(z)=z-r_{\nu}(G_{\mu\boxplus\nu}(z))$.  
    Moreover, if we are in the special situation when $b$ is a centred semicircular element of 
    variance $\sigma^2$ (meaning that $\nu$ is the centred semicircle law of variance $\sigma^2$) 
    then the aforementioned $R$-transform takes the form $r_\nu(z)=\sigma^2z$, and we find the 
    nice formula which was announced in Section 1.2 of the Introduction under the name of 
    ``free additive Tweedie's formula''.  The details of this are given in part $1^o$ of the 
    next example.  Part $2^o$ of the example puts into evidence another situation,
    related to $\boxplus$-convolution powers, when the formula (\ref{eqn:47a}) takes
    a particularly nice form.
    %
    % Moreover, if we are in the situation when $b$ is a semicircular  element (while $a$ remains 
    % a selfadjoint element with general distribution), then (\ref{eqn:46a}) takes a particularly 
    % nice form, which remarkably depends only on the distribution $\mu\boxplus\nu$. Similar 
    % explicit formulas can be obtained in the case when the distributions of $a$ and $b$ are 
    % respectively $\mu$ and $\mu^{\boxplus r}$ for some Borel probability measure $\mu$ on $\bR$
    % and some $r\geq 1$. 

    \vspace{6pt}
    
	\begin{example} \label{example:47}      
		$1^o$ Suppose that $b$ is a centred semicircular element with variance $\sigma^2$. 
        Then $\mu\boxplus\nu$ is absolutely continuous with respect to Lebesgue measure 
        and (as first observed in \cite{BianeConvolution}) the subordination function $\omega$
        satisfies the equation 
	$\omega(z) = z-\sigma^2 G_{\mu\boxplus\nu}(z), \ \ z \in \bC^{+},$
        which extends by continuity to the case when $t \in U$.
        The numerator of the fraction on the right-hand side of (\ref{eqn:47a}) thus takes 
        the form
        \begin{align*}
        \mathrm{Im} \bigl( \, (t - \sigma^2 \, G_{\mu \boxplus \nu} (t))
                                            \, G_{\mu \boxplus \nu} (t) \, \bigr)
        & = t \, \mathrm{Im} G_{\mu \boxplus \nu} (t)) 
           - 2 \sigma^2 \, \mathrm{Re} G_{\mu \boxplus \nu} (t) 
                  \cdot \mathrm{Im} G_{\mu \boxplus \nu} (t)    \\
        & = - \pi t f_{\mu \boxplus \nu} (t) - 2 \sigma^2 
                \cdot \pi \, H_{\mu \boxplus \nu} (t) 
                 \cdot \bigl( - \pi \, f_{\mu \boxplus \nu} (t) \bigr), 
        \end{align*}
         where $H_{\mu\boxplus\nu}$ is the Hilbert transform of $\mu\boxplus\nu$,
and where at the latter equality sign we took advantage of the fact that the 
continuous extension of $G_{\mu \boxplus \nu}$ at a point $t \in U$ can be written 
in the form 
$G_{\mu \boxplus \nu} (t) = \pi \bigl( \, H_{\mu \boxplus \nu} (t) 
- i f_{\mu \boxplus \nu} (t) \, \bigr).$
Upon plugging the above calculation back into Equation \eqref{eqn:47a} we find that
		\begin{equation}   \label{eqn:47x}
			h(t) = -\frac{-\pi t f_{\mu\boxplus\nu}(t)+2\pi^2\sigma^2 
                   f_{\mu\boxplus\nu}(t) H_{\mu\boxplus\nu}(t)}{\pi f_{\mu\boxplus\nu}(t)}
                = t - 2\pi \sigma^2H_{\mu\boxplus\nu}(t),
		\end{equation}
and this leads to the free additive analogue of Tweedie's formula that was announced 
in Equation (\ref{eqn:freetweedieboxplus}) of the Introduction,
		\begin{align*}
			\bE(a \, | \, a+b) = a+b - 2 \pi \sigma^2 H_{\mu\boxplus\nu}(a+b).
		\end{align*}
        As already noted in the Introduction, it is remarkable that, in this important special 
        case, the formula for $h$ uses solely the distribution of $a+b$, 
        and does not invoke the (unknown) distribution of the element $a$.
        
	For the record, we mention that in this special case the corresponding
    overlap function $\ovadd$ is given by the formula
		\[
        \ovadd(s,t) 
        = \frac{1}{(t-s-\pi \sigma^2H_{\mu\boxplus\nu}(t))^2
        +\pi^2f_{\mu\boxplus\nu}^2(t)}, \ \mbox{ for $s\in(-2,2)$ and $t\in U$.}
            \]
        The calculation leading to this formula is similar to the one showed for $h$ in 
        (\ref{eqn:47x}), and is left as exercise to the reader.

        \vspace{6pt}

        \noindent
	$2^o$ If the distributions of $a,b$ are $\mu$ and $\mu^{\boxplus r}$ respectively 
    for some Borel probability measure $\mu$ on $\mathbb{R}$ and some $r\geq 1$, then from \cite{BelinschiBercoviciAdditiveSemigroup} we know that 
    $G_{\mu^{\boxplus r+1}}(z)=G_{\mu}(\omega_{r+1}(z))$ with
	\[
           \omega_{r+1}(z)
                  =\frac{z}{r+1}+\left(1-\frac{1}{r+1}\right)
                  \, \frac{1}{G_{\mu^{\boxplus r+1}}(z)}.
        \]
	An immediate application of \eqref{eqn:47a} then gives the expected answer,
		\[
            \bE(a|a+b)=\frac{a+b}{r+1}.
            \]
	\end{example}

\vspace{0.5cm}
    
\section{Multiplicative case}
\label{MultiplicativeFormulas}	
	In this section we consider $a,b$ freely independent positive random variables with respective distributions $\mu, \nu$. Similarly as in the additive case, we first observe that considerations from Section 2.5 imply that the measures $k_s$ defined by
	\begin{align*}
		\int_{\bR}\frac{zt}{1-zt}dk_s(t)
		=\frac{\omega(z)s}{1-\omega(z)s}, \ \ z \in \bC^{+}.
	\end{align*}
	give us disintegration of the overlap measure $\muovmulp$ with respect to its first marginal $\mu$. As in the previous section we will show that for $\mu$-almost every $s$, the measure $k_s$ is absolutely continuous with respect to $\mu\boxtimes\nu$. The corresponding Radon-Nikodym derivative will allow us to
	calculate explicitly conditional expectations of the form 
	$\bE(f(a) \mid a^{1/2}ba^{1/2})$. We start with a 
	review of regularity of free multiplicative convolution
	on $\bR_+$. In this case we pay special attention to 
	the atoms at $0$, which behave differently.
	
	\vspace{6pt}
	
	\begin{review}[Review of regularity of free multiplicative convolution on $\bR_+$] \label{rev:51}
		Let $\mu,\nu$ be two Borel probability measures on $\bR_+$ neither of them a point mass; 
		denote by $\mu \boxtimes \nu$ the free multiplicative convolution).
		
		\vspace{6pt}
		
		\noindent
		$1^o$ The measure $\mu\boxtimes\nu$ has at most finitely many atoms 
		$\gamma_i,\, i=1,\ldots,n$. For every $\gamma_i>0$, $i=1,\ldots,n$ 
		there exist uniquely determined $\alpha_i,\beta_i\in (0, \infty)$ such that 
		$\alpha_i\beta_i=\gamma_i$ 
		and $\mu(\{\alpha_i\})+\nu(\{\beta_i\})>1$.
		Moreover, $\mu \boxtimes \nu (\{0\}) = \max \{ \mu( \{0\} ), \nu( \{0\}) \}$ 
		(see \cite{BelinschiAtomsMult}).
		
		\vspace{6pt}
		
		\noindent
		$2^o$ With notation as in $1^o$ we have $\lim_{z\to1/\gamma_i}\omega(z)=1/\alpha_i$, 
		for every $i=1,\ldots,n$ (see \cite{BelinschiAtomsMult}).
		
		\vspace{6pt}
		
		\noindent
		$3^o$ Measure $\mu\boxtimes\nu$ has no singular continuous part 
		(see \cite{JiRegularityMult}).
		
		\vspace{6pt}
		
		\noindent
		$4^o$ There exists a closed set $0\in E\subset \bR_+$ of Lebesgue measure 
		zero such that the density of $(\mu\boxtimes\nu)^{ac}$, denoted by 
		$f_{\mu\boxtimes\nu}$, is analytic on 
		$\bR_{+}\setminus E$ (see \cite{BelinschiThesis,Belinschi}). 
		
		\vspace{6pt}
		
		\noindent
		$5^o$ With notation from $4^o$ the subordination function $\omega$ can be 
		continued continuously to the real line, and the continuation, still denoted
		by $\omega$, has strictly positive imaginary part at every point $x\in\bR$
		such that $1/x\in \mathrm{supp}(\mu\boxtimes\nu)^{ac}\setminus E$ 
		(see \cite{BelinschiThesis,Belinschi}).
	\end{review}
	
	\begin{proof}
		Similarly as before, only $5^o$ requires an explanation. We know that $\omega_1(x)\omega_2(x)=x\eta_{\mu\boxtimes\nu}(x)$, hence we have 
		$\frac{\omega_1(x)}{x}=\frac{\eta_{\nu}(\omega_2(x))}{\omega_2(x)}$, if we know that $\omega_2(x)$ has positive imaginary part, then $\frac{\eta_{\nu}(\omega_2(x))}{\omega_2(x)}$ has positive imaginary part, as $\eta_\mu$ strictly increases argument of $\omega_2(x)$ (see \cite{BelinschiBercoviciMultSemigroup}). 
		Thus $\omega_1(x)$ has strictly positive imaginary part.
	\end{proof}
	
	\vspace{10pt}
	
	In the multiplicative case the behaviour of subordination functions at zero requires 
	a bit more attention. It is explained in detail in \cite{JiRegularityMult}, let us discuss
	some elementary results here.
	
	\vspace{6pt}
	
	\begin{lemma}\label{lem:52}
		In the framework and notation of Review \ref{rev:51} we have
		\begin{align*}
			\lim_{z\to-\infty} \omega(z)=
			\begin{cases}
				-\infty \qquad &\mbox{ if } \mu(\{0\})\geq \nu(\{0\}),\\
				\psi_{\mu}^{-1}(\nu(\{0\})-1) &\mbox{ if } \mu(\{0\})< \nu(\{0\}).
			\end{cases}     
		\end{align*}
	\end{lemma}
	
	\begin{proof}
		For any probability measure $\mu\neq \delta_0$ on $\bR_+$ the function $\psi_\mu:(-\infty,0)\mapsto (\mu(\{0\})-1,0)$ is a homeomorphism. Therefore $\psi_\mu$ has a continuous inverse $\psi_\mu^{-1}:(\mu(\{0\})-1,0)\mapsto (-\infty,0)$. We observe that for  $z\in(-\infty,0)$ function $\omega(z)=\psi_\mu^{-1}(\psi_{\mu\boxtimes\nu}(z))$ is well defined. 
		
		From Lebesgue dominated convergence it follows that 
		\[
		\lim_{z\to-\infty} \psi_{\mu\boxtimes\nu}(z)=\mu\boxtimes\nu(\{0\})-1=\max\{\mu(\{0\}),\nu(\{0\})\}-1.
		\]
		
		If $\mu\boxtimes\nu(\{0\})=\mu(\{0\})$ then  $\lim_{z\to-\infty}\psi_\mu^{-1}(\psi_{\mu\boxtimes\nu}(z))=-\infty$. In the case $\mu\boxtimes\nu(\{0\})=\nu(\{0\})>\mu(\{0\})$ we have $\lim_{z\to-\infty}\psi_\mu^{-1}(\psi_{\mu\boxtimes\nu}(z))=\psi_{\mu}^{-1}(\nu(\{0\})-1)$, in this case of course $\nu(\{0\})-1\in(\mu(\{0\})-1,0)$, hence $\psi_{\mu}^{-1}(\nu(\{0\})-1)$ is well defined.
	\end{proof}
	
	\vspace{10pt}
	
	The next proposition states the precise formula that one has for the overlap function in this case.
	The proposition does not refer to some freely independent positive elements $a,b$ in an ncps -- but 
	if such $a,b$ were in the picture, then the function $o$ defined in (\ref{eqn:53a}) would be the overlap function as
	$o_{a, a^{1/2} b a^{1/2}}$.
	
	\vspace{6pt}
	
	\begin{proposition}\label{prop:53}
		With notation from Review \ref{rev:51} above, consider the function 
		
		\noindent
		$o : \mathrm{supp}(\mu)\times \mathrm{supp}(\mu\boxtimes \nu)\to \bR_+$ 
		defined (in $\mu\times \mu\boxtimes \nu$-almost everywhere sense)
		by the following formula:
		\begin{equation}   \label{eqn:53a}
			o(s,t) :=
			\begin{cases}
				-\frac{1}{\pi} \frac{1}{tf_{\mu\boxtimes\nu}(t)}
				\operatorname{Im}\left(\frac{1}{1-\omega(1/t)s}\right)\qquad 
				&\mbox{ if } s\in\mathrm{supp}(\mu)\\ 
				&\mbox{ and } t\in \mathrm{supp}(\mu\boxtimes\nu)^{ac}\setminus E,\\
				\frac{1}{\mu(\{s\})} 1_{\omega(1/t)=1/s}\qquad 
				&\mbox{ if } s\in\mathrm{supp}(\mu)\\ 
				&\mbox{ and $t>0$ is an atom of $\mu\boxtimes\nu$},\\
				\frac{1}{\nu(\{0\})}\frac{1}{1-s \psi_\mu^{-1}(\nu(\{0\})-1)}1_{\nu(\{0\})>\mu(\{0\})} \qquad 
				&\mbox{ if } s\in\mathrm{supp}(\mu)\setminus\{0\}, t=0,\\
				\frac{1}{\mu\boxtimes\nu(\{0\})}1_{\mu\boxtimes\nu(\{0\})>0} \qquad 
				&\mbox{ if } s=t=0,
			\end{cases}
		\end{equation}
		Consider the measure $k_s$ from \eqref{eqn:25a}. 
		Then: for $\mu$-almost every $s \in \bR_+$ we have $k_s\ll \mu\boxtimes\nu$ and the Radon-Nikodym 
		derivative is given by the function $o(s, \cdot )$ (sending $t \mapsto o(s,t)$).
	\end{proposition}
	
	\begin{proof}
		The proof is very similar to the proof of Proposition \ref{prop:44}, with the set $U$ being 
		replaced by $\mathrm{supp}(\mu\boxtimes\nu)^{ac}\setminus E$. 
		
		After writing 
		\[G_{k_s}(z)=\frac{1}{z}\frac{1}{1-\omega(1/z)s},\]
		we apply Stieltjes inversion to get the formula in the case $t\in\mathrm{supp}(\mu\boxtimes\nu)^{ac}\setminus E$. Atoms different than $0$ are also treated analogously.
		
		Let us calculate the size the atoms at $0$ of $k_s$. Of course for any measures $\mu,\nu$ we have $k_0=\delta_0$, which proves the last point in definition of $\ovmulp$. 
		
		Consider $s>0$ we have $\lim_{z\to\-\infty}\psi_{k_s}(z)=k_s(\{0\})-1$, on the other hand
		\begin{align*}
			\lim_{z\to-\infty}\psi_{k_s}(z)=\lim_{z\to-\infty} \frac{\omega(z)s}{1-\omega(z)s}=\begin{cases}
				-1\qquad &\mbox{ if } \mu(\{0\})\geq\nu(\{0\})\\
				\frac{s\psi_\mu^{-1}(\nu(\{0\})-1)}{1-s\psi_\mu^{-1}(\nu(\{0\})-1)}&\mbox{ if } \mu(\{0\})<\nu(\{0\}).
			\end{cases}
		\end{align*}
		Thus if $\mu(\{0\})<\nu(\{0\})$, then for every $s>0$ we have $k_s(\{0\})=\frac{s\psi_\mu^{-1}(\nu(\{0\})-1)}{1-s\psi_\mu^{-1}(\nu(\{0\})-1)}+1=\frac{1}{1-s\psi_\mu^{-1}(\nu(\{0\})-1)}$, which justifies second line of the definition of $\ovmulp$.
	\end{proof}
	
	\begin{theorem}\label{thm:54}
		Let $a,b$ be freely independent positive random variables in an ncps $( \cA, \varphi )$. Let $\mu$ and respectively $\nu$ be the distributions of $a$ and $b$ with respect to $\varphi$, and assume that neither of $\mu, \nu$ is a point mass. Then
		$d\muovmulp(s,t)={\ovmulp(s,t)} \mu(ds)\,\mu\boxtimes\nu(dt)$. In particular $\muovmulp\ll \mu\times \mu\boxtimes\nu$ and $\ovmulp$ is the overlap function of $a$ and $a^{1/2}ba^{1/2}$.
	\end{theorem}
	
	The proof of this theorem is very similar to the one of Theorem \ref{thm:45}.
	
	\begin{corollary}\label{cor:55}
		With the notation of Theorem \ref{thm:54}, for every $f\in\mathrm{Bor}_b (\Spec(a))$ we have
		\begin{align*}
			\bE[ f(a) \mid a^{1/2}ba^{1/2} ] = h(a^{1/2}ba^{1/2}),
		\end{align*}
		where $h:\Spec(a^{1/2}ba^{1/2})\to \mathbb{C}$ is defined (in $\mu\boxtimes\nu$-almost everywhere sense) by
		\begin{align*}
			h(t)=\begin{cases}
				-\frac{1}{\pi t f_{\mu\boxtimes\nu}(t)}\int_{\bR_+}f(s)\operatorname{Im}\left(\frac{1}{1-\omega(1/t)s}\right)d\mu(s)\qquad &\mbox{ if } t\in \mathrm{supp}(\mu\boxtimes\nu)^{ac}\setminus E,\\
				f\left(\frac{1}{\omega(1/t)}\right)&\mbox{ if } t >0 \mbox{ is an atom of } \mu\boxtimes\nu ,\\
				\frac{1}{\nu(\{0\})}\int_{\bR_+}\frac{f(s)}{1-s \psi_\mu^{-1}(\nu(\{0\})-1)}d\mu(s) &\mbox{ if } t=0\mbox{ and } \mu(\{0\})<\nu(\{0\}),\\
				f(0) &\mbox{ if } t=0\mbox{ and } \mu(\{0\})\geq\nu(\{0\}).
			\end{cases}
		\end{align*}
	\end{corollary}
	
	Here again, the proof is exactly the same as in the additive case.
	
	We will calculate the conditional expectation in the case $f(s)=s$ and then we will further specialize 
	$b$ to be a free Poisson element, in which case we will retrieve the well-known Ledoit-Péché estimator 
	of covariance matrix from \cite{LedoitPeche}. The calculations are non-trivial only in two cases: 
	$t=0$ when $\mu(\{0\})<\nu(\{0\})$ and when $t\in\mathrm{supp}(\mu\boxtimes\nu)^{ac}\setminus E$.
	
	\vspace{6pt}
	
	\begin{proposition}
		With the notation of Theorem \ref{thm:54}, one has $\bE(a|a^{1/2}ba^{1/2})=h(a^{1/2}ba^{1/2})$ where $h:\Spec(a^{1/2}ba^{1/2})\to \mathbb{C}$ is defined (in $\mu\boxtimes\nu$-almost everywhere sense) by
		\begin{align}\label{eq:prop56}
			h(t)=\begin{cases}
				-\frac{1}{\pi t f_{\mu\boxtimes\nu}(t)}\operatorname{Im}\left(\frac{1}{\omega(1/t)}\psi_{\mu\boxtimes\nu}\left(\frac{1}{t}\right)\right) \qquad&\mbox { if } t\in\mathrm{supp}(\mu\boxtimes\nu)^{ac}\setminus E,\\
				\frac{1}{\omega(1/t)} \qquad&\mbox { if $t>0$ is an atom of $\mu\boxtimes\nu$},\\
				\frac{\nu(\{0\})-1}{\nu(\{0\})\psi_\mu^{-1}(\nu(\{0\})-1)} \qquad &\mbox{ if $t=0$ and $\mu(\{0\})<\nu(\{0\})$},\\
				0 \qquad &\mbox{ if $t=0$ and $\mu(\{0\})\geq \nu(\{0\})$}.
			\end{cases}
		\end{align}
	\end{proposition}
	
	\begin{proof}
		Indeed in the case $t=0$ we get  
		\begin{align*}
			h(0)&=\frac{1}{\nu(\{0\})}\int_{\bR_+}\frac{s}{1-s \psi_\mu^{-1}(\nu(\{0\})-1)}d\mu(s)
			\\&=\frac{1}{\nu(\{0\})\psi_\mu^{-1}(\nu(0)-1)}\int_{\bR_+}\frac{s\psi_\mu^{-1}(\nu(\{0\})-1)}{1-s \psi_\mu^{-1}(\nu(\{0\})-1)}d\mu(s)\\
			&=\frac{\psi_\mu\left(\psi_\mu^{-1}(\nu(\{0\})-1)\right)}{\nu(\{0\})\psi_\mu^{-1}(\nu(\{0\})-1)}=\frac{\nu(0)-1}{\nu(\{0\})\psi_\mu^{-1}(\nu(\{0\})-1)}.
		\end{align*} 
		In the case $t\in\mathrm{supp}(\mu\boxtimes\nu)^{ac}\setminus E$ we get
		\begin{align*}
			h(t)&=-\frac{1}{\pi t f_{\mu\boxtimes\nu}(t)}\operatorname{Im}\left(\int_{\bR_+}\frac{s}{1-\omega(1/t)s}\right)d\mu(s)\\
			&=-\frac{1}{\pi t f_{\mu\boxtimes\nu}(t)}\operatorname{Im}\left(\frac{1}{\omega(1/t)}\int_{\bR_+}\frac{\omega(1/t)s}{1-\omega(1/t)s}\right)d\mu(s)\\
			&=-\frac{1}{\pi t f_{\mu\boxtimes\nu}(t)}\operatorname{Im}\left(\frac{1}{\omega(1/t)}\psi_{\mu\boxtimes\nu}\left(\frac{1}{t}\right)\right).
		\end{align*}
	\end{proof}
	
	\begin{example}\label{ex:57}
		Assume further that $b$ is a free Poisson element with parameter $\lambda>0$. This means in particular that $\psi_\nu^{-1}(z)=\tfrac{z}{(\lambda+z)(1+z)}$. We have 
		\[\omega_2(z)=\psi_\nu^{-1}(\psi_{\mu\boxtimes\nu}(z))=\frac{\psi_{\mu\boxtimes\nu}(z)}{(\lambda+\psi_{\mu\boxtimes\nu}(z))(1+\psi_{\mu\boxtimes\nu}(z))}=\frac{\eta_{\mu\boxtimes\nu}(z)}{\lambda+\psi_{\mu\boxtimes\nu}(z)}.\]
		Taking into account that $\omega_1(z)\omega_2(z)=z\eta_{\mu\boxtimes\nu}(z)$, we see that $\omega(z)=\omega_1(z)=z(\lambda+\psi_{\mu\boxtimes\nu}(z))$. Hence we obtain in the case $t\in\mathrm{supp}(\mu\boxtimes\nu)^{ac}\setminus E$ that
		\begin{align*}
			h(t)=-\frac{1}{\pi f_{\mu\boxtimes\nu}(t)}\operatorname{Im}\left(\frac{\psi_{\mu\boxtimes\nu}\left(\frac{1}{t}\right)}{\lambda+\psi_{\mu\boxtimes\nu}\left(\frac{1}{t}\right)}\right).
		\end{align*}
		Using the relation $\psi_{\mu\boxtimes\nu}(z)=\frac{1}{z}G_{\mu\boxtimes\nu}\left(\frac{1}{z}\right)-1$, straightforward processing of the resulting formula gives
		\[h(t)=\frac{\lambda t}{|\lambda-1+tG_{\mu\boxtimes\nu}(t)|^2}.\]
		
		If $\lambda<1$ and $\mu(\{0\})<\nu(\{0\})=1-\lambda$ then for $t=0$ we get
		\[
		h(0)=\frac{-\lambda}{(1-\lambda)\psi_\mu^{-1}(-\lambda)}.
		\]
		Remember that $\psi_\mu^{-1}(-\lambda)=\lim_{z\to-\infty}\omega(z)=\lim_{z\to-\infty}z(\lambda+\psi_{\mu\boxtimes\nu}(z))$.  
		One easily verifies that 
		$\psi_{\mu\boxtimes\nu}(z)=\psi_{(\mu\boxtimes\nu)^{ac}}(z)$. Then we have $\lim_{z\to-\infty}\omega(z)=\lim_{z\to-\infty}z(\lambda+\psi_{(\mu\boxtimes\nu)^{ac}}(z))=\lim_{z\to-\infty}G_{(\mu\boxtimes\nu)^{ac}}\left(\frac{1}{z}\right)=G_{(\mu\boxtimes\nu)^{ac}}\left(0\right)$. Finally we obtain 
		\[h(0)=-\frac{\lambda}{(1-\lambda)G_{(\mu\boxtimes\nu)^{ac}}\left(0\right)}.
		\]
		
		One can also find the overlap function, using  $\psi_{\mu\boxtimes\nu}(z)=\frac{1}{z}G_{\mu\boxtimes\nu}\left(\frac{1}{z}\right)-1$ and $\omega(z)=\omega_1(z)=z(\lambda+\psi_{\mu\boxtimes\nu}(z))$ the overlap function from Proposition \ref{prop:53} in the case $t\in\mathrm{supp}(\mu\boxtimes\nu)\setminus E$ and $s\in\mathrm{supp}(\mu)$ simplifies to
		\begin{align*}
			o(s,t)=&-\frac{1}{\pi} \frac{1}{t f_{\mu\boxtimes\nu}(t)} \operatorname{Im}\left(\frac{1}{1-\tfrac{s}{t}(\lambda+t G_{\mu\boxtimes\nu}(t)-1)}\right)\\
			=&\frac{st}{(t-s(\lambda-1)-st \pi H_{\mu\boxtimes\nu}(t))^2+\pi^2 s^2 t^2 f_{\mu\boxtimes\nu^2}(t)}.
		\end{align*}
	\end{example}
	
	% \vspace{1cm}
	
\section{Relation to conditional freeness}
\label{CFreeFormulas}	
	In this section, we relate to conditionally free probability theory the 
	general question of determining the conditional expectation $E(x \mid y)$ 
	in the case where: $y=P(a,b)$ is a suitable selfadjoint element in the von Neumann
	algebra generated by two freely independent non-commutative variables $a,b$ 
	in some ncps $(\mathcal{A},\varphi)$, 
	and $x=f(a)$ is a selfadjoint element in the von Neumann algebra generated 
	by $a$. (For concrete examples, $P(a,b)$ could be a selfadjoint polynomial 
	in $a$ and $b$; or, in a situation where $a$ is known to be positive, it 
	could be that $P(a,b) = a^{1/2} b a^{1/2}$.) 
	
	Without much loss of generality, the considerations about $E( f(a) \mid P(a,b) )$ 
	which are made throughout this section will also take in the additional assumptions 
	that the function $f$ (which is a real-valued bounded Borel function supported 
	on the spectrum of $a$) is non-negative and is normalized by the condition that 
	$\varphi \bigl( \, f(a) \, \bigr) = 1$.
	
	\subsection{Free denoiser as Radon-Nikodym derivative}
	
	\begin{remark-and-notation}\label{freedenoiserasRadonNikodym}
		In the framework that was just described, the conditional 
		expectation $E(f(a) \mid P(a,b))$ is a positive element in the von Neumann 
		algebra generated by $P(a,b)$, and can therefore be 
		written as $h \bigl( \, P(a,b) \, \bigr)$, where $h$ is a bounded Borel 
		supported on the spectrum of $P(a,b)$ and with values in $[0, \infty )$.  
		We will refer to $h$ by calling it {\em free denoiser} (of $f(a)$ with 
		respect to $P(a,b)$).  Our point in the present remark is to observe that the 
		free denoiser $h$ can be written as a Radon-Nikodym derivative, in connection 
		to two distributions which appear naturally in the discussion.  One of the 
		distributions is just the distribution of $P(a,b)$ with respect to the state 
		$\varphi$ of our ncps $( \cA , \varphi )$; we will denote this distribution
		as $\mu_{P(a,b)}^{\varphi}$.  For the second distribution, we consider the
		new state $\chi : \cA \to \bC$ defined by putting
		\begin{equation*} 
			\chi (c) := \varphi ( \, f(a)c \, \, ) \ \ c \in \cA,
		\end{equation*}
		and we consider the distribution of $P(a,b)$ in the ncps $( \cA , \chi )$; 
		the latter distribution will be denoted as $\mu_{P(a,b)}^{\chi}$.  The claim
		we make is that:
		\begin{equation*} 
			\frac{ d \mu_{P(a,b)}^{\chi} }{ 
				d \mu_{P(a,b)}^{\varphi} } = h 
			\ \ \mbox{ (free denoiser),}
		\end{equation*}
		i.e.~that for any real valued bounded Borel function 
		$g$ supported on the spectrum of $P(a,b)$ we have
		\begin{equation}  \label{eqn:RN2}
			\int_{\bR} g (t) \mu_{P(a,b)}^{\chi} ( dt )
			= \int_{\bR} g (t) h(t) \mu_{P(a,b)}^{\varphi} ( dt ).
		\end{equation}
		For the verification of the claim, we start from the left-hand side of 
		(\ref{eqn:RN2}) and we write:
		\begin{align*}
			\int_{\bR} g (t) \mu_{P(a,b)}^{\chi} ( dt )
			& = \chi \bigl( \, g( \, P(a,b) \, \bigr)
			\mbox{ (by the definition of $\mu_{P(a,b)}^{\chi}$) }             \\
			& = \varphi \bigl( \, f(a) \, g( \, P(a,b) \, \bigr)
			\mbox{ (by the definition of $\chi$) }                            \\
			& = \varphi \Bigl( 
			\, E \bigl( f(a) \, g( P(a,b) ) \mid P(a,b) \bigr) \, \Bigr)
			\mbox{ (since $E( \cdot \mid P(a,b)$) is $\varphi$-invariant) }    \\
			& = \varphi \Bigl( 
			\, E \bigl( f(a) \mid P(a,b) \bigr) \cdot g (P(a,b)) \, \Bigr)
			\mbox{ (conditional expectation property) }                       \\
			& = \varphi \bigl( \, h( P(a,b)) \cdot g (P(a,b)) \, \bigr)
			\mbox{ (by the definition of $h$) }                              \\
			& = \varphi \bigl( \, (g \cdot h) (P(a,b)) \, \bigr)                 \\
			& =  \int_{\bR} g (t) h(t) \mu_{P(a,b)}^{\varphi} ( dt )
			\mbox{ (by the definition of $\mu_{P(a,b)}^{\varphi}$), }   
		\end{align*}
		thus arriving to the right-hand side of (\ref{eqn:RN2}).
		
		\vspace{10pt}
		
		We next move to connecting the two distributions 
		$\mu_{P(a,b)}^{\varphi}, \mu_{P(a,b)}^{\chi}$ to conditionally free probability 
		theory.  Towards that end, we will first review (following 
		\cite{BozejkoLeinertSpeicher}) some basic definitions and facts of this theory.
	\end{remark-and-notation}

	\subsection{Review of conditionally free probability theory}
	
	\begin{definition}
		$1^o$ A {\em two-state non-commutative probability space} is a 
		triple $(\mathcal{A},\varphi,\chi)$
		where $(\mathcal{A},\varphi)$ is a ncps and 
		$\chi:\mathcal{A}\to \mathbb{C}$ is another normal state (unital, 
		positive, weakly continuous linear functional) on $\mathcal{A}$.
		
		\vspace{6pt}
		
		\noindent
		$2^o$ Let $(\mathcal{A},\varphi,\chi)$ be a two-state
		non-commutative probability space and let $a$ be a selfadjoint element of $\cA$. 
		Then besides the distribution $\mu$ of $a$ with respect to $\varphi$,
		one can also consider the distribution $\nu$ of $a$ in the ncps $( \cA , \chi )$.
		The couple $( \mu , \nu )$ of compactly supported probability measures on $\bR$
		that is obtained in this way will be referred to as 
		{\em distribution of $a$ in $( \cA , \varphi, \chi )$}.
	\end{definition}

	\begin{definition}
		Unital subalgebras $(\mathcal{A}_i)_{i\in I}$ of a unital complex algebra $\mathcal{A}$ are said to be conditionally freely independent in the two-state non-commutative probability space $(\mathcal{A},\varphi,\chi)$ when they are freely independent in $(\mathcal{A},\varphi)$ and in addition the following condition holds: for all $n\geq 2$, $i_1\neq i_2\neq \ldots \neq i_n$ in $I$ and $a_1\in \mathcal{A}_{i_1}, \ldots , a_n\in \mathcal{A}_{i_n}$ such that $\varphi(a_1)=\ldots =\varphi(a_n)=0$,$$\chi(a_1\cdots a_n)=\chi(a_1)\cdots\chi(a_n).$$
		
		Accordingly, non-commutative random variables will be said to be conditionally freely independent in the two-state non-commutative probability space $(\mathcal{A},\varphi,\chi)$ when the von Neumann subalgebras they generate in $\mathcal{A}$ are so.
	\end{definition}

	\subsection{Upgrading freeness to conditional freeness}
	
	As originally observed in \cite{BozejkoLeinertSpeicher}, freely independent 
	unital subalgebras $(\mathcal{A}_i)_{i\in I}$ of a non-commutative probability 
	space $(\mathcal{A},\varphi)$ are conditionally freely independent in the two-state 
	non-commutative probability space given by $(\mathcal{A},\varphi,\varphi)$. 
	We will use the following slightly more general result.
	
	\begin{proposition}
		Let $(\mathcal{A},\varphi)$ be a non-commutative probability space and $(\mathcal{A}_i)_{i\in I}$ be freely independent unital subalgebras in $(\mathcal{A},\varphi)$. Fix $i_0\in I$ and a positive element 
		$x \in \mathcal{A}_{i_0}$ such that $\varphi(x)=1$. Define another state $\chi:\mathcal{A}\to \mathbb{C}$ by
		\[
		\chi(y):=\varphi(x y),\qquad \forall y\in\mathcal{A}.
		\]
		Then the subalgebras $(\mathcal{A}_i)_{i\in I}$ are conditionally freely independent in $(\mathcal{A},\varphi,\chi)$.
	\end{proposition}
	
	\begin{proof}
		Let $n\geq 2$, $i_1\neq i_2\neq \ldots \neq i_n$ in $I$ and 
		$a_1\in \mathcal{A}_{i_1}, \ldots , a_n\in \mathcal{A}_{i_n}$ be 
		such that $\varphi(a_1)=\ldots =\varphi(a_n)=0$. We want to show that 
		\begin{equation}\label{c-free condition}
			\chi(a_1\cdots a_n)=\chi(a_1)\cdots\chi(a_n).
		\end{equation} 
		
		Observe that there always exists $k\in \{1,\ldots,n\}$ such that $i_k\neq i_0$. Then, by free independence of $\mathcal{A}_{i_0}$ and $\mathcal{A}_{i_k}$, $\chi(a_k)=\varphi(xa_k)=\varphi(x)\varphi(a_k)=0$ so that the right-hand side of \eqref{c-free condition} vanishes.
		
		To prove that the left-hand side of \eqref{c-free condition} vanishes, we distinguish two cases: $i_0\neq i_1$ or $i_0=i_1$. On the one hand, if $i_0\neq i_1$, then \begin{eqnarray*}
			\chi(a_1\cdots a_n)
			&=&\varphi(xa_1\cdots a_n)\\
			&=&\varphi((x-1)a_1\cdots a_n)+\varphi(a_1\cdots a_n)\\
			&=&0,
		\end{eqnarray*}
		where we used twice in the last line free independence of $(\mathcal{A}_i)_{i\in I}$. 
		On the other hand, if $i_0=i_1$, then \begin{eqnarray*}
			\chi(a_1a_2\cdots a_n)
			&=&\varphi(xa_1a_2\cdots a_n)\\
			&=&\varphi((xa_1-\varphi(xa_1))a_2\cdots a_n)+\varphi(xa_1)\varphi(a_2\cdots a_n)\\
			&=&0,
		\end{eqnarray*}
		where we again used twice in the last line free independence
		of $(\mathcal{A}_i)_{i\in I}$. In any case, we have proved that $$\chi(a_1\cdots a_n)=0=\chi(a_1)\cdots\chi(a_n),$$ as required.\end{proof} 
	
	\subsection{Conditional expectations and conditional freeness}
	
	\begin{definition}
		Let $a,b$ be conditionally freely independent selfadjoint non-commutative random variables with respective distributions $(\mu_1,\mu_2)$ and $(\nu_1,\nu_2)$ in some two-state non-commutative probability space $(\mathcal{A},\varphi,\chi)$ and $P$ be a selfadjoint polynomial in two non-commuting variables. We will denote by $(\mu_1\Box^P\nu_1,(\mu_1,\mu_2)\Box_c^P(\nu_1,\nu_2))$ the distribution of the non-commutative random variable $P(a,b)$ in $(\mathcal{A},\varphi,\chi)$.
	\end{definition}
	
	\begin{lemma} \label{twostatesdistributionofP(a,b)}
		Let $a,b$ be freely independent selfadjoint non-commutative random variables with respective distributions $\mu, \nu$ in a ncps $(\mathcal{A},\varphi)$, $f:\mathbb{R}\to \mathbb{R}$ be a bounded positive Borel-measurable function such that $\varphi(f(a))=1$ and $P$ be a selfadjoint polynomial in two non-commuting variables. Define another state $\chi:\mathcal{A}\to \mathbb{C}$ by
		\[\chi(y):=\varphi(f(a)y),\qquad \forall y\in\mathcal{A}.\] Then the distribution of the non-commutative random variable $P(a,b)$ in the two-state non-commutative probability space $(\mathcal{A},\varphi,\chi)$ is $(\mu\Box^P\nu,(\mu,f\cdot\mu)\Box_c^P(\nu,\nu))$.
	\end{lemma}
	
	\begin{proof}
		By assumption, the respective distributions of $a$ and $b$ with respect to $\varphi$ are $\mu$ and $\nu$. We check then that the respective distributions of $a$ and $b$ with respect to $\chi$ are $f\cdot \mu$ and $\nu$: for any bounded Borel-measurable function $g:\mathbb{R}\to \mathbb{R}$, $$\chi(g(a))=\varphi(f(a)g(a))=\int_{\mathbb{R}}f(s)g(s)\mu(ds)=\int_{\mathbb{R}}g(s)(f\cdot \mu)(ds);$$
		$$\chi(g(b))=\varphi(f(a)g(b))=\varphi(f(a))\varphi(g(b))=\int_{\mathbb{R}}g(s)\nu(ds),$$
		where we have used that $a$ and $b$ are freely independent and $\varphi(f(a))=1$. Now observe that $a$ and $b$ are conditionally freely independent in $(\mathcal{A},\varphi,\chi)$ by a direct application of Proposition 6.3 (with $x=f(a)$) to the von Neumann subalgebras generated by $a$ and $b$, and the conclusion follows.
	\end{proof} 
	
	\begin{lemma}
		Let $\mu, \nu$ be two compactly supported Borel probability
		measures on $\mathbb{R}$, $f:\mathbb{R}\to \mathbb{R}$ be a bounded positive Borel-measurable function such that $\int_{\mathbb{R}}f(x)\mu(dx)=1$ and $P$ be a polynomial in two non-commuting variables. Then $(\mu,f\cdot\mu)\Box_c^P(\nu,\nu)$ is absolutely continuous with respect to $\mu\Box^P\nu$.\end{lemma}
	
	\begin{proof}
		Let $E$ be a Borel subset of $\mathbb{R}$ such that $(\mu\Box^P\nu)(E)=0$. By the preceding Lemma, \begin{eqnarray*}\left((\mu,f\cdot\mu)\Box_c^P(\nu,\nu)\right)(E)
			&=&\chi(\mathbf{1}_E(P(a,b)))\\
			&=&\varphi(f(a)\mathbf{1}_E(P(a,b)))\\
			&\leq &\sqrt{\varphi(f(a)^2)}\sqrt{\varphi(\mathbf{1}_E(P(a,b)))}\\
			&\leq &\Vert f\Vert_{\infty}\sqrt{(\mu\Box^P\nu)(E)}=0. \end{eqnarray*} This proves absolute continuity of $(\mu,f\cdot\mu)\Box_c^P(\nu,\nu)$ with respect to $\mu\Box^P\nu$.
	\end{proof} 
	
	\vspace{10pt}
	
	The main result of this section is the following.
	
	\vspace{6pt}
	
	\begin{theorem}   \label{thm:68}
		Let $a,b$ be freely independent selfadjoint non-commutative random variables with respective 
		distributions $\mu, \nu$ in a ncps $(\mathcal{A},\varphi)$.  Let $f:\mathbb{R}\to \mathbb{R}$ 
		be a bounded non-negative Borel-measurable function such that $\varphi(f(a))=1$, and let $P$ 
		be a selfadjoint polynomial in two non-commuting variables. Then 
		\[
		E(f(a)\mid P(a,b))
		= \frac{d\left((\mu,f\cdot\mu)\Box_c^P(\nu,\nu)\right)}{d\left(\mu\Box^P\nu\right)}(P(a,b)).
		\]
	\end{theorem}
	
	\begin{proof}
		As explained in Remark \ref{freedenoiserasRadonNikodym}, the free denoiser, 
		that is the bounded Borel-measurable function $h:\mathbb{R}\to \mathbb{R}$ 
		such that $$E(f(a)\mid P(a,b))=h(P(a,b))$$ is given by the Radon-Nikodym derivative 
		\[
		h=\frac{ d \mu_{P(a,b)}^{\chi} }{ d \mu_{P(a,b)}^{\varphi} }.
		\]
		Then, according to Lemma \ref{twostatesdistributionofP(a,b)}, $(\mu_{P(a,b)}^{\varphi},\mu_{P(a,b)}^{\chi})=(\mu\Box^P\nu,(\mu,f\cdot\mu)\Box_c^P(\nu,\nu))$, which concludes the proof.
	\end{proof} 
	
	\vspace{10pt}
	
	\begin{example}
		In order to illustrate how Theorem \ref{thm:68} works, let us work out the didactical example 
		when the distributions $\mu, \nu$ of our freely independent elements $a$ and $b$ are 
		just given as $\mu=\nu=\frac{1}{2}\delta_{0}+\frac{1}{2}\delta_{2}$, and we want to compute 
		$E( a \mid a+b )$.  (So, in the setting of Theorem \ref{thm:68}, the function $f$ is a suitable 
		truncation of $f(s) = s$, and the polynomial $P(a,b)$ is $a+b$.)  We know from the beginning, due
		to symmetry considerations, that the free denoiser $h$ will come out, in this case, as
		$h(t) = t/2$; let us, nevertheless, see how c-freeness considerations confirm this result.
		
		We first recall the well-known fact that the distribution $\mu\boxplus \nu$ of $a+b$ is the
		arcsine distribution, which is absolutely continuous with respect to Lebesgue measure with
		density $\frac{1}{\pi\sqrt{t(4-t)}}\mathbf{1}_{]0;4[}(t)$, and has reciprocal Cauchy transform 
		$F_{\mu\boxplus \nu}(z)=\sqrt{z(z-4)}$. 
		
		It is easy to check that the probability distribution $f\cdot \mu$ is the point mass at $2$, with reciprocal Cauchy transform $F_{f\cdot \mu}(\omega)=\omega-2$ and that the subordination function is $\omega(z)=\frac{1}{2}z+\frac{1}{2}\sqrt{z(z-4)}$. Hence, using Corollary 4 from \cite{Belinschi3}, the reciprocal Cauchy transform of the c-free convolution $(\mu,f\cdot \mu)\boxplus_c (\nu,\nu)$ is given by:
		$$F_{(\mu,f\cdot \mu)\boxplus_c (\nu,\nu)}(z)=F_{f\cdot \mu}(\omega(z))=\frac{1}{2}z+\frac{1}{2}\sqrt{z(z-4)}-2.$$ By Stieltjes inversion formula, one gets that the density of $(\mu,f\cdot \mu)\boxplus_c (\nu,\nu)$ is equal to $$-\frac{1}{\pi}\operatorname{Im}G_{(\mu,f\cdot \mu)\boxplus_c (\nu,\nu)}(t)=\frac{\operatorname{Im} F_{(\mu,f\cdot \mu)\boxplus_c (\nu,\nu)}(t)}{\pi |F_{(\mu,f\cdot \mu)\boxplus_c (\nu,\nu)}(t)|^2}=\frac{\sqrt{t}}{2\pi \sqrt{4-t}}\mathbf{1}_{]0;4[}(t).$$ Finally, the free denoiser is $h(t)=\frac{t}{2}$, as expected.
		
		We note here that, due to the fact that the spectrum of $a$ is the set $\{ 0, 2 \}$, we 
		can easily extend the formula discussed above to 
		\[
		E(f(a)\vert a+b)=E
		\left(\frac{f(2)-f(0)}{2}a+f(0)\vert a+b\right)=\frac{f(2)-f(0)}{2}\frac{a+b}{2}+f(0),
		\]
		holding for a general bounded Borel function
		$f:\mathbb{R}\to \mathbb{R}$.
	\end{example}
	
	\vspace{10pt}

	\begin{example}[Compression with a free projection]
		Let us consider $p$ a projection with $0<\varphi(p)<1$ and a free, positive element $a$ with $\varphi(a)=1$. We will calculate $E(a|pap)$, observe that overlap measure technique developed in Section 5 allows for computing $E(p|pap)$ and for $E(a|pap)$ we need to use the technique developed in this section. As we mentioned before the theory of finding distributions of polynomials in c-free variables is not well developed, but in this case we can bypass this difficulty.
		
		Define $\chi(c)=\varphi(a c)$. Of course we have $\mu^\varphi_{pap}=\mu_a\boxtimes\mu_{p}$. On the other hand $\chi(1)=1$ and for $n\geq 1$ we have
		\[\chi\left((pap)^n\right)=\varphi\left((pap)^{n+1}\right).\]
		This gives immediately
		\[G_{\mu^{\chi}_{pap}}(z)=\frac{1-\varphi(p)}{z}+z\left(G_{\mu_a\boxtimes\mu_p}(z)-\frac{1}{z}\right).\]
		Hence we see that $E(a|pap)=h(pap)$ with
		\begin{align*}   
			h(t)=
			\begin{cases}
				\frac{1-\varphi(p)}{\mu_a\boxtimes\mu_p(\{0\})} \qquad &\mbox{for } t=0,\\
				t &\mbox{for } t>0.
			\end{cases}
		\end{align*}
	\end{example}
	
	\vspace{10pt}
	
	% \subsection{Continuity of free denoisers}    \label{subsec:continuity}
	%
	%
	% In this subsection we use the two‐state framework to derive compact formulas 
	% for $E(a\mid a+b)$ and $E(a\mid a^{1/2}ba^{1/2})$.

		\subsection{Some more details on \boldmath{$E(a \mid a+b)$} and 
			\boldmath{$E( a \mid a^{1/2}b a^{1/2})$}. }
		\label{subsec:continuity}
		
		$\ $
		
		\noindent
		In this subsection we use the two‐state framework to derive compact formulas 
		for $E(a\mid a+b)$ and for $E(a\mid a^{1/2}ba^{1/2})$.  
	
	\vspace{6pt}
	
	We begin with the additive case. We note that when one is interested only in 
	$E(a\mid a+b)$, a formula can be derived directly from Biane’s subordination.
	Namely, define $\chi(c)=\varphi(ac)$ then
	
	\begin{align*}
		G^{\chi}_{a+b}(z)&=\varphi\left(a\frac{1}{z-a-b}\right)=\varphi\left(a\frac{1}{\omega(z)-a}\right)=\varphi\left(\frac{a-\omega(z)}{\omega(z)-a}+\omega(z)\frac{1}{\omega(z)-a}\right)\\&=\omega(z)G_{\mu\boxplus\nu}(z)-1.
	\end{align*}
	Since the conditional expectation is given by the Radon-Nikodym derivative with 
	respect to $\mu\boxplus\nu$ which has only an absolutely continuous part and a 
	discrete part, it follows that
	\begin{equation}  \label{eqn:65a}
		h(t)=\lim_{\varepsilon\to 0}
		\frac{\operatorname{Im}\Bigl(G_{\mu\boxplus\nu}(t+i\varepsilon)\,\omega(t+i\varepsilon)\Bigr)}{\operatorname{Im}\Bigl(G_{\mu\boxplus\nu}(t+i\varepsilon)\Bigr)}.
	\end{equation}
	
	Indeed, (\ref{eqn:65a}) clearly holds for $t\in U$, where $U$ is the set on
		which the density of the absolutely continuous part of $\mu \boxplus \nu$ is positive.  
		For the atomic part, recall that if $t$ is an atom of a probability measure 
		$\sigma$ on $\bR$, then
		$\sigma(\{ t \})
		=\lim_{\varepsilon\to 0} i \varepsilon G_{\sigma}( t + i \varepsilon)
		=- \lim_{\varepsilon \to 0} \varepsilon \operatorname{Im} G_\sigma( t + i\varepsilon)$
		(cf.~Lemma \ref{lem:22}).
		Applying this to $\sigma = \mu\boxplus\nu$, we find that
		if $t$ is an atom of $\mu\boxplus\nu$, then
		\[
		h(t)=\lim_{\varepsilon\to 0}\frac{\varepsilon\operatorname{Im}\Bigl(G_{\mu\boxplus\nu}(t+i\varepsilon)\,\omega(t+i\varepsilon)\Bigr)}{\varepsilon\operatorname{Im}\Bigl(G_{\mu\boxplus\nu}(t+i\varepsilon)\Bigr)}=\lim_{\varepsilon\to 0}\frac{\operatorname{Im}\Bigl(G_{\mu\boxplus\nu}(t+i\varepsilon)\,\omega(t+i\varepsilon)\Bigr)}{\operatorname{Im}\Bigl(G_{\mu\boxplus\nu}(t+i\varepsilon)\Bigr)}.
		\]
		
		It is instructive to verify that, at an atom $t$ of $\mu \boxplus \nu$, (\ref{eqn:65a}) 
		yields the same value $h(t)$ as we had in Proposition \ref{prop:13}.  Indeed, let us consider 
		such a $t$, which we write as $t = \alpha+\beta$ with $\mu(\{\alpha\})+\nu(\{\beta\})>1$.  We then have
		\begin{align*}
			h(t)&=\lim_{\varepsilon\to 0}\frac{\operatorname{Im}\Bigl(G_{\mu\boxplus\nu}(t+i\varepsilon)\,\omega(t+i\varepsilon)\Bigr)}{\operatorname{Im}\Bigl(G_{\mu\boxplus\nu}(t+i\varepsilon)\Bigr)}=\alpha+\lim_{\varepsilon\to 0}\frac{\operatorname{Im}\Bigl(G_{\mu\boxplus\nu}(t+i\varepsilon)\,\left(\omega(t+i\varepsilon)-\alpha\right)\Bigr)}{\operatorname{Im}\Bigl(G_{\mu\boxplus\nu}(t+i\varepsilon)\Bigr)}\\
			&=\alpha+\lim_{\varepsilon\to 0}\frac{\operatorname{Im}\Bigl(G_{\mu}(\omega(t+i\varepsilon))\,\left(\omega(t+i\varepsilon)-\alpha\right)\Bigr)}{\operatorname{Im}\Bigl(G_{\mu\boxplus\nu}(t+i\varepsilon)\Bigr)}.
		\end{align*}
		But, similarly to what we discussed in the proof of Proposition \ref{prop:44},
        we have that 
		\[ 
		\lim_{\varepsilon\to 0}
		\operatorname{Im}\Bigl(G_{\mu}(\omega(t+i\varepsilon))\,\left(\omega(t+i\varepsilon)-\alpha\right)\Bigr)=\mu(\{\alpha\}),
		\]
		while ${\operatorname{Im}\Bigl(G_{\mu\boxplus\nu}(t+i\varepsilon)\Bigr)}$
		diverges; this implies, altogether, that $h(t)=\alpha=\omega(t)$.

	\vspace{6pt}
	
	Similar considerations to those shown above apply to the multiplicative case. 
	When we are interested in $E(a|a^{1/2}ba{1/2})$ (with $\chi$ defined as above) we have
	\[1+\psi^{\chi}_{a^{1/2}ba^{1/2}}(z)=\varphi(a(1-za^{1/2}ba^{1/2})^{-1})=\varphi(a(1-\omega(z)a)^{-1})=\frac{\psi^{\varphi}_{a^{1/2}ba^{1/2}}(z)}{\omega(z)}.\]
	Hence,
	\[G^{\chi}_{a^{1/2}ba^{1/2}}(z)=\frac{G_{a^{1/2}ba^{1/2}}^{\varphi}(z)-\frac{1}{z}}{\omega(\frac{1}{z})}.\]
	Thus for $\mu\boxtimes\nu$ a.e. $t$, we define the free denoiser  in the multiplicative case as
	\[
	h(t)=\lim_{\varepsilon\to 0}\frac{\operatorname{Im}\left(\frac{G_{a^{1/2}ba^{1/2}}^{\varphi}(t+i\varepsilon)-\frac{1}{t+i\varepsilon}}{\omega(\frac{1}{t+i\varepsilon})}\right)}{\operatorname{Im}\left(G_{a^{1/2}ba^{1/2}}^{\varphi}(t+i\varepsilon)\right)}.
	\]
	A simple calculation, similar to the additive case shows that this expression 
    agrees with the one given in \eqref{eq:prop56}. One point to explain is that the limit 
    $\lim_{\varepsilon\to 0}\omega\left(\frac{1}{i\varepsilon}\right)$ agrees with the 
    corresponding limit at $-\infty$ found in Lemma \ref{lem:52}; we refer to 
    \cite{JiRegularityMult}, where this was shown. 
    % By \cite{Belinschi} all functions involved admit continuous extensions to $\bR_+$.
    % However continuity of $G_{a^{1/2} b a^{1/2}}$ at $t=0$ is not covered by the 
    % existing results. Therefore we have that $h$ is $\mu\boxplus\nu$-almost everywhere
    % continuous, with the possible exception of $t=0$. We note that in \cite{JiRegularityMult}
    % it is shown that $\Omega(z)=\frac{1}{\omega(1/z)}$ has continuous continuation to the
    % whole real line.

    \vspace{20pt}

\section{Relation to matrix denoising}  \label{Sec:MatrixDenoising}
	
    As mentioned in the Introduction, the explicit formulas of free denoisers 
    noticed in \eqref{eqn:freetweedieboxplus} and \eqref{eqn:16a} of this paper 
    had previously appeared in \cite{BunAllezBouchaudPotters}, \cite{LedoitPeche}, 
    derived via random matrix techniques. We will now see that this is not
    coincidental, as the matrix denoising problem addressed in the aforementioned 
    references relates to free denoising via the phenomenon of asymptotic freeness.  

\vspace{10pt}

\subsection{Weak continuity of overlap measure.}
In this subsection we record, in the next Proposition \ref{continuity} and 
its Corollary \ref{cor:93}, a general
observation about the behaviour of overlap measures under convergence in
moments for the pairs of selfadjoint elements that are considered.

\vspace{6pt}

\begin{proposition}   \label{continuity}
Suppose we are given the following data.

\vspace{6pt}

$\to$ $( \cA, \varphi )$ is a $W^{*}$-probability space as considered in 
Section \ref{section:Z} and $(x,y)$ is a pair 

\hspace{0.5cm}
of elements in $\cA^{\mathrm{sa}}$, with overlap measure $\muovxy$.

\vspace{6pt}

$\to$ For every $N \in \bN$, $( \cA_N, \varphi_N )$ is a Borel-ncps, as considered 
in Section \ref{section:X}, and $(x_N,y_N)$ 

\hspace{0.5cm} is a pair of elements in $\cA_N^{\mathrm{sa}}$,
with overlap measure $\mu_{(x_N, y_N)}^{\mathrm{ov}}$.

\vspace{6pt}

\noindent
Suppose, moreover, that
\begin{equation}   \label{eqn:91a}
\lim_{N \to \infty} \varphi_N \bigl( x_N^p \, y_N^q \bigr)
= \varphi ( x^p \, y^q ),
\ \ \forall \, p,q \in \bN \cup \{ 0 \}.
\end{equation}
Then it follows that
\begin{equation}   \label{eqn:91b}
\lim_{N \to \infty}  \int_{\bR^2} f(s) \, g(t) 
\, d \mu_{(x_N, y_N)}^{\mathrm{ov}} (s,t) 
= \int_{\bR^2} f(s) \, g(t) \, d \muovxy (s,t),
\ \ \forall \, f,g \in C_{\mathrm{pol}} ( \bR ),
\end{equation}
where $C_{\mathrm{pol}} ( \bR ) 
= \{ f \in \Borpol ( \bR ) \mid f \mbox{ is continuous} \}$ 
(the space of continuous functions with polynomial growth from $\bR$ to $\bC$).
\hfill $\square$
\end{proposition}

\vspace{10pt}

\begin{remark}   \label{rem:92}
$1^o$ The proof of Proposition \ref{continuity} is an application of the
well-known ``method of moments'', where we start with Equation (\ref{eqn:91a}) 
written in the form 
\[
\lim_{N \to \infty} \int_{\bR^2} s^p \, t^q
\, d \mu_{(x_N, y_N)}^{\mathrm{ov}} (s,t) 
= \int_{\bR^2} s^p \, t^q \, d \muovxy (s,t),
\ \ \forall \, p,q \in \mathbb{N} \cup \{ 0 \},
\]
and where an intermediate step of the proof is to infer (by using a 
2-dimensional version of Helly's selection theorem) that the measures 
$\mu_{(x_N, y_N)}^{\mathrm{(ov)}}$ converge weakly to $\muovxy$. For the said 
intermediate step it is relevant to note that the measure $\muovxy$, having 
compact support, is uniquely determined by its moments.  The details of this
proof are left to the reader.

\vspace{6pt}

\noindent
$2^o$ In Proposition \ref{continuity} we assumed that $x,y$ live in a 
$W^{*}$-probability space, while on the $x_N, y_N$ we only made the looser 
assumption that they live in a Borel-ncps (which could be, in particular, a
space of $N \times N$ random matrices).  As reviewed in Section \ref{section:Z},
the $W^{*}$-framework allows us to consider the non-commutative conditional 
expectation $E( x \, | \, y )$.  While the analogous non-commutative conditional 
expectations may not be guaranteed for the pairs $(x_N, y_N)$, one can nevertheless
observe the following immediate consequence of Proposition \ref{continuity}, which 
only refers to $E(x \, | \, y )$.
\end{remark}

\vspace{6pt}

\begin{corollary}   \label{cor:93}
In the framework and notation of Proposition \ref{continuity}, one has
\begin{equation}   \label{eqn:93a}
\lim_{N \to \infty}  || \, x_N - g (y_N) \, ||_{2, \varphi_N}
= || \, x - g(y) ||_{2, \varphi} 
\geq || \, x - E(x \, | \, y ) \, ||_{2, \varphi},
\ \ \forall \, g \in C_{\mathrm{pol}} ( \bR ),
\end{equation}
where $|| \cdot ||_{2, \varphi}$ and $|| \cdot ||_{2, \varphi_N}$ 
are the $L^2$-norms associated to the spaces $( \cA , \varphi )$ and 
$( \cA_N, \varphi_N )$, respectively.  
\end{corollary}

\vspace{10pt}

Equation (\ref{eqn:93a}) has the following meaning: suppose that the denoiser
of $x$ with respect to $y$ is given by a continuous function 
$h \in C_{\mathrm{pol}} ( \bR )$ (that is, $h$ is such that 
$E( x \, | \, y ) = h(y)$).  Then the same $h$ must act as an asymptotic 
denoiser for $x_N$ with respect to $y_N$, at least in the weak sense that no
$g \in C_{\mathrm{pol}} ( \bR )$ can systematically outperform $h$
in the large $N$ limit.

\vspace{16pt}

\subsection{Additive matrix denoising.}
A rich source of asymptotically free pairs of sequences of selfadjoint non-commutative 
random variables is provided by random matrices. In this context, matrix denoising is 
a general problem arising in high-dimensional statistics: an unknown signal represented
by a Hermitian $N\times N$ random matrix $A_N$ is corrupted by a noise which is itself 
a Hermitian $N\times N$ random matrix $B_N$ and one would like to 
``recover the signal'' $A_N$ as a function of the noisy observation $A_N + B_N$.
% ``recover the signal'' $A_N$ from its noisy observation (find an estimate of 
% the signal $A_N$ as a function of the noisy observation).
	
In the lineage of \cite{LedoitPeche}, Allez, Bouchaud, Bun and Potters addressed
in \cite{BunAllezBouchaudPotters} a general problem of additive/multiplicative matrix
denoising: given two independent symmetric (respectively semi-positive definite) 
$N\times N$ random matrices $A_N, B_N$, where the distribution of the noise $B_N$ is 
invariant under the action by conjugation of the orthogonal group, the problem is 
to find the optimal (in the $L^2$ sense) Rotationally Invariant Estimator (RIE for 
short) of the signal $A_N$, based on the observation of $C_N=A_N+B_N$ (or 
$C_N=A_N^{1/2}B_NA_N^{1/2}$, respectively).  The optimal RIE is a bounded measurable 
function $\delta_N:\mathbb{R}\to \mathbb{R}$ which minimizes
$\mathbb{E}[N^{-1}\text{Tr}\left((A_N-\delta_N(C_N))^2\right)]$.  
The paper \cite{BunAllezBouchaudPotters} provides arguments supporting the claim that 
the sequence $(\delta_N)_{N\in \mathbb{N}}$ converges to a function 
$\delta:\mathbb{R}\to \mathbb{R}$ which may be expressed by using tools 
from free probability theory. 
	
When particularized to a $N\times N$ GOE-distributed additive noise $B_N$
with variance parameter $\sigma^2N^{-1}$, the sequence $(\delta_N)_{N\in \mathbb{N}}$ 
of optimal RIE of $A_N$ in the general additive matrix denoising problem converges
to the function $\delta:\mathbb{R}\to \mathbb{R}$ defined by: 
\[
\delta(t)=t-2\sigma^2\pi H(t),
\]
where $H$ is the Hilbert transform of the limiting empirical spectral 
distribution of $A_N+B_N$.
	
This result may be interpreted in terms of free denoising. Moreover, using 
this interpretation, we prove the following proposition, stating that no 
$g \in C_{\mathrm{pol}} ( \bR )$ can systematically outperform $\delta$ as
an asymptotic denoiser in the large $N$ limit. 

\vspace{6pt}

\begin{proposition} \label{prop:94}
Let, for each $N\in \mathbb{N}$, $A_N$ be a $N\times N$ deterministic
symmetric matrix and $B_N$ be a $N\times N$ random matrix distributed 
according to GOE with variance $\sigma^2N^{-1}$. Assume that 
\[
N^{-1}\text{Tr}(A_N^k)\underset{N\to \infty}{\longrightarrow}
\int_{\mathbb{R}}x^k\mu(dx),
\]
for some compactly supported Borel probability measure $\mu$ on $\mathbb{R}$.
Let $C_N=A_N+B_N$. Then, for all $g \in C_{\mathrm{pol}} ( \bR )$, 
\[
\lim_{N\to \infty}\mathbb{E}[N^{-1}\text{Tr}\left((A_N-g(C_N))^2\right)]
\geq \lim_{N\to \infty}\mathbb{E}[N^{-1}\text{Tr}\left((A_N-\delta(C_N))^2\right)].
\]
\end{proposition}

\begin{proof}
    We apply Corollary \ref{cor:93} to $x_N=A_N$ and $y_N=C_N$. More precisely, we work in $(\mathcal{A}_N=M_N(L^{\infty -}(\Omega,\mathcal{F},\mathbb{P})),
    \varphi_N=\mathbb{E}N^{-1}\text{Tr})$, which is a Borel-ncps (see Example \ref{example:X9})
    and observe that $(x_N,y_N)$ is a pair of elements in $\mathcal{A}_N^{\mathrm{sa}}$.
    It is well-known that $(x_N,y_N)_{N\in \mathbb{N}}$ are asymptotically free: in particular, 
    condition \eqref{eqn:91a} is satisfied for some pair $(x=a,y=a+b)$ of elements built 
    from freely independent selfadjoint elements $a,b$ in a $W^{*}$-probability space 
    $(\mathcal{A},\varphi)$. Here, by assumption, $a$ is distributed according to $\mu$ and,
    by Wigner's Theorem, $b$ is distributed according to Wigner semicircle law with variance
    $\sigma^2$, which is denoted by $\nu$.  Corollary \ref{cor:93} then implies that, for all
    $g \in C_{\mathrm{pol}} ( \bR )$,
    \begin{align*}
    \lim_{N\to \infty} \mathbb{E}[N^{-1} \text{Tr} \left((A_N-g(C_N))^2\right)]
    & = \lim_{N \to \infty}|| \, x_N - g(y_N) \, ||_{2, \varphi_N}^2  \\
    & = || \, x - g(y) \, ||_{2, \varphi}^2
      \geq || \, x - E(x \, | \, y ) \, ||_{2, \varphi}^2.
    \end{align*}
    It follows from Example \ref{example:47} that 
    $E(x\vert y)=y-2\sigma^2\pi H_{\mu\boxplus \nu}(y)$, which is precisely $\delta(y)$
    for the function $\delta:\mathbb{R}\to \mathbb{R}$ predicted in 
    \cite{BunAllezBouchaudPotters}. Since  the Hilbert transform $H_{\mu\boxplus \nu}$ of
    the free additive convolution $\mu\boxplus \nu$ is continuous (see \cite{BianeConvolution}),
    the function $\delta:\mathbb{R}\to \mathbb{R}$ is itself in $C_{\mathrm{pol}} ( \bR )$ and 
    \begin{align*}
    || \, x - E(x \, | \, y ) \, ||_{2, \varphi}^2 
    &  =|| \, x - \delta(y) \, ||_{2, \varphi}^2
       =\lim_{N \to \infty}|| \, x_N - \delta(y_N) \, ||_{2, \varphi_N}^2  \\
    & =\lim_{N\to \infty}\mathbb{E}[N^{-1}\text{Tr}\left((A_N-\delta(C_N))^2\right)].
    \end{align*}
\end{proof}

Similarly, when the multiplicative noise $B_N$ is Wishart-distributed with 
aspect ratio converging towards $\gamma\in (0, \infty)$, \cite{BunAllezBouchaudPotters} 
recovers the estimator previously derived by Ledoit and P\'ech\'e in \cite{LedoitPeche}
that is reviewed in the following subsection.

\vspace{10pt}
    
\subsection{Covariance matrix estimation}	
Another important instance of the general matrix denoising problem is the so-called 
covariance matrix estimation. Let $X_1,\ldots ,X_p$ be independent identically distributed 
random vectors in $\mathbb{C}^N$ with independent standardized entries, and let $\Sigma_N$
be a positive semi-definite Hermitian $N\times N$ matrix. The goal is to estimate
$\Sigma_N$ from the observation of the independent identically distributed random vectors
$\Sigma_N^{1/2}X_1,\ldots ,\Sigma_N^{1/2}X_p$ with common covariance matrix $\Sigma_N$.
It is well-known that the sample covariance matrix 
\[
Y_N=p^{-1}\sum_{i=1}^p(\Sigma_N^{1/2}X_i)(\Sigma_N^{1/2}X_i)^*
\]
is a consistent estimator of $\Sigma_N$ in the regime where $N$ is fixed and $p$ is large.
But in the regime where $p$ and $N$ are both large with ratio $p/N$ converging to 
$\gamma\in (0,\infty)$, the sample covariance matrix $Y_N$ becomes a poor estimator of 
$\Sigma_N$: it is empirically observed that the eigenvalues of the sample covariance 
matrix $Y_N$ spread out significantly compared to the eigenvalues of the true covariance
matrix $\Sigma_N$. Based on this observation, Ledoit and P\'ech\'e suggested in
\cite{LedoitPeche} to design estimators of $\Sigma_N$ which shrink eigenvalues of $Y_N$ 
without modifying the corresponding eigenvectors, hence are of the form $\delta(Y_N)$, where 
$\delta:\mathbb{R}\to \mathbb{R}$ is a bounded measurable function. At fixed $N$, there is an 
optimal (in the $L^2$ sense) bounded measurable function $\delta_N:\mathbb{R}\to \mathbb{R}$.
	
	Assuming that the $12$th moment of entries of $X_1,\ldots ,X_p$ is bounded by some constant 
    independent of $N$ and $p$, that $\Sigma_N$ is positive definite, with empirical spectral 
    distribution weakly converging to some probability measure $\mu$ with compact support included 
    in $(0, \infty)$, and assuming that $\gamma\neq 1$, Ledoit and P\'ech\'e proved that the
    sequence $(\delta_N)_{N\in \mathbb{N}}$ converges to the function 
    $\delta:\mathbb{R}\to \mathbb{R}$ defined by: 
	
	\begin{equation*}
		\delta(t)=\left\{ \begin{array}{c}
			\frac{t}{|1-\gamma^{-1}+\gamma^{-1}tG(t)|^2},\quad \mbox{if $t>0$}\\
			-\frac{\gamma}{(1-\gamma)\underline{G}(0)}\mathbf{1}_{\gamma<1},\quad \mbox{if $t=0$,}\\
		\end{array}  \right.
	\end{equation*}
	where $G$ is the Stieltjes transform of the limiting empirical spectral distribution 
    of $Y_N$ and $\underline{G}(t)=(1-\gamma^{-1})t^{-1}+\gamma^{-1}G(t)$. 
    % The almost sure weak convergence of the empirical spectral distribution of $Y_N$
    % is attributed to Silverstein, and the continuous extension to the real line of $G$ 
    % to Choi and Silverstein.

This result also has an interpretation in terms of free denoising: one may argue as in the 
proof of Proposition \ref{prop:94} that $(x_N,y_N)=(\Sigma_N,Y_N)$ defines for each 
$N\in \mathbb{N}$ a pair of selfadjoint elements in a Borel-ncps in such a way that 
condition (\ref{eqn:91a}) is satisfied for some pair $(x,y)=(a,a^{1/2}ba^{1/2})$ of elements
built from freely independent positive elements $a,b$ in a $W^*$-probability space. Here,
$a$ is distributed according to $\mu$ and $b$ is distributed according to a dilation of 
the free Poisson distribution with parameter $\gamma$. The function $\delta$ found in 
\cite{LedoitPeche} and reviewed above coincides, up to a dilation factor, to the
free denoiser computed in Example \ref{ex:57}; one may then use Corollary \ref{cor:93} 
to conclude that no $g \in C_{\mathrm{pol}} ( \bR )$ can outperform $\delta$ as an 
asymptotic denoiser in the large $N$ limit.

\vspace{0.5cm}

\noindent
{\bf Acknowledgements.}

% \vspace{6pt}

\noindent
A preliminary version of the present paper was given at a workshop 
on ``Algebraic aspects of random matrices'' at CIRM Luminy in September 2024;
we are grateful to the organizers of the workshop for the opportunity to participate 
and to present this work.  We would also like to thank Franck Gabriel for pointing 
out to us, at that same workshop, the connection to the Tweedie's formula. 

% \vspace{6pt}

\noindent
M.F. acknowledges a useful discussion with Guilhem Semerjian at the 
Institut Henri Poincar\'e in March 2025, providing us with some enlightening 
comments on the work from \cite{BunAllezBouchaudPotters, Semerjian}, and also 
bringing to our attention the book reference \cite{PottersBouchaud}.

	\appendix
	
\section{Multiplicative convolution on the unit circle}
	
Unitary operators do not fit into the framework of this paper, where we consider selfadjoint operators. However we point out that a similar development to the one from Sections 3, 4 and 5 can be done in the case of multiplication of freely independent unitary random variables. For the purpose of this section we fix $u,v$ freely independent unitary random variables in a $W^*$-probability space, with respective distributions $\mu$ and $\nu$. The distribution of $uv$ is the free multiplicative convolution of $\mu$ and $\nu$, also denoted by $\mu\boxtimes \nu$. One can also define the overlap measure $\mu^{(ov)}_{u,uv}$, but in this case this measure is defined on the Borel sigma-algebra of $\mathbb{T}^2$.
	
	In \cite{Biane} Biane proved analogous subordination results in the case of multiplicative convolution on the unit circle. More precisely, for $|z|<1$ we have : \[
	E\left(\frac{z u v}{1-zuv}\Big|u\right)=\frac{\omega(z)u}{1-\omega(z)u}.
	\]
	Similarly as before, for any $s\in\mathbb{T}$, the formula 
	\[\psi_{k_s}(z)=\frac{\omega(z) s}{1-\omega(z)s},\quad |z|<1,
	\]
	defines the moment transform of a probability measure $k_s$ supported on $\mathbb{T}$, that is \[\psi_{k_s}(z)=\int_{\mathbb{T}}\frac{z t}{1-zt}k_s(dt).\]
	One can retrieve the measure $k_s$ as a weak limit via
	\[2 \pi d k_s(t)=\lim_{r\uparrow 1}( 2 Re(\psi_{k_s}(r \overline{t}))+1 )dt,\]
	
	where $dt$ is the Lebesgue measure on $\mathbb{T}$.
	
	For any bounded Borel function $g$ on $\mathbb{T}$ we have
	\[E(g(uv)|u)=\cK g(u)\]
	where kernel $\cK$ is defined as
	\[\cK g(s)=\int_{\mathbb{T}} g(t)k_s(dt).\]
	
	As in the additive and multiplicative case on $\bR_+$, Biane's subordination gives the first disintegration of $\mu^{(ov)}_{u,uv}$, namely we have
	\begin{align*}
		\varphi(f(u) g(uv))&=\varphi(f(u)E(g(uv)|u))=\varphi(f(u)\cK g(u))=\int_{\mathbb{T}} f(s) \cK g(s) d\mu(s)\\
		&=\int_{\mathbb{T}}f(s)\int_{\mathbb{T}}g(t) dk_s(t) d\mu(s).
	\end{align*}
	
	Regularity and atoms of free multiplicative convolution on the unit circle together with boundary behaviour of subordination functions are well understood (see \cite{BelinschiAtomsMult,Belinschi}). Recently in \cite{BelinschiBercoviciHo} the authors showed that the distribution of $uv$ has no singular continuous part. Using similar methods as in Sections 4 and 5, one can show that for $\mu$-almost every $s\in \mathbb{T}$, the measure $k_s$ is absolutely continuous with respect to $\mu\boxtimes \nu$. We have the following.
	\begin{proposition}
		With notation as above suppose that $\mu$ and $\nu$ are not point masses. There exists a subset $U\subset \mathbb{T}$ of full measure with respect to $(\mu\boxtimes\nu)^{ac}$  such that  for $\mu$-almost every $s \in \mathbb{T}$, $k_s(dt)=o(s,t)\mu\boxtimes \nu(dt)$, where $o:\mathrm{supp}(\mu)\times \mathrm{supp}(\mu\boxtimes\nu)\to \mathbb{R}_+$ is 
			defined (in $\mu\times \mu\boxtimes \nu$-almost everywhere sense) by the following formula:
		\begin{align*}
			\forall s\in \mathrm{supp}(\mu), o(s,t)=\begin{cases}
				\frac{1}{2\pi f_{\mu\boxtimes \nu}(t)} Re\left(2\frac{\omega\left(\overline{t}\right) s}{1-\omega\left(\overline{t}\right)s}+1\right)\qquad&\mbox{if }t\in U,\\
				\frac{1}{\mu(s)}1_{\omega(\overline{t})=\overline{s}}&\mbox{if } t \mbox{ is an atom of } \mu\boxtimes\nu.
			\end{cases}
		\end{align*}
	\end{proposition}
	
	\begin{theorem}
			Let $u,v$ be freely independent unitary random variables in a $W^*$-probability space, with respective distributions $\mu$ and $\nu$ not being point masses. Then $d\mu^{(ov)}_{u,uv}(s,t)=o(s,t) \mu(ds)\,\mu\boxtimes\nu(dt)$, where $o$ is defined in the previous proposition. 
	\end{theorem}
	
	\begin{corollary}
		The result above implies that for every $f\in\mathrm{Bor}_b ( \mathbb{T} )$ we have
		\begin{align*}
			\bE[ f(u) \mid uv ] = h(uv),
		\end{align*}
		where $h:\Spec(uv)\to \mathbb{C}$ is defined $\mu\boxtimes\nu$-almost everywhere by
			\begin{align*}
				h(t)=\begin{cases}
					\frac{1}{2 \pi f_{\mu\boxtimes\nu}(t)}\int_{\mathbb{T}}f(s)Re\left(2\frac{\omega\left(\overline{t}\right) s}{1-\omega\left(\overline{t}\right)s}+1\right)d\mu(s)\qquad &\mbox{ if } t\in U,\\
					f\left(\overline{\omega(\overline{t})}\right)&\mbox{ if } t \mbox{ is an atom of } \mu\boxtimes\nu.
				\end{cases}
		\end{align*}
	\end{corollary}
	
\end{document}